\documentclass[reqno,11pt]{amsart}
\usepackage{amsmath}
\usepackage{amsxtra}
\usepackage{amscd}
\usepackage{amsthm}
\usepackage{amsfonts}
\usepackage{amssymb}
\usepackage{scalerel}
\usepackage{verbatim}
\usepackage[matrix,arrow,curve]{xy}
\usepackage{a4wide}
\usepackage[T1]{fontenc}

\allowdisplaybreaks[4]
\usepackage{indentfirst}
\usepackage{stmaryrd}
\usepackage{leftidx}
\usepackage{accents}

\usepackage{xcolor}
\usepackage{color}
\usepackage{dynkin-diagrams}
\usepackage{micro type}
\usepackage{cite}
\usepackage{enumitem}

\usepackage{mathtools,extarrows}
\usepackage{tikz}
\usetikzlibrary{chains}

\tikzset{node distance=2em, ch/.style={circle,draw,on chain,inner sep=2pt},chj/.style={ch,join},every path/.style={shorten >=4pt,shorten <=4pt},line width=1pt,baseline=-1ex}

\let\dlabel=\alabel

\newcommand{\dnode}[2][chj]{%
\node[#1,label={below:\dlabel{#2}}] {};
}

\newcommand{\dnodenj}[1]{%
\dnode[ch]{#1}
}

\newcommand{\dydots}{%
\node[chj,draw=none,inner sep=1pt] {\dots};
}

\newenvironment{proofprop}{\noindent {\em Proof of Proposition} {\rm \ref{lintegralG}}.  }{\hfill $\square$\par}

\usepackage{xr-hyper}
\usepackage[
pdftex,
bookmarks=false,
colorlinks=true,
debug=true,
pdfnewwindow=true]{hyperref}

\newtheorem{Thm}{Theorem}[section]

\newtheorem{Def}[Thm]{Definition}
\newtheorem{Prop}[Thm]{Proposition}
\newtheorem{Lem}[Thm]{Lemma}
\newtheorem{Cor}[Thm]{Corollary}

\theoremstyle{definition}
\newtheorem{Rk}[Thm]{Remark}

\numberwithin{equation}{section}

\def\QC{U_{v}^{>}(L\mathfrak{g})}
\def\SF{\mathbb{S}}

\def\Yangian{Y_{\hbar}^{>}(\mathfrak{g})}
\def\dgdual{\dot{\mathbf{Y}}^{>}_{\hbar}(\mathfrak{g})}
\def\integralb{\mathcal{U}_{v}^{>}(L\mathfrak{o}_{2n+1})}
\def\integralbl{\mathbf{U}_{v}^{>}(L\mathfrak{o}_{2n+1})}
\def\qfb{U_{v}^{>}(L\mathfrak{o}_{2n+1})}
\def\qgb{U_{v}^{>}(\mathfrak{o}_{2n+1})}
\def\qgbi{\mathbf{U}^{>}_{v}(\mathfrak{o}_{2n+1})}
\def\integralgl{\mathbf{U}_{v}^{>}(L\mathfrak{g}_{2})}
\def\qfg{U_{v}^{>}(L\mathfrak{g}_{2})}
\def\qgg{U_{v}^{>}(\mathfrak{g}_{2})}
\def\qggi{\mathbf{U}_{v}^{>}(\mathfrak{g}_{2})}

\newif\ifShowLabels
\ShowLabelstrue
\newdimen\theight
\def\TeXref#1{%
    \leavevmode\vadjust{\setbox0=\hbox{{\tt
        \quad\quad  {\small \rm #1}}}%
    \theight=\ht0
    \advance\theight by \lineskip
    \kern -\theight \vbox to
    \theight{\rightline{\rlap{\box0}}%
    \vss}%
    }}%

\ShowLabelsfalse% comment this out if labels should be printed

\newcommand\nc{\newcommand}
\nc{\unl}{\underline}
\nc{\ol}{\overline}
\nc{\on}{\operatorname}
\nc{\BA}{{\mathbb{A}}}
\nc{\BC}{{\mathbb{C}}}
\nc{\BD}{{\mathbb{D}}}
\nc{\BF}{{\mathbb{F}}}
\nc{\BG}{{\mathbb{G}}}
\nc{\BM}{{\mathbb{M}}}
\nc{\BN}{{\mathbb{N}}}
\nc{\BO}{{\mathbb{O}}}
\nc{\BQ}{{\mathbb{Q}}}
\nc{\BP}{{\mathbb{P}}}
\nc{\BR}{{\mathbb{R}}}
\nc{\BZ}{{\mathbb{Z}}}
\nc{\BS}{{\mathbb{S}}}
\nc{\BK}{{\mathbb{K}}}
\nc{\BW}{W}

\nc{\CA}{{\mathcal{A}}} \nc{\CB}{{\mathcal{B}}} \nc{\CalC}{{\mathcal
C}} \nc{\CalD}{{\mathcal D}} \nc{\CE}{{\mathcal{E}}}
\nc{\CF}{{\mathcal{F}}} \nc{\CG}{{\mathcal{G}}}
\nc{\CH}{{\mathcal{H}}} \nc{\CI}{{\mathcal{I}}}
\nc{\CK}{{\mathcal{K}}} \nc{\CL}{{\mathcal{L}}}
\nc{\CM}{{\mathcal{M}}} \nc{\CN}{{\mathcal{N}}}
\nc{\CO}{{\mathcal{O}}} \nc{\CP}{{\mathcal{P}}}
\nc{\CQ}{{\mathcal{Q}}} \nc{\CR}{{\mathcal{R}}}
\nc{\CS}{{\mathcal{S}}} \nc{\CT}{{\mathcal{T}}}
\nc{\CU}{{\mathcal{U}}} \nc{\CV}{{\mathcal{V}}}
\nc{\CW}{{\mathcal{W}}} \nc{\CX}{{\mathcal{X}}}
\nc{\CY}{{\mathcal{Y}}} \nc{\CZ}{{\mathcal{Z}}}

\nc{\fa}{{\mathfrak{a}}}
\nc{\fb}{{\mathfrak{b}}}
\nc{\fg}{{\mathfrak{g}}}
\nc{\fgl}{{\mathfrak{gl}}}
\nc{\fh}{{\mathfrak{h}}}
\nc{\fj}{{\mathfrak{j}}}
\nc{\fl}{{\mathfrak{l}}}
\nc{\fm}{{\mathfrak{m}}}
\nc{\fn}{{\mathfrak{n}}}
\nc{\fu}{{\mathfrak{u}}}
\nc{\fp}{{\mathfrak{p}}}
\nc{\frr}{{\mathfrak{r}}}
\nc{\fs}{{\mathfrak{s}}}
\nc{\ft}{{\mathfrak{t}}}
\nc{\fw}{{\mathfrak{w}}}
\nc{\fz}{{\mathfrak{z}}}

\nc{\fA}{{\mathfrak{A}}}
\nc{\fB}{{\mathfrak{B}}}
\nc{\fD}{{\mathfrak{D}}}
\nc{\fE}{{\mathfrak{E}}}
\nc{\fF}{{\mathfrak{F}}}
\nc{\fG}{{\mathfrak{G}}}
\nc{\fI}{{\mathfrak{I}}}
\nc{\fJ}{{\mathfrak{J}}}
\nc{\fK}{{\mathfrak{K}}}
\nc{\fL}{{\mathfrak{L}}}
\nc{\fM}{{\mathfrak{M}}}
\nc{\fN}{{\mathfrak{N}}}
\nc{\frP}{{\mathfrak{P}}}
\nc{\fQ}{{\mathfrak Q}}
\nc{\fR}{{\mathfrak R}}
\nc{\fS}{{\mathfrak S}}
\nc{\fT}{{\mathfrak{T}}}
\nc{\fU}{{\mathfrak{U}}}
\nc{\fW}{{\mathfrak{W}}}
\nc{\fY}{{\mathfrak{Y}}}
\nc{\fZ}{{\mathfrak{Z}}}

\nc{\ba}{{\mathbf{a}}}
\nc{\bb}{{\mathbf{b}}}
\nc{\bc}{{\mathbf{c}}}
\nc{\bd}{{\mathbf{d}}}
\nc{\be}{{\mathbf{e}}}
\nc{\bi}{{\mathbf{i}}}
\nc{\bj}{{\mathbf{j}}}
\nc{\bn}{{\mathbf{n}}}
\nc{\bp}{{\mathbf{p}}}
\nc{\bq}{{\mathbf{q}}}
\nc{\bu}{{\mathbf{u}}}
\nc{\bv}{{\mathbf{v}}}
\nc{\bw}{{\mathbf{w}}}
\nc{\bx}{{\mathbf{x}}}
\nc{\by}{{\mathbf{y}}}
\nc{\bz}{{\mathbf{z}}}

\nc{\bA}{{\mathbf{A}}}
\nc{\bB}{{\mathbf{B}}}
\nc{\bC}{{\mathbf{C}}}
\nc{\bD}{{\mathbf{D}}}
\nc{\bE}{{\mathbf{E}}}
\nc{\bI}{{\mathbf{I}}}
\nc{\bK}{{\mathbf{K}}}
\nc{\bH}{{\mathbf{H}}}
\nc{\bM}{{\mathbf{M}}}
\nc{\bN}{{\mathbf{N}}}
\nc{\bO}{{\mathbf{O}}}
\nc{\bQ}{{\mathbf Q}}
\nc{\bS}{{\mathbf{S}}}
\nc{\bT}{{\mathbf{T}}}
\nc{\bV}{{\mathbf{V}}}
\nc{\bW}{{\mathbf{W}}}
\nc{\bX}{{\mathbf{X}}}
\nc{\bP}{{\mathbf{P}}}
\nc{\bY}{{\mathbf{Y}}}
\nc{\bZ}{{\mathbf{Z}}}

\nc{\sA}{{\mathsf{A}}}
\nc{\sB}{{\mathsf{B}}}
\nc{\sC}{{\mathsf{C}}}
\nc{\sD}{{\mathsf{D}}}
\nc{\sF}{{\mathsf{F}}}
\nc{\sK}{{\mathsf{K}}}
\nc{\sM}{{\mathsf{M}}}
\nc{\sO}{{\mathsf{O}}}
\nc{\sQ}{{\mathsf{Q}}}
\nc{\sP}{{\mathsf{P}}}
\nc{\sT}{{\mathsf{T}}}
\nc{\sV}{{\mathsf{V}}}
\nc{\sW}{{\mathsf{W}}}
\nc{\sX}{{\mathsf{X}}}
\nc{\sZ}{{\mathsf{Z}}}
\nc{\sU}{{\mathsf{U}}}
\nc{\sS}{{\mathsf{S}}}
\nc{\sH}{{\mathsf{H}}}

\nc{\sfb}{{\mathsf{b}}}
\nc{\sfc}{{\mathsf{c}}}
\nc{\sd}{{\mathsf{d}}}
\nc{\sg}{{\mathsf{g}}}
\nc{\sk}{{\mathsf{k}}}
\nc{\sfl}{{\mathsf{l}}}
\nc{\sfp}{{\mathsf{p}}}
\nc{\sr}{{\mathsf{r}}}
\nc{\st}{{\mathsf{t}}}
\nc{\sfu}{{\mathsf{u}}}
\nc{\sw}{{\mathsf{w}}}
\nc{\sz}{{\mathsf{z}}}
\nc{\sx}{{\mathsf{x}}}
\nc{\se}{{\mathsf{e}}}
\nc{\sfv}{{\mathsf{v}}}

\nc{\bLambda}{{\boldsymbol{\Lambda}}}
\nc{\vv}{{\boldsymbol{v}}}
\nc{\Fl}{{{\mathcal F}\ell}}
\nc{\Gr}{{\on{Gr}}}
\nc{\CHH}{{\CH\!\!\CH}}
\nc{\lambdavee}{{\lambda^{\!\scriptscriptstyle\vee}}}
\nc{\alphavee}{\alpha^{\!\scriptscriptstyle\vee}}
\nc{\rhovee}{{\rho^{\!\scriptscriptstyle\vee}}}
\newcommand\iso{\,\vphantom{j^{X^2}}\smash{\overset{\sim}{\vphantom{\rule{0pt}{0.20em}}\smash{\longrightarrow}}}\,}
\nc{\oQM}{\vphantom{j^{X^2}}\smash{\overset{\circ}{\vphantom{\vstretch{0.7}{A}}\smash{\QM}}}}
\nc{\oZ}{{}^\dagger\!\vphantom{j^{X^2}}\smash{\overset{\circ}{\vphantom{\vstretch{0.7}{A}}\smash{Z}}}}
\nc{\odZ}{{}^\dagger\!\vphantom{j^{X^2}}\smash{\overset{\circ}{\vphantom{\vstretch{0.7}{A}}\smash{\mathfrak Z}}}^{c',c}}
\nc{\bdZ}{{}^\dagger\!\vphantom{j^{X^2}}\smash{\overset{\bullet}{\vphantom{\vstretch{0.7}{A}}\smash{\mathfrak Z}}}^{c',c}}
\nc{\oS}{\vphantom{j^{X^2}}\smash{\overset{\circ}{\vphantom{\vstretch{0.7}{A}}\smash{S}}}}
\nc{\buM}{\vphantom{j^{X^2}}\smash{\overset{\bullet}{\vphantom{\vstretch{0.7}{A}}\smash{M}}}}
\nc{\dW}{{}^\dagger\ol\CW{}}
\nc{\hW}{{}^\dagger\hat\CW{}}
\nc{\wW}{{}^\dagger\widetilde\CW{}}
\nc{\dZ}{{}^\dagger\!\fZ^{c',c}}
\nc{\dZc}{{}^\dagger\!\fZ^{c,c}}
\nc{\tZ}{{}^\dagger\!\tilde{Z}{}}
\nc{\hZ}{{}^\dagger\!\hat{Z}{}}

\nc{\ssl}{\mathfrak{sl}} \nc{\gl}{\mathfrak{gl}}
\nc{\wt}{\widetilde} \nc{\Sym}{\mathrm{Sym}} \nc{\Res}{\mathrm{Res}}
\nc{\sE}{{\mathsf{E}}} \nc{\bs}{{\mathbf{s}}}
\nc{\trig}{\mathrm{trig}} \nc{\rat}{\mathrm{rat}}
\nc{\sign}{\mathrm{sign}} \nc{\sL}{{\mathsf{L}}}
\nc{\fv}{{\mathfrak{v}}} \nc{\ad}{\mathrm{ad}}
\nc{\spsi}{{\mathsf{\psi}}} \nc{\sh}{{\mathsf{h}}}
\nc{\rtt}{\mathrm{rtt}} \nc{\qdet}{\mathrm{qdet}} \nc{\pt}{{\operatorname{pt}}}
\nc{\M}{\mathrm{M}} \nc{\Ker}{\mathrm{Ker}} \nc{\ssc}{\mathrm{sc}}
\nc{\loc}{\mathrm{loc}} \nc{\fra}{\mathrm{frac}}
\nc{\ddj}{\mathrm{DJ}} \nc{\End}{\mathrm{End}} \nc{\ev}{\mathrm{ev}}
\nc{\GL}{\mathrm{GL}}

\nc{\hzeta}{\hat{\zeta}}

\nc{\sfq}{\mathsf{q}}

\newcommand{\sso}{\mathfrak{o}}
\newcommand{\rtU}{\mathcal{U}}
\newcommand{\Lf}{\mathcal{L}}

\begin{document}

%\linenumbers

\title[]{Shuffle algebras and their integral forms:\\ specialization map approach in types $B_n$ and $G_2$}

\author[]{Yue Hu}
\address[]{Y.H.: Beijing University of Posts and Telecommunications, School of Science, Beijing, China}
\email[]{ldkhtys@gmail.com}
\author[]{Alexander Tsymbaliuk}
\address[]{A.T.: Purdue University, Department of Mathematics, West Lafayette, IN 47907, USA}
\email[]{sashikts@gmail.com}

\date{}

   %%%%%%%%%%%%%%%%%%%%%%%%%%%%%%%%%%%%%%%%%%%%%%%%%%%%%%%%%%%%%%%%%%
   %%%%%%%%%%%%%%%%%%%%%%%%% Abstract %%%%%%%%%%%%%%%%%%%%%%%%%%%%%%%
   %%%%%%%%%%%%%%%%%%%%%%%%%%%%%%%%%%%%%%%%%%%%%%%%%%%%%%%%%%%%%%%%%%

\begin{abstract}
We construct a family of PBWD (Poincar{\'e}–Birkhoff–Witt–Drinfeld) bases for the positive subalgebras of quantum
loop algebras of type $B_n$ and $G_2$, as well as their Lusztig and RTT (for type $B_n$ only) integral forms, in the
new Drinfeld realization. We also establish a shuffle algebra realization of these $\BQ(v)$-algebras (proved earlier
in~\cite{NT21} by completely different tools) and generalize the latter to the above $\BZ[v,v^{-1}]$-forms. The rational
counterparts provide shuffle algebra realizations of positive subalgebras of type $B_n$ and $G_2$ Yangians and their
Drinfeld-Gavarini duals. All of this generalizes the type $A_n$ results of~\cite{Tsy18}.
\end{abstract}

\maketitle

   %%%%%%%%%%%%%%%%%%%%%%%%%%%%%%%%%%%%%%%%%%%%%%%%%%%%%%%%%%%%%%%%%%
   %%%%%%%%%%%%%%%%%%%%%%% Introduction %%%%%%%%%%%%%%%%%%%%%%%%%%%%%
   %%%%%%%%%%%%%%%%%%%%%%%%%%%%%%%%%%%%%%%%%%%%%%%%%%%%%%%%%%%%%%%%%%

\section{Introduction}

   %%%%%%%%%%%%%%%%%%%%%%%%%%%%%%%%%%%%%%%%%%%%%%%%%%%%%%%%%%%%%%%%%%
   %%%%%%%%%%%%%%%%%%%%%%%%%%%%%%%%%%%%%%%%%%%%%%%%%%%%%%%%%%%%%%%%%%
   %%%%%%%%%%%%%%%%%%%%%%%%%%%%%%%%%%%%%%%%%%%%%%%%%%%%%%%%%%%%%%%%%%

\subsection{Summary.}

The quantum loop algebras (aka quantum affine algebras with a trivial central charge) associated to a simple
finite dimensional Lie algebra $\fg$ admit two well-known presentations: the original Drinfeld-Jimbo realization
$U^{\ddj}_v(L\fg)$ and the new Drinfeld realization $U_v(L\fg)$, the latter introduced in~\cite{Dri88}.
The explicit isomorphism $U^{\ddj}_v(L\fg)\simeq U_v(L\fg)$ was actually upgraded in~\cite[Theorem 3]{Dri88}
to the isomorphism of the corresponding quantum affine algebras
\begin{equation}
\label{eq:drinfeld-iso-aff}
  U^{\ddj}_v(\widehat{\fg}) \simeq U_v(\widehat{\fg}).
\end{equation}
Many intrinsic properties of quantum affine algebras have been developed in the Drinfeld-Jimbo realization.
For example, the classical Poincar{\'e}–Birkhoff–Witt theorem for Lie algebras was generalized by Beck in~\cite{b}
to the case of $U^{\ddj}_v(\widehat{\fg})$. More precisely, he constructed the bases of each of the subalgebras
featuring in the triangular decomposition
\begin{equation}
\label{triang DJ}
  U^{\ddj}_v(\widehat{\fg})\simeq
  U^{\ddj,>}_v(\widehat{\fg})\otimes U^{\ddj,0}_v(\widehat{\fg})\otimes U^{\ddj,<}_v(\widehat{\fg}).
\end{equation}
This result is a natural upgrade of Lusztig's PBW theorem for finite quantum groups $U_q(\fg)$.

On the other hand, the new Drinfeld realization $U_v(\widehat{\fg})$ is essential to develop the representation
theory of quantum affine algebras. In this realization, there are infinitely many generators which can be conveniently
encoded by currents $e_i(z),f_i(z),\varphi^\pm_i(z)$. Already in the classical case, this approach played the prominent
role manifestly featuring affine Lie algebras $\widehat{\fg}$ in the conformal field theory. It is thus natural to develop
algebraic aspects of $U_v(\widehat{\fg})$ intrinsic to the loop realization (with the hope to generalize this to generalized
Kac-Moody Lie algebras~$\fg$). Let us note right away that $U_v(\widehat{\fg})$ also has a triangular decomposition
(a vector space isomorphism)
\begin{equation}
\label{triang Dr-loop}
  U_v(\widehat{\fg})\simeq U^>_v(\widehat{\fg})\otimes U^0_v(\widehat{\fg})\otimes U^<_v(\widehat{\fg}).
\end{equation}
However, the isomorphism~(\ref{eq:drinfeld-iso-aff}) \underline{does not} intertwine the triangular
decompositions~\eqref{triang DJ}--\eqref{triang Dr-loop}. To this end, we note that specific PBW-type bases of
$U^>_v(\widehat{\fg}),U^<_v(\widehat{\fg})$ were constructed in~\cite{NT21}.

\medskip

While most often quantum groups are defined by generators and relations, there is an alternative combinatorial
approach sweeping the defining relations under the rug. For finite quantum groups, this manifests in the algebra
embedding (observed independently in \cite{Gre97, Ros02, S}):
\begin{equation}\label{eq:shuffle-finite}
  U^>_v(\fg) \hookrightarrow \mathcal{F}=\bigoplus_{i_1,\ldots,i_k\in I}^{k\in \BN}\ \BQ(v)\cdot [i_1\ldots i_k].
\end{equation}
Here, $I$ denotes the set of simple roots of $\fg$, $\mathcal{F}$ has a basis labeled by finite length words
in $I$ and is endowed with the \emph{quantum shuffle} product. As shown by Lalonde-Ram in~\cite{LR95}, there is
a bijection between the set $\Delta^+$ of positive roots of $\fg$ and so-called \emph{standard Lyndon} words in~$I$:
\begin{equation}
\label{eqn:1-to-1 intro}
  \ell \colon \Delta^+ \,\iso\, \Big\{\text{standard Lyndon words}\Big\}.
\end{equation}
In this case, every order on the alphabet $I$ gives rise to a convex order on $\Delta^+$, and the corresponding
Lusztig's PBW basis of $U^>_v(\fg)$ can be constructed combinatorially via iterated $v$-commutators, due to
Levendorskii-Soibelman convexity property of~\cite{LS91}, see~\cite{Lec04,NT21}.

Using similar ideas, Feigin-Odesskii introduced the elliptic shuffle algebras in~\cite{FO89}--\cite{FO98}.
Their trigonometric counterpart (but in the formal setup with $\BQ(v)$ been replaced by $\BQ[[\hbar]]$) was
further studied by Enriquez in~\cite{Enr00, Enr03}. Explicitly, this manifests in the algebra embedding:
\begin{equation}
\label{eq:shuffle-loop}
  \Psi\colon U^>_v(L\fg) \hookrightarrow S.
\end{equation}
Here, $S$ consists of symmetric rational functions in $\{x_{i,r}\}_{i\in I}^{r\in \BZ}$ subject to so-called
\emph{pole} and \emph{wheel} conditions, and endowed with the shuffle product. Thus,~\eqref{eq:shuffle-loop}
is a \emph{functional version} of \eqref{eq:shuffle-finite}.

The major benefit of \eqref{eq:shuffle-loop} is that it allows to conveniently work with the elements of
$U_v(L\fg)$ given by rather complicated non-commutative polynomials in the original generators. Within
the last decade, this realization has already found major applications in the geometric representation theory
and quantum integrable systems. To make this approach self-contained, it is important to have an explicit
description of the image $\mathrm{Im}(\Psi)$. In fact, Enriquez conjectured:
\begin{equation}\label{eq:shuffle-iso}
  \Psi\colon U^>_v(L\fg) \, \iso\, S.
\end{equation}

To prove~\eqref{eq:shuffle-iso}, it is crucial to ``compare the size'' of $\QC$ and $S$. For types $A_{1}$ and
$\hat{A}_{1}$, this was accomplished in~\cite{Neg14} by crucially utilizing \emph{specialization maps} analogous to
those from~\cite{FS94, FHHSY09}. The same approach was later used to prove~\eqref{eq:shuffle-iso} for types
$A_{n}$ and $\hat{A}_{n}$ in \cite{Neg13}; for two-parameter and super counterparts of type $A_n$ in~\cite{Tsy18};
for type $\hat{\mathfrak{D}}(2,1;\theta)$ in \cite{FH21}.

In the present note, we generalize most of the results from~\cite{Tsy18} to types $G_{2}$ and $B_{n}$, thus
establishing the isomorphism \eqref{eq:shuffle-iso} and constructing families of PBWD-bases of $U^>_v(L\fg)$
in these types. To do so, we introduce the corresponding specialization maps on the shuffle algebra of the
associated type, and establish their key properties similarly to type $A_n$ from~\cite{Tsy18}. Let us note
that these specialization maps arise naturally from the specific convex orders~on~ $\Delta^+$.

\medskip

It should be emphasized right away that Enriquez's conjecture \eqref{eq:shuffle-iso} has been recently proved for all
finite $\fg$ in \cite{NT21}, the joint work of Negu\c{t} and the second author. However, the approach of \emph{loc.cit}.\
is completely different, as it crucially uses a loop version of \eqref{eq:shuffle-finite} instead of the specialization maps.
The present approach has its own benefits as it can also be used to upgrade both results (the shuffle algebra realization
and the PBWD-type bases) to the integral $\BZ[v,v^{-1}]$-forms of $U^>_v(L\fg)$ and the Yangian version $Y^>_\hbar(\fg)$,
generalizing type $A_n$ from~\cite{Tsy18, Tsy19}. We conclude this introduction by noting that similar specialization maps
actually exist for all finite types (which was already known to~\cite{NT21}), though their definition is more involved.
%However, types $A_n\ (n\geq 1),\, B_n\ (n\geq 2),\, G_2$ are the only ones for which the construction is the most elementary.
%The general case of any finite $\fg$ with an arbitrary convex order on $\Delta^+$ shall be addressed in the
%forthcoming work of Negu\c{t} and the second author.

   %%%%%%%%%%%%%%%%%%%%%%%%%%%%%%%%%%%%%%%%%%%%%%%%%%%%%%%%%%%%%%%%%%
   %%%%%%%%%%%%%%%%%%%%%%%%%%%%%%%%%%%%%%%%%%%%%%%%%%%%%%%%%%%%%%%%%%
   %%%%%%%%%%%%%%%%%%%%%%%%%%%%%%%%%%%%%%%%%%%%%%%%%%%%%%%%%%%%%%%%%%

\subsection{Outline of the paper.}

The structure of the present paper is the following:
\begin{itemize}

\item[$\bullet$]
In Section \ref{pre}, we recall the notion of quantum loop algebras $U^>_v(L\fg)$ in the new Drinfeld realization as well
as shuffle algebras $S$, introduce certain families of quantum root vectors (associated to specific convex orders on the
set $\Delta^+$ of positive roots), and state the key results (PBWD bases and shuffle algebra isomorphism) for $U^>_v(L\fg)$
of types $B_n$ and $G_2$. We briefly recall how such results were proved in~\cite{Tsy18} for type $A_n$ using specialization
maps on $S$, and summarize their key properties in Lemmas \ref{shuffleelement}--\ref{span}.

\medskip
\item[$\bullet$]
In Section \ref{tG}, we define the specialization maps for the shuffle algebra $S$ of exceptional type $G_{2}$, establish
the counterparts of Lemmas~\ref{shuffleelement}--\ref{span} in that setup, and use the latter to prove Theorems~\ref{mainthm}
and~\ref{pbwtheorem} for type $G_2$, see Theorem~\ref{shufflePBWDG}. We upgrade both results to the
Lusztig/Chari-Pressley/Grojnowski integral form $\integralgl$ in Theorem~\ref{srig}.

\medskip
\item[$\bullet$]
In Section \ref{type B}, we define the specialization maps for the shuffle algebra $S$ of type $B_n$, establish the
counterparts of Lemmas~\ref{shuffleelement}--\ref{span} in that setup, and use the latter to prove Theorems~\ref{mainthm}
and~\ref{pbwtheorem} for type $B_n$, see Theorem~\ref{shufflePBWDB}. We upgrade both results to the
Lusztig/Chari-Pressley/Grojnowski integral form $\integralbl$ in Theorem~\ref{lusthmb}. Likewise, we upgrade both results
to the RTT integral form $\integralb$ in Theorem~\ref{rttthmb}.

\medskip
\item[$\bullet$]
In Section \ref{yangian}, we generalize the results of Sections~\ref{tG}--\ref{type B} to the rational setup by providing
the shuffle realization and constructing PBWD bases for the positive subalgebras of the Yangians and their Drinfeld-Gavarini
duals in types $G_2$ and $B_n$, see Theorems~\ref{yangshuffle},~\ref{dgdualshuffle}.

\medskip
\item[$\bullet$]
In Appendix~\ref{sec:app}, we use the RTT realization of $U_v(L{\mathfrak{o}}_{2n+1})$ from~\cite{JLM20} to
explain the natural origin and the name of the RTT integral form $\integralb$ from Subsection~\ref{rttb}.

\end{itemize}

   %%%%%%%%%%%%%%%%%%%%%%%%%%%%%%%%%%%%%%%%%%%%%%%%%%%%%%%%%%%%%%%%%%
   %%%%%%%%%%%%%%%%%%%%%%%%%%%%%%%%%%%%%%%%%%%%%%%%%%%%%%%%%%%%%%%%%%
   %%%%%%%%%%%%%%%%%%%%%%%%%%%%%%%%%%%%%%%%%%%%%%%%%%%%%%%%%%%%%%%%%%

\subsection{Acknowledgements}

Both authors are indebted to B.~Feigin for numerous discussions. Y.H.\ is grateful to P.~Shan for stimulating
discussions. A.T.\ is deeply indebted to A.~Negu\c{t} for enlightening discussions over the years. We are also
grateful to the anonymous referees for very useful suggestions.
Y.H.\ gratefully acknowledges support from Yau Mathematical Sciences Center, Tsinghua University, at which part of
the research for this note was performed. A.T.\ is grateful to IHES (Bures-sur-Yvette, France) for the hospitality
and wonderful working conditions in the Spring 2023, when the final version of this note was prepared.
The work of A.T.\ was partially supported by NSF Grants DMS-$2037602$ and~DMS-$2302661$.

   %%%%%%%%%%%%%%%%%%%%%%%%%%%%%%%%%%%%%%%%%%%%%%%%%%%%%%%%%%%%%%%%%%
   %%%%%%%%%%%%%%%%%%%%%%%%%%% Section 1 %%%%%%%%%%%%%%%%%%%%%%%%%%%%
   %%%%%%%%%%%%%%%%%%%%%%%%%%%%%%%%%%%%%%%%%%%%%%%%%%%%%%%%%%%%%%%%%%

\section{Preliminaries}\label{pre}

   %%%%%%%%%%%%%%%%%%%%%%%%%%%%%%%%%%%%%%%%%%%%%%%%%%%%%%%%%%%%%%%%%%
   %%%%%%%%%%%%%%%%%%%%%%%%%%%%%%%%%%%%%%%%%%%%%%%%%%%%%%%%%%%%%%%%%%
   %%%%%%%%%%%%%%%%%%%%%%%%%%%%%%%%%%%%%%%%%%%%%%%%%%%%%%%%%%%%%%%%%%

\subsection{Quantum loop algebras and shuffle algebras.}\label{ssec:qlas}

Let $\mathfrak{g}$ be a finite dimensional simple Lie algebra with simple positive roots $\{\alpha_{i}\}_{i\in I}$.
We denote the set of positive roots by $\Delta^{+}$. Each $\beta\in\Delta^{+}$ can be uniquely expressed as a sum
of simple roots: $\beta=\sum_{i\in I}\nu_{\beta,i}\alpha_{i}$ with $\nu_{\beta,i}\in\BN$ (the set $\BN$ will be assumed to include $0$).
We shall refer to $\nu_{\beta,i}$ as the \emph{coefficient of $\alpha_{i}$ in $\beta$}, and we shall use the following notation:
\begin{equation}
\label{eq:i-in-beta}
  i\in\beta  \Longleftrightarrow \nu_{\beta,i}\neq 0.
\end{equation}

We fix a nondegenerate invariant bilinear form on the Cartan subalgebra $\fh$ of $\fg$. This gives rise to a
nondegenerate form on the dual $\fh^*$, and we set $d_{i}\coloneqq \frac{(\alpha_{i},\alpha_{i})}{2}$. The choice of
the form is such that $d_i=1$ for short roots $\alpha_i$. Let $A=(a_{ij})_{i,j\in I}$ be the Cartan matrix of $\fg$,
so that $d_{i}a_{ij}=(\alpha_{i},\alpha_{j})=d_{j}a_{ji}$. In this paper, we consider simple Lie algebras of types
$A_{n}$, $B_{n}$, $G_{2}$. The corresponding Dynkin diagrams look as follows:
\begin{align}
  A_{n}\text{-type}\ (n\geq 1): &\qquad
  \begin{tikzpicture}[start chain]
    \dnode{1}
    \dnode{2}
    \dydots
    \dnode{n-1}
    \dnode{n}
  \end{tikzpicture}\label{dya} \\
  B_{n}\text{-type}\ (n\geq 2): &\qquad
  \begin{tikzpicture}[start chain]
    \dnode{1}
    \dnode{2}
    \dydots
    \dnode{n-1}
    \dnodenj{n}
    \path (chain-4) -- node[anchor=mid] {\(\Rightarrow\)} (chain-5);
  \end{tikzpicture} \label{dyb}\\
  G_2\text{-type}: &\qquad
  \begin{tikzpicture}[start chain]
    \dnodenj{1}
    \dnodenj{2}
    \path (chain-1) -- node {\(\Rrightarrow\)} (chain-2);
  \end{tikzpicture}\label{dyg}
\end{align}
For these types, we have
\begin{align}
  A_{n}\text{-type}\ (n\geq 1) \colon & \quad d_{i}=1\ (1\leq i\leq n),\\
  B_{n}\text{-type}\ (n\geq 2) \colon & \quad d_{i}=2\ (1\leq i\leq n-1),\  d_{n}=1,\\
  G_2\text{-type} \colon & \quad d_{1}=3,\ d_{2}=1.
\end{align}

Let $v$ be a formal variable. We define $v_{\alpha}=v^{(\alpha,\alpha)/2}$ for any $\alpha\in\Delta^{+}$,
and denote $v_{\alpha_{i}}=v^{d_i}$ simply by $v_{i}$ for any $i\in I$. Let $\fS_{m}$ denote the symmetric group of degree $m$.
Let $U_{v}^{>}(L\mathfrak{g})$ be the \textbf{``positive subalgebra'' of the quantum loop algebra} $U_{v}(L\mathfrak{g})$
associated to $\fg$ in the new Drinfeld realization. Explicitly, $U_{v}^{>}(L\mathfrak{g})$ is the $\BQ(v)$-algebra
generated by $\{e_{i,r}\}_{i\in I}^{r\in\mathbb{Z}}$ subject to the following defining relations:
\begin{equation}
  (z-v_{i}^{a_{ij}}w)e_{i}(z)e_{j}(w)=(v_{i}^{a_{ij}}z-w)e_{j}(w)e_{i}(z) \qquad \forall\ i,j \in I,
\end{equation}
\begin{equation}
  \mathop{Sym}_{z_{1},\dots,z_{1-a_{ij}}}
  \sum_{k=0}^{1-a_{ij}} (-1)^{k}\left[\begin{matrix} 1-a_{ij}\\k\end{matrix}\right]_{v_{i}}
  e_{i}(z_{1})\cdots e_{i}(z_{k})e_{j}(w)e_{i}(z_{k+1})\cdots e_{i}(z_{1-a_{ij}})=0\qquad \forall\ i\neq j.
\label{serreloop}
\end{equation}
Here, we use the following notations:
\begin{equation}\label{eq:basic-def}
\begin{aligned}
  & [\ell]_{u}\coloneqq\frac{u^{\ell}-u^{-\ell}}{u-u^{-1}},\quad  [\ell]_{u}!\coloneqq\prod_{k=1}^{\ell}[k]_{u},\quad
    \left[\begin{matrix} \ell\\m\end{matrix}\right]_{u}\coloneqq\frac{[\ell]_{u}!}{[\ell-m]_{u}![m]_{u}!},\\
  & e_{i}(z)\coloneqq\sum_{r\in\mathbb{Z}} e_{i,r}z^{-r},\quad
    \mathop{Sym}_{z_{1},\dots,z_{m}}V(z_{1},\dots,z_{m})\coloneqq\sum_{\sigma\in\mathfrak{S}_{m}}
    V(z_{\sigma(1)},\dots,z_{\sigma(m)}).
\end{aligned}
\end{equation}
We shall also need the following notation later:
\begin{equation}
  \langle m \rangle_{u}\coloneqq u^m-u^{-m}\qquad \forall\ m\in\BN.
\label{anglev}
\end{equation}

We define $\mathfrak{S}_{\underline{k}}\coloneqq\prod_{i\in I}\mathfrak{S}_{k_{i}}$ for any
$\underline{k}=(k_{1},\dots,k_{|I|})\in\mathbb{N}^{I}$. Associated to the Cartan matrix $A=(a_{ij})_{i,j\in I}$,
we also have the trigonometric version of the Feigin-Odesskii shuffle algebra. To this end, consider
the following $\BN^{I}$-graded $\BQ(v)$-vector space
  \[\SF=\bigoplus_{\underline{k}\in \mathbb{N}^{I}}\SF_{\underline{k}},\]
where $\SF_{\underline{k}}$ consists of rational functions $F$ in the variables
$\{x_{i,r}\}_{i\in I}^{1\leq r\leq k_{i}}$ such that:
\begin{itemize}[leftmargin=0.7cm]

\item
$F$ is $\fS_{\unl{k}}$-symmetric, that is, symmetric in $\{x_{i,r}\}_{r=1}^{k_{i}}$ for each $i\in I$,

\medskip
\item
(\emph{pole conditions}) $F$ has the form
\begin{equation}
  F=\frac{f(\{x_{i,r}\}_{i\in I}^{1\leq r\leq k_{i}})}
         {\prod_{i<j}^{a_{ij}\neq 0}\prod_{1\leq r\leq k_{i}}^{1\leq s\leq k_{j}}(x_{i,r}-x_{j,s})},
\label{polecondition}
\end{equation}
where $f\in \BQ(v)[\{x_{i,r}^{\pm 1}\}_{i\in I}^{1\leq r\leq k_{i}}]^{\mathfrak{S}_{\underline{k}}}$
and an arbitrary order $<$ is chosen on $I$ to make sense of $i<j$
(though the space $\SF_{\underline{k}}$ is clearly independent of this order).

\end{itemize}

Let $(\zeta_{i,j}(z))_{i,j\in I}$ be the matrix of rational functions in $z$ given by
\begin{equation}\label{eq:zeta}
  \zeta_{i,j}(z)=\frac{z-v^{-(\alpha_{i},\alpha_{j})}}{z-1}.
\end{equation}
For $\unl{k},\unl{\ell}\in\BN^{I}$, let
\begin{equation*}
  \unl{k}+\unl{\ell}=(k_{i}+\ell_{i})_{i\in I}\in\BN^{I}.
\end{equation*}
Let us introduce the bilinear {\em shuffle product} $\star$ on $\SF$ as follows:
for $F\in \SF_{\underline{k}}$ and $G\in \SF_{\underline{\ell}}$, we set
\begin{equation}
\begin{aligned}
  & F\star G \big(\{x_{i,r}\}_{i\in I}^{1\leq r\leq k_{i}+\ell_i}\big)=\\
  & \frac{1}{\unl{k}!\cdot \unl{\ell}!}\cdot
    {\mathop{Sym}}_{\mathfrak{S}_{\underline{k}+\underline{\ell}}}
    \bigg(F\big(\{x_{i,r}\}_{i\in I}^{1\leq r\leq k_{i}}\big)\cdot
          G\big(\{x_{j,s}\}_{j\in I}^{k_{j}<s\leq k_{j}+\ell_{j}}\big)
          \prod_{i,j\in I}\prod_{r\leq k_{i}}^{s>k_{j}}\zeta_{i,j}\Big(\frac{x_{i,r}}{x_{j,s}}\Big)\bigg).
\label{shuffleproduct}
\end{aligned}
\end{equation}
Here, for $\unl{k}\in \BN^I$, we set $\unl{k}!=\prod_{i\in I} k_{i}!$, and define the \emph{symmetrization}
\begin{equation}
  {\mathop{Sym}}_{\mathfrak{S}_{\unl{k}}}\big(F(\{x_{i,r}\}_{i\in I}^{1\leq r\leq k_{i}})\big)\, \coloneqq
  \sum_{(\sigma_{1},\dots,\sigma_{|I|})\in \mathfrak{S}_{\unl{k}}}F(\{x_{i,\sigma_{i}(r)}\}_{i\in I}^{1\leq r\leq k_{i}}).
\end{equation}
This endows $\SF$ with a structure of an associative unital algebra. The resulting algebra $(\SF,\star)$ is related
to $\QC$ via the following result (cf.~\cite[Theorem 3]{Enr00}, \cite[Proposition 1.2]{Enr03}):

\begin{Prop}\label{morphism}
The assignment $e_{i,r}\mapsto x_{i,1}^{r}\in \SF_{\mathbf{1}_i} \ (i\in I, r\in\BZ)$, where $\mathbf{1}_i=(0,\ldots,1,\ldots,0)$
with $1$ at the $i$-th coordinate, gives rise to a $\BQ(v)$-algebra homomorphism
\begin{equation}\label{eq:Psi-homom}
  \Psi\colon U_{v}^{>}(L\mathfrak{g}) \longrightarrow \SF.
\end{equation}
Moreover, for any $F\in \text{\rm Im}(\Psi)$, its numerator $f$ from~\eqref{polecondition} satisfies:
\begin{equation}
  f(\{x_{i,r}\}_{i\in I}^{1\leq r\leq k_{i}})=0 \quad \text{once} \quad
  x_{i,s_{1}}=v_{i}^{2}x_{i,s_{2}}=\cdots=v_{i}^{-2a_{ij}}x_{i,s_{1-a_{ij}}}=v_{i}^{-a_{ij}}x_{j,r}
\label{wheelcon}
\end{equation}
for any $i\neq j$ such that $a_{ij}\neq 0$, pairwise distinct $1\leq s_{1},\dots,s_{1-a_{ij}}\leq k_{i}$, and $1\leq r\leq k_{j}$.
\end{Prop}

The vanishing conditions~\eqref{wheelcon} are usually called {\em wheel conditions}. Let $S_{\underline{k}}$
denote the subspace of all elements of $\SF_{\underline{k}}$ satisfying the wheel conditions, and set
$S=\bigoplus_{\underline{k}\in \mathbb{N}^{I}} S_{\underline{k}}$. The following is straightforward:

\begin{Lem}
$S$ is a subalgebra of $\SF$ under the shuffle product $\star$ determined by~\eqref{shuffleproduct}.
\end{Lem}

From now on, we will refer to $(S,\star)$ as the \textbf{(trigonometric Feigin-Odesskii) shuffle algebra}
(of type $\fg$). According to~\cite[Proposition 5.7]{NT21} (cf.~\cite[Corollary 1.4]{Enr03}), we have:

\begin{Prop}\label{inj}
The algebra homomorphism $\Psi$ of~\eqref{eq:Psi-homom} is injective.
\end{Prop}

In fact, $\Psi$ is an algebra isomorphism, cf.~\eqref{eq:shuffle-iso}. This was first conjectured by Enriquez, established
for $A_{n}$-type in \cite{Neg13} (see also \cite{Tsy18}), and finally proved in the full generality in \cite{NT21}.

\begin{Thm}\label{mainthm}
$\Psi\colon \QC \, \iso \, S$ of~\eqref{eq:Psi-homom} is a $\BQ(v)$-algebra isomorphism.
\end{Thm}

The key to the proof of Theorem \ref{mainthm} is to obtain an upper bound estimate for the dimensions of the graded
components of the shuffle algebra $S$. To this end, we note that \cite{NT21} instrumentally used a loop version
of the formal quantum shuffle algebra due to Green, Rosso, and Schauenburg. In contrast, the arguments of~\cite{Neg13}
and~\cite{Tsy18} for type $A_{n}$ were quite different: an estimate of the graded dimensions was achieved by using
certain specialization maps. One benefit of the latter approach is that it allows for the shuffle realization of
various integral $\BZ[v,v^{-1}]$-forms.

The key objective of the present paper is to extend the method used in \cite{Neg13,Tsy18} to types $B_{n}$ and $G_{2}$.
This will provide a new proof of Theorem \ref{mainthm} in these types, different from~\cite{NT21}.

   %%%%%%%%%%%%%%%%%%%%%%%%%%%%%%%%%%%%%%%%%%%%%%%%%%%%%%%%%%%%%%%%%%
   %%%%%%%%%%%%%%%%%%%%%%%%%%%%%%%%%%%%%%%%%%%%%%%%%%%%%%%%%%%%%%%%%%
   %%%%%%%%%%%%%%%%%%%%%%%%%%%%%%%%%%%%%%%%%%%%%%%%%%%%%%%%%%%%%%%%%%

\subsection{Root vectors and PBWD bases}

Our construction of the specialization maps and PBWD bases is based on the specific choice of a convex order
on $\Delta^{+}$. The one that is best suited for our purposes is arising through the lexicographical order on
standard Lyndon words, see~\cite{LR95,Lec04}, as we recall next.

Recall that $I$ is the indexing set of the simple roots of $\mathfrak{g}$. The labeling of the simple roots in
the corresponding Dynkin diagrams \eqref{dya}--\eqref{dyg} provides a total order on the set $I$, and hence the
lexicographical order on the set of words in the alphabet $I$. According to \cite[Proposition~3.2]{LR95}, there is
a natural bijection between the sets of positive roots $\Delta^{+}$ and so-called {\em standard Lyndon words},
cf.~\eqref{eqn:1-to-1 intro}. Thus, the lexicographical order on the latter gives rise to an order $<$ on $\Delta^{+}$,
which is convex by~\cite[Proposition 26]{Lec04}, cf.~\cite[Proposition 2.34]{NT21}. In what follows, we fix this
specific convex order on $\Delta^{+}$ and use standard Lyndon words to parametrize the positive roots.

Let us work this out explicitly for types $A_{n}$, $B_{n}$, $G_{2}$ with the specific order on $I$ as in
\eqref{dya}--\eqref{dyg}. Applying~\cite[Proposition 25]{Lec04}, we find the set of all standard Lyndon words:
\begin{equation}\label{eq:SL-words}
\begin{aligned}
  A_{n}\text{-type}\ (n\geq 1) \colon
  & \quad  \Delta^{+}=\big\{[i,i+1,\dots,j] \, \big| \,  1\leq i\leq j\leq n\big\},\\
  B_{n}\text{-type}\ (n\geq 2) \colon & \quad  \Delta^{+}=\ \big\{[i,i+1,\dots,j] \, \big| \, 1\leq i\leq j\leq n\big\}\\
  & \qquad \qquad \cup \big\{[i,\dots,n,n,n-1,\dots,j] \, \big| \, 1\leq i<j\leq n\big\} ,\\
  G_{2}\text{-type} \colon & \quad  \Delta^{+}=\big\{[1],[1,2],[1,2,1,2,2],[1,2,2],[1,2,2,2],[2]\big\}.
\end{aligned}
\end{equation}
For convenience, we shall use the following notations for positive roots in types $A_{n}$ and $B_{n}$:
\begin{equation}
\begin{aligned}
  & [i,j]\coloneqq [i,i+1,\dots,j] \quad \mathrm{for} \quad 1\leq i\leq j\leq n,\\
  & [i,n,j]\coloneqq [i,\dots,n,n,n-1,\dots,j] \quad \mathrm{for} \quad 1\leq i<j\leq n.
\end{aligned}
\end{equation}
The aforementioned specific convex order on $\Delta^{+}$ in types $A_{n}$, $B_{n}$, $G_{2}$ looks as follows:
\begin{itemize}[leftmargin=0.7cm]

\item
Type $A_{n}\ (n\geq 1)\colon$
\begin{equation}
  [1]<[1,2]<\cdots<[1,n]<[2]<\cdots<[n-1]<[n-1,n]<[n].
\label{lynordera}
\end{equation}

\medskip
\item
Type $B_{n}\ (n\geq 2)\colon$
\begin{equation}
  [1]<[1,2]<\cdots<[1,n]<[1,n,n]<\cdots<[1,n,2]<[2]<\cdots<[n-1,n,n]<[n].
\label{lynorderb}
\end{equation}

\medskip
\item Type $G_{2}\colon$
\begin{equation}
  [1]<[1,2]<[1,2,1,2,2]<[1,2,2]<[1,2,2,2]<[2].
\label{lynorderg}
\end{equation}

\end{itemize}

The \emph{quantum root vectors} $\{E_{\beta,s}\}_{\beta\in\Delta^{+}}^{s\in\BZ}$ of $\QC$ in type $A_{n}$ were
defined in~\cite[(2.12)]{Tsy18} via iterated $v$-commutators (they were called the {\em PBWD basis elements}
and depended on certain extra choices). Here, for $x,y\in\QC$ and $u\in\BQ(v)$, the \emph{$u$-commutator}
$[x,y]_{u}$ is defined via:
\begin{equation}\label{eq:q-commutator}
  [x,y]_{u}\coloneqq xy-u\cdot yx.
\end{equation}
We shall now similarly define the \emph{quantum root vectors} of $\QC$ for $\fg$ of type $B_{n}$ and $G_{2}$:
\begin{itemize}[leftmargin=0.7cm]

\item
$B_{n}$-type.

\noindent
For any $\beta=[i_{1},\dots,i_{\ell}]\in\Delta^{+}$ from~\eqref{eq:SL-words} and $s\in\BZ$, choose a collection
$\lambda_{1},\dots,\lambda_{\ell-1}\in v^{\BZ}$ and a decomposition $s=s_{1}+\cdots+s_{\ell}$ with
$s_{1}, \dots,s_{\ell}\in\BZ$. Then, we define
\begin{equation}
  E_{\beta,s}\coloneqq
  [\cdots[[e_{i_{1},s_{1}},e_{i_{2},s_{2}}]_{\lambda_{1}},e_{i_{3},s_{3}}]_{\lambda_{2}},\cdots,
   e_{i_{\ell},s_{\ell}}]_{\lambda_{\ell-1}}.
\label{rootvector1}
\end{equation}

\medskip
\item
$G_{2}$-type.

\noindent
For $\beta=[i_{1},\dots,i_{\ell}]\neq [1,2,1,2,2]$, $s\in \BZ$, the elements $E_{\beta,s}$ are defined exactly
as in~\eqref{rootvector1}.

\noindent
For $\beta=[1,2,1,2,2]$, $s\in \BZ$, we choose a decomposition $s=s_{1}+\cdots+s_{5}$ with $s_{1},\dots,s_{5}\in\BZ$,
a collection $\lambda_{1},\dots,\lambda_{4}\in v^{\BZ}$, and define
\begin{equation}
  E_{\beta,s}\coloneqq
  [[e_{1,s_{1}},e_{2,s_{2}}]_{\lambda_{1}},
  [[e_{1,s_{3}},e_{2,s_{4}}]_{\lambda_{2}},e_{2,s_{5}}]_{\lambda_{3}}]_{\lambda_{4}}.
\label{rootvector2}
\end{equation}

\end{itemize}

In particular, we have the following specific choices $\{\tilde{E}^{\pm}_{\beta,s}\}_{\beta\in\Delta^{+}}^{s\in \BZ}$
which will be used to construct PBWD bases of the integral forms in Sections \ref{lusg}, \ref{rttb}, \ref{lusb}:

\begin{itemize}[leftmargin=0.7cm]

\item
$B_{n}$-type.

\noindent
If $\beta=[i,j]$, $s\in \BZ$, we choose any decomposition $s=s_{i}+\cdots+s_{j}$, fix a sign $\pm$, and define
\begin{equation}
  \tilde{E}^{\pm}_{[i,j],s}\coloneqq
  [\cdots[[e_{i,s_{i}},e_{i+1,s_{i+1}}]_{v^{\pm 2}},e_{i+2,s_{i+2}}]_{v^{\pm 2}},\cdots,e_{j,s_{j}}]_{v^{\pm 2}}.
\label{rvb1}
\end{equation}
If $\beta=[i,n,j]$, $s\in \BZ$, we choose any decomposition $s=s_{i}+\cdots+s_{j-1}+2s_{j}+\cdots +2s_{n}$,
fix a sign~$\pm$, and define
\begin{equation}
  \tilde{E}^{\pm}_{[i,n,j],s}\coloneqq
  [\cdots[[[\cdots[e_{i,s_{i}},e_{i+1,s_{i+1}}]_{v^{\pm 2}},\cdots, e_{n,s_{n}}]_{v^{\pm 2}},
   e_{n,s_{n}}],e_{n-1,s_{n-1}}]_{v^{\pm 2}},\cdots,e_{j,s_{j}}]_{v^{\pm2}}.
\label{rvb2}
\end{equation}

\medskip
\item
$G_{2}$-type.

\noindent
If $\beta=[1]$ or $[2]$, $s\in \BZ$, we define
\begin{equation}
  \tilde{E}^{\pm}_{[i],s}\coloneqq e_{i,s} \quad \mathrm{for} \quad 1\leq i\leq 2.
\label{rvg1}
\end{equation}
If $\beta=[1,2]$, $s\in \BZ$, we choose any decomposition $s=s_{1}+s_{2}$, and define
\begin{equation}
  \tilde{E}^{\pm}_{[1,2],s}\coloneqq [e_{1,s_{1}},e_{2,s_{2}}]_{v^{\pm 3}}.
\label{rvg2}
\end{equation}
If $\beta=[1,2,2]$, $s\in \BZ$, we choose any decomposition $s=s_{1}+2s_{2}$, and define
\begin{equation}
  \tilde{E}^{\pm}_{[1,2,2],s}\coloneqq [[e_{1,s_{1}},e_{2,s_{2}}]_{v^{\pm 3}},e_{2,s_{2}}]_{v^{\pm 1}}.
\label{rvg3}
\end{equation}
If $\beta=[1,2,2,2]$, $s\in \BZ$, we choose any decomposition $s=s_{1}+3s_{2}$, and define
\begin{equation}
  \tilde{E}^{\pm}_{[1,2,2,2],s}\coloneqq
  [[[e_{1,s_{1}},e_{2,s_{2}}]_{v^{\pm 3}},e_{2,s_{2}}]_{v^{\pm 1}},e_{2,s_{2}}]_{v^{\mp1}}.
\label{rvg4}
\end{equation}
If $\beta=[1,2,1,2,2]$, $s\in \BZ$, we choose any decomposition $s=2s_{1}+3s_{2}$, and define
\begin{equation}
  \tilde{E}^{\pm}_{[1,2,1,2,2],s}\coloneqq
  [[e_{1,s_{1}},e_{2,s_{2}}]_{v^{\pm 3}},[[e_{1,s_{1}},e_{2,s_{2}}]_{v^{\pm 3}},e_{2,s_{2}}]_{v^{\pm 1}}]_{v^{\mp 1}}.
\label{rvg5}
\end{equation}
\end{itemize}

Evoking the specific convex orders $<$ on $\Delta^+$ from~\eqref{lynordera}--\eqref{lynorderg},
let us consider the following order $<$ on the set $\Delta^{+}\times \mathbb{Z}$:
\begin{equation}
  (\alpha,s) < (\beta,t) \quad  \text{iff} \quad \alpha<\beta \ \text{ or }\  \alpha=\beta, s < t.
\label{orderbetas}
\end{equation}
Let $H$ denote the set of all functions $h\colon \Delta^{+}\times\BZ\rightarrow \BN$ with finite support.
The monomials
\begin{equation}
  E_{h}\ :=\prod_{(\beta,s)\in\Delta^{+}\times\mathbb{Z}}\limits^{\rightarrow}E_{\beta,s}^{h(\beta,s)}
  \qquad \forall\ h\in H
\label{PBWDbases}
\end{equation}
will be called the \emph{ordered PBWD monomials} of $\QC$. Here, the arrow $\rightarrow$ over the product sign
refers to the total order \eqref{orderbetas}.

Our first key result generalizes~\cite[Theorem 2.16]{Tsy18} from type $A_n$ to types $G_{2}$ and $B_{n}$
(the proof is presented in Sections \ref{tG} and \ref{type B}, respectively, and is based on the shuffle approach):

\begin{Thm}\label{pbwtheorem}
The ordered PBWD monomials $\{E_{h}\}_{h\in H}$ of~\eqref{PBWDbases} form $\BQ(v)$-bases of~$\QC$
for $\fg$ of type $B_{n}$ and $G_{2}$.
\end{Thm}

   %%%%%%%%%%%%%%%%%%%%%%%%%%%%%%%%%%%%%%%%%%%%%%%%%%%%%%%%%%%%%%%%%%
   %%%%%%%%%%%%%%%%%%%%%%%%%%%%%%%%%%%%%%%%%%%%%%%%%%%%%%%%%%%%%%%%%%
   %%%%%%%%%%%%%%%%%%%%%%%%%%%%%%%%%%%%%%%%%%%%%%%%%%%%%%%%%%%%%%%%%%

\subsection{Specialization maps in type $A_{n}$}

As already mentioned, our key technique used to prove Theorem~\ref{pbwtheorem} as well as Theorem~\ref{mainthm}
in types $B_n$ and $G_2$ is that of specialization maps. Following~\cite{Tsy18}, we shall now briefly recall the
construction and the key properties of the latter in type $A_{n}$, while Sections \ref{tG}--\ref{type B} will
generalize this technique to types $G_{2}$ and $B_{n}$, respectively.

Identifying each simple root $\alpha_{i}\ (i\in I)$ with a basis element $\mathbf{1}_{i}\in\mathbb{N}^{I}$
(having the $i$-th coordinate equal to $1$ and the rest equal to $0$), we can view $\mathbb{N}^{I}$ as the positive cone
of the root lattice of $\fg$. For any $\underline{k}\in\mathbb{N}^{I}$, let $\text{KP}(\underline{k})$ be the set of
{\em Kostant partitions}, which consists of all unordered vector partitions of $\underline{k}$ into a sum of
positive roots. Explicitly, a Kostant partition of $\underline{k}$ is the same as a tuple
$\unl{d}=\{d_{\beta}\}_{\beta\in\Delta^{+}}\in \BN^{\Delta^{+}}$ satisfying
\begin{equation}\label{eq:KP-partition}
  \sum_{i\in I} k_i\alpha_i \, = \sum_{\beta\in\Delta^{+}}d_{\beta}\beta.
\end{equation}
The convex order~\eqref{lynordera} on $\Delta^{+}$ induces a total order on $\text{KP}(\underline{k})$
(opposite to that of~\cite{Tsy18}):
\begin{equation}\label{eq:KP-order}
  \{d_\beta\}_{\beta\in \Delta^+}<\{d'_\beta\}_{\beta\in \Delta^+}\Longleftrightarrow \exists\  \gamma\in \Delta^+\
  \mathrm{s.t.}\ d_\gamma<d'_\gamma\ \mathrm{and}\
  d_\beta=d'_\beta\ \mathrm{for\ all}\ \beta<\gamma.
\end{equation}

For any $h\in H$, we define its \emph{degree} $\text{deg}(h)\in \BN^{\Delta^{+}}$ as the Kostant partition
$\unl{d}=\{d_{\beta}\}_{\beta\in\Delta^{+}}$ with $d_{\beta}=\sum_{s\in \mathbb{Z}} h(\beta,s)\in \BN$ for all
$\beta\in \Delta^{+}$, and the \emph{grading} $\text{gr}(h)\in \BN^I$ so that
$\text{deg}(h)\in\text{KP}(\text{gr}(h))$. For any $\unl{k}\in\BN^{I}$ and $\unl{d}\in\text{KP}(\unl{k})$,
we define the following subsets of $H$:
\begin{equation}\label{hunlkunld}
  H_{\underline{k}}\coloneqq \big\{h\in H \, \big| \,  \text{gr}(h)=\underline{k}\big\}, \qquad
  H_{\underline{k},\underline{d}}\coloneqq \big\{h\in H \, \big| \, \text{deg}(h)=\underline{d}\big\}.
\end{equation}
For any $h\in H_{\unl{k},\unl{d}}$ and $\beta\in\Delta^{+}$, we consider the collection
\begin{equation}\label{eq:lambda-collection}
  \lambda_{h,\beta}=\big\{r_{\beta}(h,1)\leq \dots\leq r_{\beta}(h,d_{\beta})\big\}
\end{equation}
obtained by listing all integers $r\in\mathbb{Z}$ with multiplicity $h(\beta,r)>0$ in the nondecreasing order.
Thus, $E_{h}$ of~\eqref{PBWDbases} can be represented by
\begin{equation}\label{reppbwd}
  E_{h}=
  \prod_{\beta\in\Delta^{+}}\limits^{\rightarrow}\left(E_{\beta,r_{\beta}(h,1)}\cdots E_{\beta,r_{\beta}(h,d_\beta)}\right)
  \qquad \forall\ h\in H_{\unl{k},\unl{d}},
\end{equation}
where the arrow $\rightarrow$ over the product sign refers to the convex order~\eqref{lynordera} on $\Delta^{+}$.

Let us now recall the definition of the specialization maps in type $A_{n}$. For any $F\in S_{\underline{k}}$ and
$\underline{d}\in\text{KP}(\underline{k})$, we split the variables $\{x_{i,l}\}_{i\in I}^{1\leq l\leq k_{i}}$
into the disjoint union of $\sum_{\beta\in\Delta^{+}}d_{\beta}$ groups
\begin{equation}
  \bigsqcup_{\substack{\beta\in\Delta^{+}\\1\leq s\leq d_{\beta}}}
  \Big\{x^{(\beta,s)}_{i,t} \, \Big| \, i\in I, 1\leq t\leq \nu_{\beta,i}\Big\} \,,
\label{splitvariable}
\end{equation}
where the integer $\nu_{\beta,i}$ is the coefficient of $\alpha_i$ in $\beta$ as defined in the beginning of Subsection~\ref{ssec:qlas}.
In type $A_{n}$, any positive root $\beta\in\Delta^{+}$ is of the form $\beta=[i,j]=\sum_{s=i}^{j}\alpha_{s}$ for some
$1\leq i\leq j\leq n$, and so $\nu_{\beta,i}\in\{0,1\}$ for any $1\leq i\leq n$. For $F\in S_{\unl{k}}$, let $f$ denote
its numerator from~\eqref{polecondition}. Then, we define $\phi_{\underline{d}}(F)$ by specializing the variables in $f$
as follows:
\begin{equation}\label{eq:spec-A}
  x^{(\beta,s)}_{i,1}\mapsto v^{-i}w_{\beta,s} \,,\,  \dots \,,\, x^{(\beta,s)}_{j,1}\mapsto v^{-j}w_{\beta,s}
  \qquad \forall\ \beta=[i,j],\ 1\leq s\leq d_{\beta}.
\end{equation}
We note that $\phi_{\unl{d}}(F)$ is symmetric in $\{w_{\beta,s}\}_{s=1}^{d_{\beta}}$ for any $\beta\in\Delta^{+}$,
and is independent of our splitting \eqref{splitvariable} of the variables $\{x_{i,l}\}_{i\in I}^{1\leq l\leq k_{i}}$
into groups. This gives rise to the \textbf{specialization~map}
\begin{equation}
  \phi_{\underline{d}}\colon S_{\underline{k}}\longrightarrow
  \BQ(v)[\{w_{\beta,s}^{\pm 1}\}_{\beta\in\Delta^{+}}^{1\leq s\leq d_{\beta}}]^{\mathfrak{S}_{\unl{d}}}.
\end{equation}
We shall further extend it to the specialization map $\phi_{\unl{d}}$ on the entire shuffle algebra $S$:
\begin{equation}\label{speentireshuffle}
  \phi_{\underline{d}}\colon S\longrightarrow
  \BQ(v)[\{w_{\beta,s}^{\pm 1}\}_{\beta\in\Delta^{+}}^{1\leq s\leq d_{\beta}}]^{\mathfrak{S}_{\unl{d}}}
\end{equation}
by declaring $\phi_{\unl{d}}(F')=0$ for any $\unl{\ell}\neq \unl{k}$ and $F'\in S_{\unl{\ell}}$.

Let us now summarize the key properties of the above specialization maps $\phi_{\unl{d}}$ that were crucially used
in~\cite{Tsy18} to derive the shuffle algebra realization and the PBWD bases in type $A_n$. Our presentation is
adapted to allow for an immediate formulation of $B_n$ and $G_2$ counterparts.

In what follows, we will use the notation $\doteq$ to denote an equality up to $\BQ^{\times}\cdot v^{\BZ}$:
 \begin{equation}\label{eq:trig-const}
  A\doteq B \quad \text{if} \quad  A=c\cdot B \quad \text{for some} \ c\in \BQ^{\times}\cdot v^{\BZ}.
\end{equation}
In type $A_{n}$, we define $\kappa_{\beta}\coloneqq |\beta|-1$ for $\beta\in\Delta^{+}$, where $|\beta|$ denotes
the height of a root $\beta\in \Delta^+$:
\begin{equation}\label{eq:root-height}
  |\beta|:=\sum_{i\in I} \nu_{\beta,i}.
\end{equation}
We also recall the notation $\langle m \rangle_{v}$ from~\eqref{anglev}. The first two lemmas imply the linear
independence of the ordered PBWD monomials $\{E_{h}\}_{h\in H}$ in type $A_{n}$ (see~\cite[\S3.2.2]{Tsy18}):

\begin{Lem}\label{shuffleelement}
For any $h\in H_{\unl{k},\unl{d}}$, we have
\begin{equation}
  \phi_{\unl{d}}(\Psi(E_{h}))\doteq
  {\langle 1\rangle_{v}}^{\sum_{\beta\in\Delta^{+}}d_{\beta}\kappa_{\beta}}\cdot
  \prod_{\beta,\beta'\in \Delta^+}^{\beta<\beta'}G_{\beta,\beta'} \cdot \prod_{\beta\in\Delta^{+}}G_{\beta} \cdot
  \prod_{\beta\in\Delta^{+}}P_{\lambda_{h,\beta}}.
\label{speform}
\end{equation}
Here, the factors $G_{\beta,\beta'}$ and $ G_{\beta}$ are products of linear factors $w_{\beta,s}$ and
$w_{\beta,s}-v^{\BZ}w_{\beta',s'}$ which are \underline{independent of $h\in H_{\unl{k},\unl{d}}$} and are
symmetric with respect to $\mathfrak{S}_{\unl{d}}$, $\lambda_{h,\beta}$ are as in~\eqref{eq:lambda-collection}, and
\begin{equation}
  P_{\lambda_{h,\beta}}={\mathop{Sym}}_{\mathfrak{S}_{d_{\beta}}}
  \left(w_{\beta,1}^{r_{\beta}(h,1)}\cdots w_{\beta,d_{\beta}}^{r_{\beta}(h,d_{\beta})}
       \prod_{1\leq i<j\leq d_{\beta}}\frac{w_{\beta,i}-v_{\beta}^{-2}w_{\beta,j}}{w_{\beta,i}-w_{\beta,j}}\right).
\label{hlp}
\end{equation}
\end{Lem}

This result (cf.~\cite[Lemma 3.17]{Tsy18}) features a ``\textbf{rank $1$ reduction}'': each $P_{\lambda_{h,\beta}}$
from~\eqref{hlp} can be viewed as the shuffle product $x^{r_{\beta}(h,1)}\star\cdots \star x^{r_{\beta}(h,d_\beta)}$
in the shuffle algebra of type $A_1$, evaluated at  $\{w_{\beta,s}\}_{s=1}^{d_\beta}$. The following is~\cite[Lemma 3.16]{Tsy18}
(keeping in mind that our total order~\eqref{eq:KP-order} on $\text{KP}(\underline{k})$ is opposite to that of~\cite{Tsy18}):

\begin{Lem}\label{vanish}
For  any $h\in H_{\unl{k},\unl{d}}$ and $\unl{d}'<\unl{d}$, we have $\phi_{\unl{d}'}(\Psi(E_{h}))=0$.
\end{Lem}

Let $S'_{\unl{k}}$ be the $\BQ(v)$-subspace of $S_{\unl{k}}$ spanned by $\{\Psi(E_{h})\}_{h\in H_{\unl{k}}}$.
The following result (which implies Theorem~\ref{mainthm} and establishes PBWD bases in type $A_n$) is proved
in~\cite[\S3.2.3]{Tsy18}:

\begin{Lem}\label{span}
For any $F\in S_{\unl{k}}$ and $\unl{d}\in \text{\rm KP}(\unl{k})$, if $\phi_{\unl{d}'}(F)=0$ for all
$\unl{d}'\in \text{\rm KP}(\unl{k})$ such that $\unl{d}'<\unl{d}$, then there exists $F_{\unl{d}}\in S'_{\unl{k}}$
such that $\phi_{\unl{d}}(F)=\phi_{\unl{d}}(F_{\unl{d}})$ and $\phi_{\unl{d}'}(F_{\unl{d}})=0$ for all $\unl{d}'<\unl{d}$.
\end{Lem}

Combining Lemmas~\ref{shuffleelement}--\ref{span}, we derive the $A_n$-type counterparts of Theorems \ref{mainthm},~\ref{pbwtheorem} using
\begin{itemize}
\item[(1)]
  the validity of Theorems \ref{mainthm},~\ref{pbwtheorem} in type $A_1$, see~\cite[\S 3.2.1]{Tsy18},

\item[(2)]
  the observation that $\phi_{\unl{d}}(F)=0\ \forall\, \unl{d}\in \text{\rm KP}(\unl{k})$ implies $F=0$
  by taking the maximal element of $\text{\rm KP}(\unl{k})$ with respect to~\eqref{eq:KP-order} corresponding to
  a partition of $\unl{k}$ into a sum of simple roots, see~\cite[Proposition 1.6]{Tsy22}.
\end{itemize}
Let us emphasize that proofs of both theorems \underline{use only} the properties of the specialization maps
$\phi_{\unl{d}}$ stated above. With this in mind, we will now define similar specialization maps and verify
the validity of analogous lemmas in types $G_{2}$ (Section~\ref{tG}) and $B_{n}$ (Section~\ref{type B}). This
will yield a new proof of Theorem \ref{mainthm} and a proof of Theorem~\ref{pbwtheorem} in types $G_{2}$, $B_{n}$.
As an important benefit of this approach, in contrast to \cite{NT21}, we also establish the counterparts of the above
two results for integral $\BZ[v,v^{-1}]$-forms as well as for the Yangian version (treated in Section~\ref{yangian}).

   %%%%%%%%%%%%%%%%%%%%%%%%%%%%%%%%%%%%%%%%%%%%%%%%%%%%%%%%%%%%%%%%%%
   %%%%%%%%%%%%%%%%%%%%%%%%%%%%%%%%%%%%%%%%%%%%%%%%%%%%%%%%%%%%%%%%%%
   %%%%%%%%%%%%%%%%%%%%%%%%%%%%%%%%%%%%%%%%%%%%%%%%%%%%%%%%%%%%%%%%%%

\section{Specialization maps for type $G_{2}$}\label{tG}

In this section, we define specialization maps for the shuffle algebra of type $G_{2}$ and verify their key properties.
This implies the shuffle algebra realization and PBWD-type theorems for $\qfg$ as well as its Lusztig integral form $\integralgl$
introduced in Subsection~\ref{lusg}. Here, $\fg_{2}$ is the simple Lie algebra of type $G_{2}$.

   %%%%%%%%%%%%%%%%%%%%%%%%%%%%%%%%%%%%%%%%%%%%%%%%%%%%%%%%%%%%%%%%%%
   %%%%%%%%%%%%%%%%%%%%%%%%%%%%%%%%%%%%%%%%%%%%%%%%%%%%%%%%%%%%%%%%%%
   %%%%%%%%%%%%%%%%%%%%%%%%%%%%%%%%%%%%%%%%%%%%%%%%%%%%%%%%%%%%%%%%%%

\subsection{$\qfg$ and its shuffle algebra realization}

In type $G_{2}$, for any $F\in  S_{\unl{k}}$ with $\unl{k}=(k_{1},k_{2})\in\mathbb{N}^{2}$, the wheel conditions are:
\begin{equation}\label{eq:wheel-G2}
\begin{aligned}
  F(\{x_{i,r}\}_{1\leq i\leq 2}^{1\leq r\leq k_{i}})=0 \quad \text{once} \quad
  & x_{1,1}=v^{6}x_{1,2}=v^{3}x_{2,1}, \\ \text{or} \quad
  & x_{2,1}=v^{2}x_{2,2}=v^{4}x_{2,3}=v^{6}x_{2,4}=v^{3}x_{1,1}.
\end{aligned}
\end{equation}

For any $\unl{k}\in\mathbb{N}^{2}$ and $\unl{d}\in\text{\rm KP}(\unl{k})$, the specialization map $\phi_{\unl{d}}$
as in~\eqref{speentireshuffle} is defined by the following specialization of the $x^{(*,*)}_{*,*}$-variables
(replacing~\eqref{eq:spec-A} for type $A_n$):
\begin{equation}
  x^{(\beta,s)}_{1,t}\mapsto v^{2t}w_{\beta,s}\quad (1\leq t\leq \nu_{\beta,1}) \,,
  \qquad
  x^{(\beta,s)}_{2,t}\mapsto v^{-3+2t}w_{\beta,s}\quad (1\leq t\leq \nu_{\beta,2})
\label{speG}
\end{equation}
for any positive root $\beta\in \Delta^+$, where $\nu_{\beta,i}$ is as in Subsection~\ref{ssec:qlas}.

\begin{Lem}\label{imageEtildeg}
Consider the particular choices~\eqref{rvg1}--\eqref{rvg5} of quantum root vectors
$\{\tilde{E}^{\pm}_{\beta,s}\}_{\beta\in\Delta^{+}}^{s\in\BZ}$. Their images under $\Psi$
in the shuffle algebra $S$ of type $G_{2}$ are as follows:
\begin{itemize}[leftmargin=0.7cm]

\item
For $\{\tilde{E}^{+}_{\beta,s}\}_{\beta\in\Delta^{+}}^{s\in\BZ}\colon$
\begin{align}
  & \Psi(\tilde{E}^{+}_{[i],s})\doteq x_{i,1}^{s},\   i=1,2,\\
  & \Psi(\tilde{E}^{+}_{[1,2],s})\doteq \frac{\langle 3\rangle_{v}x_{1,1}^{s_{1}+1}x_{2,1}^{s_{2}}}{x_{1,1}-x_{2,1}},\
    \text{with}\ s=s_{1}+s_{2},\\
  & \Psi(\tilde{E}^{+}_{[1,2,2],s})\doteq \frac{\langle 3\rangle_{v}\langle 2\rangle_{v}[2]_{v}\cdot
    x_{1,1}^{s_{1}+2}(x_{2,1}x_{2,2})^{s_{2}}}{(x_{1,1}-x_{2,1})(x_{1,1}-x_{2,2})},\  \text{with}\ s=s_{1}+2s_{2},\\
  & \Psi(\tilde{E}^{+}_{[1,2,2,2],s})\doteq \frac{\langle 3\rangle^{2}_{v}\langle 2\rangle_{v}[2]_{v}\cdot x_{1,1}^{s_{1}+3}
    (x_{2,1}x_{2,2}x_{2,3})^{s_{2}}}{(x_{1,1}-x_{2,1})(x_{1,1}-x_{2,2})(x_{1,1}-x_{2,3})}, \ \text{with}\  s=s_{1}+3s_{2},\\
  & \Psi(\tilde{E}^{+}_{[1,2,1,2,2],s})\doteq \frac{\langle 3\rangle^{3}_{v}\langle 2\rangle_{v}[2]_{v}\cdot
    (x_{1,1}x_{1,2})^{s_{1}+1}(x_{2,1}x_{2,2}x_{2,3})^{s_{2}}\cdot g_{1}}{\prod_{1\leq r\leq 2}^{1\leq t\leq 3}(x_{1,r}-x_{2,t})},
    \ \text{with}\  s=2s_{1}+3s_{2},
\end{align}
where
\begin{equation}
\begin{aligned}
  & g_{1}=(v^{6}+1)x_{1,1}^{2}x_{1,2}^{2}+(v^{6}+1)x_{1,1}x_{1,2}(x_{2,1}x_{2,2}+x_{2,1}x_{2,3}+x_{2,2}x_{2,3})\\
  & \ \ \ \ \ \
    -v^{3}(x_{1,1}+x_{1,2})(x_{1,1}x_{1,2}x_{2,1}+x_{1,1}x_{1,2}x_{2,2}+x_{1,1}x_{1,2}x_{2,3}+x_{2,1}x_{2,2}x_{2,3}).
\end{aligned}
\end{equation}

\medskip

\item
For $\{\tilde{E}^{-}_{\beta,s}\}_{\beta\in\Delta^{+}}^{s\in\BZ}\colon$
\begin{align}
  & \Psi(\tilde{E}^{-}_{[i],s})\doteq x_{i,1}^{s},\  i=1,2,\\
  & \Psi(\tilde{E}^{-}_{[1,2],s})\doteq \frac{\langle 3\rangle_{v}x_{1,1}^{s_{1}}x_{2,1}^{s_{2}+1}}{x_{1,1}-x_{2,1}},\
    \text{with}\ s=s_{1}+s_{2},\\
  & \Psi(\tilde{E}^{-}_{[1,2,2],s})\doteq \frac{\langle 3\rangle_{v}\langle 2\rangle_{v}[2]_{v}\cdot x_{1,1}^{s_{1}}
    (x_{2,1}x_{2,2})^{s_{2}+1}}{(x_{1,1}-x_{2,1})(x_{1,1}-x_{2,2})},\ \text{with}\  s=s_{1}+2s_{2},\\
  & \Psi(\tilde{E}^{-}_{[1,2,2,2],s})\doteq
    \frac{\langle 3\rangle^{2}_{v}\langle 2\rangle_{v}[2]_{v}\cdot x_{1,1}^{s_{1}}
    (x_{2,1}x_{2,2}x_{2,3})^{s_{2}+1}}
         {(x_{1,1}-x_{2,1})(x_{1,1}-x_{2,2})(x_{1,1}-x_{2,3})}, \ \text{with}\  s=s_{1}+3s_{2},\\
  & \Psi(\tilde{E}^{-}_{[1,2,1,2,2],s})\doteq \frac{\langle 3\rangle^{3}_{v}\langle 2\rangle_{v}[2]_{v}\cdot
    (x_{1,1}x_{1,2})^{s_{1}}(x_{2,1}x_{2,2}x_{2,3})^{s_{2}+1}\cdot g_{2}}
    {\prod_{1\leq r\leq 2}^{1\leq t\leq 3}(x_{1,r}-x_{2,t})},\ \text{with}\ s=2s_{1}+3s_{2},
\end{align}
where
\begin{equation}
\begin{aligned}
  & g_{2}=(v^{6}+1)x_{2,1}x_{2,2}x_{2,3}+(v^{6}+1)x_{1,1}x_{1,2}(x_{2,1}+x_{2,2}+x_{2,3})\\
  & \ \ \ \ \ \  -v^{3}(x_{1,1}+x_{1,2})(x_{1,1}x_{1,2}+x_{2,1}x_{2,2}+x_{2,1}x_{2,3}+x_{2,2}x_{2,3}).
\end{aligned}
\end{equation}
\end{itemize}
\end{Lem}

\begin{proof}
Straightforward computation.
\end{proof}

For more general quantum root vectors $\{E_{\beta,s}\}_{\beta\in\Delta^{+}}^{s\in\BZ}$ of $U^>_v(L\fg_2)$
defined by \eqref{rootvector2}, their images under $\Psi$ in $S$ are not so well factorized as for
the particular choices above. Nevertheless, what is actually important is that they behave well under the
specialization maps. For $\beta\in\Delta^{+}$, we shall use $\phi_{\beta}$ to denote the specialization map
$\phi_{\unl{d}}$ with $\unl{d}=\{d_{\alpha}=\delta_{\beta,\alpha}\}_{\alpha\in\Delta^{+}}$.

\begin{Lem}\label{Gs1}
For any $s\in \BZ$ and any choices of $s_k$ and $\lambda_k$ in (\ref{rootvector1},~\ref{rootvector2}), we have:
\begin{align}
  & \phi_{[i]}(\Psi(E_{[i],s}))\doteq w_{[i],1}^{s},\ 1\leq i\leq 2, \label{cbetag1}\\
  & \phi_{[1,2]}(\Psi(E_{[1,2],s}))\doteq \langle 3\rangle_{v}\cdot w_{[1,2],1}^{s+1}, \label{cbetag2}\\
  & \phi_{[1,2,2]}(\Psi(E_{[1,2,2],s}))\doteq \langle 3\rangle_{v}\langle 2\rangle_{v}[2]_{v}\cdot w_{[1,2,2],1}^{s+2},
    \label{cbetag3}\\
  & \phi_{[1,2,2,2]}(\Psi(E_{[1,2,2,2],s}))\doteq
    \langle 3\rangle^{2}_{v}\langle 2\rangle_{v}[2]_{v}\cdot w_{[1,2,2,2],1}^{s+3},
    \label{cbetag4}\\
  & \phi_{[1,2,1,2,2]}(\Psi(E_{[1,2,1,2,2],s}))\doteq
    \langle 4\rangle_{v}\langle 3\rangle^{3}_{v}\langle 2\rangle^{2}_{v}[2]_{v} \cdot w_{[1,2,1,2,2],1}^{s+6}.
\label{cbetag5}
\end{align}
Thus, we have:
\begin{equation}
  \phi_{\beta}(\Psi(E_{\beta,s}))\doteq c_{\beta}\cdot w_{\beta,1}^{s+\kappa_{\beta}}
  \qquad \forall\ (\beta,s)\in\Delta^{+}\times\BZ,
\end{equation}
where the constant $c_{\beta}$ is explicitly specified in \eqref{cbetag1}--\eqref{cbetag5},
and $\kappa_{\beta}$ is given by
\begin{equation}\label{kappaG}
  \kappa_{\beta}=
  \begin{cases}
    |\beta|-1  &\ \text{if}\ \beta\neq [1,2,1,2,2]\\
    |\beta|+1 &\ \text{if}\ \beta=[1,2,1,2,2]
  \end{cases}.
\end{equation}
\end{Lem}

\begin{proof}
This follows from straightforward computations and the fact that for any positive roots $\alpha_{1}<\alpha_{2}$
such that $\alpha_{1}+\alpha_{2}$ is a root, we have
  \[\phi_{\alpha_{1}+\alpha_{2}}(\Psi(E_{\alpha_{1},s_{1}})\star \Psi(E_{\alpha_{2},s_{2}}))=0
    \qquad \forall\  s_{1},s_{2}\in\BZ.\ \]
The latter fact is a special case of Proposition \ref{vanishG}.
\end{proof}

Let us now generalize the above lemma by computing $\phi_{\unl{d}}(\Psi(E_{h}))$ for any $h\in H_{\unl{k},\unl{d}}$.
According to~\eqref{reppbwd}, we have:
\begin{equation*}
  \Psi(E_{h})=\prod^{\rightarrow}_{\beta\in\Delta^{+}}
  \left(\Psi(E_{h,r_{\beta}(h,1)})\star\cdots\star \Psi(E_{h,r_{\beta}(h,d_\beta)})\right).
\end{equation*}
Here, the product refers to the shuffle product and the arrow $\rightarrow$ over the product sign refers to the order
\eqref{lynorderg}. Thus, we can choose a special splitting such that the variables in $\Psi(E_{\beta,r_{\beta}(h,s)})$
are taken to be the group $\{x^{(\beta,s)}_{i,t}\}^{1\leq t\leq \nu_{\beta,i}}_{1\leq i\leq 2}$, and under
$\phi_{\unl{d}}$ they are specialized as in \eqref{speG}.

For each $1\leq i\leq 2$, let us consider an order on all the variables
\begin{equation}
  X_{i}=\big\{x^{(\beta,s)}_{i,t} \, \big| \, \beta\in\Delta^{+},1\leq s\leq d_{\beta}, 1\leq t\leq \nu_{\beta,i}\big\},
\label{variablei}
\end{equation}
defined by (cf.~\eqref{orderbetas})
\begin{equation}\label{ordervariable}
  x^{(\beta,s)}_{i,t}<x^{(\beta',s')}_{i,t'} \quad \text{iff} \quad
  (\beta,s)<(\beta',s')\ \text{ or }\ (\beta,s)=(\beta',s'), t<t'.
\end{equation}
Once we fix the above splitting of the $x_{*,*}$-variables, the group $\fS_{\unl{k}}$ acts on each
$X_{i}\ (1\leq i\leq 2)$ by permuting the tuples
\begin{equation}\label{permutationtuples}
  \big\{(\beta,s,t)\in\Delta^{+}\times\BN\times\BN \, \big| \, 1\leq s\leq d_{\beta}, 1\leq t\leq \nu_{\beta,i}\big\}.
\end{equation}

For any $\unl{d}\in\text{KP}(\unl{k})$, let $\text{Sh}_{\unl{d}}\subset \mathfrak{S}_{\unl{k}}$ be the subset of
``$\unl{d}$-shuffle permutations'', defined via:
\begin{equation}\label{eq:d-perm-G}
  \text{Sh}_{\unl{d}}:=
  \Big\{\sigma\in \mathfrak{S}_{\unl{k}} \, \Big| \, \sigma(x^{(\beta,s)}_{i,t})<\sigma(x^{(\beta,s)}_{i,t'})\quad
    \forall\ t<t', \beta\in\Delta^{+},1\leq s\leq d_{\beta},1\leq i\leq 2\Big\},
\end{equation}
where the order $<$ on the variables is defined in \eqref{ordervariable}.
Then, by definition of the shuffle product~\eqref{shuffleproduct}, we have:
\begin{equation}\label{eq:part-symm}
  \Psi(E_{h})\doteq
  \sum_{\sigma\in\text{Sh}_{\unl{d}}} \sigma\big(F_{h}(\{x^{(*,*)}_{*,*}\})\big)=
  \sum_{\sigma\in\text{Sh}_{\unl{d}}} F_{h}\big(\{\sigma(x^{(*,*)}_{*,*})\}\big),
\end{equation}
where
\begin{equation}\label{eq:part-symm-2}
    F_{h}\coloneqq
    \prod_{\substack{\beta\in\Delta^{+}\\1\leq s\leq d_{\beta}}}\Psi(E_{\beta,r_{\beta}(h,s)})
    \prod^{(\beta,p)<(\beta',q)}_{\substack{\beta,\beta'\in\Delta^{+}\\ 1\leq p\leq d_{\beta},1\leq q\leq d_{\beta'}}}
    \prod_{1\leq i,j\leq 2}\prod_{1\leq \ell\leq \nu_{\beta,i}}^{1\leq r\leq \nu_{\beta',j}}
      \frac{x^{(\beta,p)}_{i,\ell}-v_{i}^{-a_{ij}}x^{(\beta',q)}_{j,r}}{x^{(\beta,p)}_{i,\ell}-x^{(\beta',q)}_{j,r}}.
\end{equation}

%%%%%%%%%%%%%%%%%%%%%%%%%%%%%%%%%%%%%%%%%%%%%%%%%%%%%%%%%%%%%%%%%%%%%%%%%%%%%%%%%%%%%%%%%%%%%%%%%%%%%%
%%%%%%%%%%%%%%%%%%%%%%%%%%%% Update of Yue's original text over the field %%%%%%%%%%%%%%%%%%%%%%%%%%%%
%%%%%%%%%%%%%%%%%%%%%%%%%%%%%%%%%%%%%%%%%%%%%%%%%%%%%%%%%%%%%%%%%%%%%%%%%%%%%%%%%%%%%%%%%%%%%%%%%%%%%%

Let us now evaluate the $\phi_{\unl{d}}$-specialization of each term $\sigma(F_{h})$ in the symmetrization~\eqref{eq:part-symm},
many of which actually vanish. To this end, consider the elements $\sigma\in \fS_{\unl{k}}$ which satisfy the following condition:
as a permutation of the tuples~\eqref{permutationtuples} $\sigma$ fixes the indices $\beta$, $t$ (as well as $i$) and only
permutes the index $s$. In other words, for each $\beta\in \Delta^+$ there is $\sigma_\beta\in \fS_{d_{\beta}}$ such that:
\begin{equation}
   \sigma(x^{(\beta,s)}_{i,t})=x^{(\beta,\sigma_{\beta}(s))}_{i,t}\qquad
   \forall\ \beta\in \Delta^+,\ 1\leq s\leq d_\beta,\ i\in \beta,\ 1\leq t\leq \nu_{\beta,i}.
\label{fsunld}
\end{equation}
Such permutations $\sigma$ form a subgroup of $\fS_{\unl{k}}$ isomorphic to
$\fS_{\unl{d}}\coloneqq\prod_{\beta\in\Delta^{+}}\fS_{d_{\beta}}$, and we shall denote this subgroup simply by $\fS_{\unl{d}}$.

In what follows, given a collection $\mathsf{A}$ of the variables $x^{(*,*)}_{i,*}$ with a fixed index $i$ and a collection
$\mathsf{B}$ of the variables $x^{(*,*)}_{j,*}$ with a fixed index $j$, we shall use the following notation:
\begin{equation}\label{eq:zeta-sets}
  \zeta_{i,j}(\mathsf{A}/\mathsf{B}):=\prod_{x\in \mathsf{A}}^{y\in \mathsf{B}}\zeta_{i,j}(x/y).
\end{equation}
Then, we have:

\begin{Lem}\label{Gs2}
$\phi_{\unl{d}}(\sigma(F_{h}))=0$ for $\sigma\notin\mathfrak{S}_{\unl{d}}$.
 \end{Lem}

\begin{proof}
Define the following sets of variables (cf.\ the notation~\eqref{eq:i-in-beta}):
\begin{gather}
  Z^{\beta}_{i}=\big\{x^{(\beta,s)}_{i,t} \, \big| \, 1\leq s\leq d_{\beta}, 1\leq t\leq \nu_{\beta,i}\big\}
    \qquad \forall\ \beta\in\Delta^{+}, i\in \beta, \\
  Z^{>\beta}_{i}=\big\{x^{(\alpha,s)}_{i,t} \, \big| \, \beta<\alpha, 1\leq s\leq d_{\alpha},
    i\in\alpha, 1\leq t\leq \nu_{\alpha,i}\} \qquad \forall\ \beta\in\Delta^{+}, 1\leq i\leq 2.
\end{gather}
It suffices to show that $\phi_{\unl{d}}(\sigma(F_{h}))\neq 0$ only if \eqref{fsunld} holds for every $\beta\in\Delta^{+}$.
We shall prove this by induction on $\beta$ with respect to the increasing order~\eqref{lynorderg}.

We note that $F_{h}$ contains $\zeta$-factors $\zeta_{1,2}(Z^{[1]}_{1}/Z^{>[1]}_{2})$, cf.~\eqref{eq:zeta-sets}, and clearly
$\sigma(Z^{>[1]}_{2})=Z^{>[1]}_{2}$. If $\sigma(Z^{[1]}_{1})\neq Z^{[1]}_{1}$, then there is some $x^{([1],s)}_{1,1}$
such that $\sigma(x^{([1],s)}_{1,1})=x^{(\gamma,r)}_{1,t}$ for some $\gamma>[1]$. In the latter case, $\sigma(F_{h})$
contains the $\zeta$-factor $\zeta_{1,2}(x^{(\gamma,r)}_{1,t}/x^{(\gamma,r)}_{2,t})$ and so $\phi_{\unl{d}}(\sigma(F_{h}))=0$,
which is a contradiction. This establishes \eqref{fsunld} for $\beta=[1]$, which is the base of induction.

Let us now prove \eqref{fsunld} for $\beta$, assuming \eqref{fsunld} holds for all $\alpha<\beta$.
The only nontrivial check is for the case $\beta=[1,2,1,2,2]$. Assuming by induction that for all $s<s_{0}$
there is $s'$ with $1\leq s'\leq d_{\beta}$ such that $\sigma(x^{(\beta,s)}_{i,t})=x^{(\beta,s')}_{i,t}$ for any
$i\in \beta$, $1\leq t\leq \nu_{\beta,i}$, we shall prove the same also holds for $s =s_{0}$.
The base of induction is $s_{0}=1$, in which case the statement is vacuous.

Let $\sigma(x^{(\beta,s_{0})}_{1,1})=x^{(\gamma_{1},r_{1})}_{1,t_{1}}$. By the induction hypothesis, we have
$(\beta,s_{0})\leq (\gamma_{1},r_{1})$. Suppose $\sigma(x^{(\beta_{1},s_{1})}_{2,\ell_{1}})=x^{(\gamma_{1},r_{1})}_{2,t_{1}}$.
If $\beta_{1}<\beta$, then by the induction hypothesis we get $\gamma_{1}=\beta_{1}<\beta$, a contradiction.
On the other hand, if $\beta_{1}> \beta$, then we note that $F_{h}$ contains the $\zeta$-factors
$\zeta_{1,2}(x^{(\beta,s_{0})}_{1,1}/Z^{>\beta}_{2})$, and thus $\sigma(F_{h})$ contains the factor
$\zeta_{1,2}(x^{(\gamma_{1},r_{1})}_{1,t_{1}}/x^{(\gamma_{1},r_{1})}_{2,t_{1}})$, which specializes to zero under
$\phi_{\unl{d}}$, a contradiction. Thus $\beta_{1}=\beta$. Similarly, there is some
$x^{(\beta,s_{2})}_{2,\ell_{2}}$ such that $\sigma(x^{(\beta,s_{2})}_{2,\ell_{2}})=x^{(\gamma_{1},r_{1})}_{2,t_{1}+1}$.
Moreover, if $s_{0}<s_{1}$ or $s_{1}<s_{2}$, then $\sigma(F_{h})$ contains the $\zeta$-factor
$\zeta_{1,2}(x^{(\gamma_{1},r_{1})}_{1,t_{1}}/x^{(\gamma_{1},r_{1})}_{2,t_{1}})$ or
$\zeta_{2,2}(x^{(\gamma_{1},r_{1})}_{2,t_{1}}/x^{(\gamma_{1},r_{1})}_{2,t_{1}+1})$, which implies
$\phi_{\unl{d}}(\sigma(F_{h}))=0$, hence, a contradiction. Thus we have $s_{0}\geq s_{1}\geq s_{2}$.
If $s_{1}<s_{0}$, then by induction hypothesis, we get a contradiction. Similarly, if $s_{2}<s_{0}$,
we also get a contradiction. Thus $s_{0}=s_{1}=s_{2}$. If $\ell_{2}<\ell_{1}$, then
$x^{(\beta,s_{0})}_{2,\ell_{2}}<x^{(\beta,s_{0})}_{2,\ell_{1}}$, and the condition $\sigma\in\text{Sh}_{\unl{d}}$ implies
  \[x^{(\gamma_{1},r_{1})}_{2,t_{1}+1}=\sigma(x^{(\beta,s_{0})}_{2,\ell_{2}})
    < \sigma(x^{(\beta,s_{0})}_{2,\ell_{1}})=x^{(\gamma_{1},r_{1})}_{2,t_{1}},\]
    a contradiction. Thus $\ell_{1}<\ell_{2}$.
Combining all the above, we get $s_{1}=s_{2}=s_{0}$ and $\ell_{1}<\ell_{2}$.

Similarly, let $\sigma(x^{(\beta,s_{0})}_{1,2})=x^{(\gamma_{2},r_{2})}_{1,t_{2}}$. Then, there are
$x^{(\beta,s_{0})}_{2,\ell'_{1}}$, $x^{(\beta,s_{0})}_{2,\ell'_{2}}$ with $\ell'_{1}<\ell'_{2}$ and
  \[\sigma(x^{(\beta,s_{0})}_{2,\ell'_{1}})=x^{(\gamma_{2},r_{2})}_{2,t_{2}},\quad
    \sigma(x^{(\beta,s_{0})}_{2,\ell'_{2}})=x^{(\gamma_{2},r_{2})}_{2,t_{2}+1}.\]
Note that $1\leq \ell_{1}<\ell_{2}\leq 3$ and $1\leq \ell'_{1}<\ell'_{2}\leq 3$.
Then, we have three cases:
\begin{itemize}[leftmargin=*]

\item
if $\ell_{1}=\ell'_{1}$, then $(\gamma_{1},r_{1})=(\gamma_{2},r_{2})$ and $t_{1}=t_{2}$, a contradiction.

\item
if $\ell_{1}>\ell'_{1}$, then $\ell'_{1}=1, \ell_{1}=2, \ell_{2}=3$, $\ell'_{2}=\ell_{1}$ or $\ell_{2}$. In either case,
we get $(\gamma_{1},r_{1})=(\gamma_{2},r_{2})$. However, since $x^{(\beta,s_{0})}_{1,1}<x^{(\beta,s_{0})}_{1,2}$ we have
$t_{1}<t_{2}$; since  $x^{(\beta,s_{0})}_{2,1}<x^{(\beta,s_{0})}_{2,2}$, we have $t_{2}<t_{1}$, which is a contradiction.

\item
if $\ell_{1}<\ell_{2}$, then $\ell_{1}=1, \ell'_{1}=2, \ell'_{2}=3$. Similarly to above, we must have
$(\gamma_{1},r_{1})=(\gamma_{2},r_{2})$, $t_{1}=1$, $t_{2}=2$. Then $\gamma_{1}=\gamma_{2}=\beta$, which is precisely
what we wanted.

\end{itemize}
This completes our proof of the induction step.
\end{proof}

Combining Lemmas \ref{Gs1}--\ref{Gs2}, we obtain the following analogue of Lemma \ref{shuffleelement} for type $G_{2}$:

\begin{Prop}\label{spekpG}
For any $h\in H_{\unl{k},\unl{d}}$, we have
\begin{equation}\label{eq:Gfactors-G2}
  \phi_{\unl{d}}(\Psi(E_{h}))\doteq
  \prod_{\beta,\beta'\in \Delta^+}^{\beta<\beta'}G_{\beta,\beta'} \cdot
  \prod_{\beta\in\Delta^{+}}\big(c_{\beta}^{d_{\beta}}\cdot G_{\beta}\big) \cdot
  \prod_{\beta\in\Delta^{+}}P_{\lambda_{h,\beta}},
\end{equation}
where the factors $\{P_{\lambda_{h,\beta}}\}_{\beta\in\Delta^{+}}$ are given by \eqref{hlp}, the constants
$\{c_{\beta}\}_{\beta\in\Delta^{+}}$ are as in Lemma~{\rm \ref{Gs1}}, and the factors $G_{\beta}$, $G_{\beta,\beta'}$
are explicitly given in the proof below.
\end{Prop}

\begin{proof}
If $\sigma\in\mathfrak{S}_{\unl{d}}$, then in $F_{h}$ the following factor is invariant under $\sigma$:
\begin{equation}
  \prod_{\beta<\beta'}^{\beta,\beta'\in \Delta^+}
  \prod_{1\leq p\leq d_{\beta}}^{1\leq q\leq d_{\beta'}}
  \prod_{1\leq i,j\leq 2}\prod_{1\leq l\leq \nu_{\beta,i}}^{1\leq r\leq \nu_{\beta',j}}
  \frac{x^{(\beta,p)}_{i,l}-v_{i}^{-a_{ij}}x^{(\beta',q)}_{j,r}}{x^{(\beta,p)}_{i,l}-x^{(\beta',q)}_{j,r}},
\end{equation}
and we denote its $\phi_{\unl{d}}$-specialization by $\prod_{\beta<\beta'} G_{\beta,\beta'}$
(with matching indices $\beta,\beta'$). Thus, we only need to prove the following equality for any $\beta\in\Delta^{+}$:
\begin{equation}
  \phi_{\unl{d}}\left(\Psi\Big(E_{\beta,r_{\beta}(h,1)}\star\cdots\star E_{\beta,r_{\beta}(h,d_\beta)}\Big)\right)\doteq
  c_{\beta}^{d_{\beta}}\cdot G_{\beta} \cdot P_{\lambda_{h,\beta}},
\end{equation}
with the factor $G_{\beta}$ being independent of $\sigma\in\mathfrak{S}_{\unl{d}}$.
This is proved by straightforward computation. Explicitly, we have the following formulas:
\vspace{0.2cm}

\noindent
$G_{\beta} \doteq $

\begin{itemize}[leftmargin=0.7cm]

\item
$1$, for $\beta=[1],[2]$,

\item
  $\prod_{1\leq s\neq r\leq d_{\beta}}(w_{\beta,s}-v^{6}w_{\beta,r})\cdot \prod_{\ell=1}^{d_{\beta}} w_{\beta,\ell}$,
for $\beta=[1,2]$,

\item
  $\prod_{1\leq s\neq r\leq d_{\beta}} \big\{(w_{\beta,s}-v^{6}w_{\beta,r})(w_{\beta,s}-v^{4}w_{\beta,r})\big\}
   \cdot \prod_{\ell=1}^{d_{\beta}} w_{\beta,\ell}^{2}$,
for $\beta=[1,2,2]$,

\item
  $\prod_{1\leq s\neq r\leq d_{\beta}} \big\{(w_{\beta,s}-v^{6}w_{\beta,r})(w_{\beta,s}-v^{4}w_{\beta,r})
   (w_{\beta,s}-v^{2}w_{\beta,r})\big\} \cdot \prod_{\ell=1}^{d_{\beta}} w_{\beta,\ell}^{3}$,
for $\beta=[1,2,2,2]$,

\item
  $\prod_{1\leq s\neq r\leq d_{\beta}} \big\{ (w_{\beta,s}-v^{8}w_{\beta,r})(w_{\beta,s}-v^{6}w_{\beta,r})^{2}
   (w_{\beta,s}-v^{4}w_{\beta,r})^{2}(w_{\beta,s}-v^{2}w_{\beta,r})\big\} \cdot \prod_{\ell=1}^{d_{\beta}} w_{\beta,\ell}^{6}$,
for $\beta=[1,2,1,2,2]$.

\end{itemize}

\noindent
$G_{\beta,\beta'} \doteq$
\begin{itemize}[leftmargin=0.7cm]

\item
  $\prod^{1\leq r\leq d_{\beta'}}_{1\leq s\leq d_{\beta}}(w_{\beta,s}-v^{-6}w_{\beta',r})$,
for $\beta=[1]$ and $\beta'=[1,2]$,

\item
  $\prod^{1\leq r\leq d_{\beta'}}_{1\leq s\leq d_{\beta}}
   \big\{(w_{\beta,s}-v^{-6}w_{\beta',r})(w_{\beta,s}-v^{-4}w_{\beta',r})(w_{\beta,s}-v^{4}w_{\beta',r})\big\}$,
for $\beta=[1]$ and $\beta'=[1,2,1,2,2]$,

\item
  $\prod^{1\leq r\leq d_{\beta'}}_{1\leq s\leq d_{\beta}}
   \big\{(w_{\beta,s}-v^{-6}w_{\beta',r})(w_{\beta,s}-v^{2}w_{\beta',r})\big\}$,
for $\beta=[1]$ and $\beta'=[1,2,2]$,

\item
  $\prod^{1\leq r\leq d_{\beta'}}_{1\leq s\leq d_{\beta}}
   \big\{(w_{\beta,s}-v^{-6}w_{\beta',r})(w_{\beta,s}-v^{2}w_{\beta',r})(w_{\beta,s}-v^{4}w_{\beta',r})\big\}$,
for $\beta=[1]$ and $\beta'=[1,2,2,2]$,

\item
  $\prod^{1\leq r\leq d_{\beta'}}_{1\leq s\leq d_{\beta}}(w_{\beta,s}-w_{\beta',r})$,
for $\beta=[1]$ and $\beta'=[2]$,

\item
  $\prod^{1\leq r\leq d_{\beta'}}_{1\leq s\leq d_{\beta}}
   \big\{ (w_{\beta,s}-v^{8}w_{\beta',r})(w_{\beta,s}-v^{-6}w_{\beta',r})(w_{\beta,s}-v^{6}w_{\beta',r})
   (w_{\beta,s}-v^{-4}w_{\beta',r})(w_{\beta,s}-v^{-2}w_{\beta',r})\big\}$,
for $\beta=[1,2]$ and $\beta'=[1,2,1,2,2]$,

\item
  $\prod^{1\leq r\leq d_{\beta'}}_{1\leq s\leq d_{\beta}}
   \big\{ (w_{\beta,s}-v^{-6}w_{\beta',r})(w_{\beta,s}-v^{6}w_{\beta',r}) (w_{\beta,s}-v^{-2}w_{\beta',r})\big\}$,
for $\beta=[1,2]$ and $\beta'=[1,2,2]$,

\item
  $\prod^{1\leq r\leq d_{\beta'}}_{1\leq s\leq d_{\beta}}
   \big\{ (w_{\beta,s}-v^{-6}w_{\beta',r})(w_{\beta,s}-v^{6}w_{\beta',r}) (w_{\beta,s}-v^{-2}w_{\beta',r})
    (w_{\beta,s}-v^{2}w_{\beta',r})\big\}$,
for $\beta=[1,2]$ and $\beta'=[1,2,2,2]$,

\item
  $\prod^{1\leq r\leq d_{\beta'}}_{1\leq s\leq d_{\beta}}(w_{\beta,s}-v^{-2}w_{\beta',r})$,
for $\beta=[1,2]$ and $\beta'=[2]$,

\item
  $\prod^{1\leq r\leq d_{\beta'}}_{1\leq s\leq d_{\beta}}
   \big\{ (w_{\beta,s}-v^{-8}w_{\beta',r})(w_{\beta,s}-v^{-6}w_{\beta',r})^{2}(w_{\beta,s}-v^{6}w_{\beta',r})
   (w_{\beta,s}-v^{-4}w_{\beta',r})(w_{\beta,s}-v^{4}w_{\beta',r})(w_{\beta,s}-v^{2}w_{\beta',r})\big\}$,
for $\beta=[1,2,1,2,2]$ and $\beta'=[1,2,2]$,

\item
  $\prod^{1\leq r\leq d_{\beta'}}_{1\leq s\leq d_{\beta}}
   \big\{(w_{\beta,s}-v^{-8}w_{\beta',r})(w_{\beta,s}-v^{-6}w_{\beta',r})^{2}(w_{\beta,s}-v^{6}w_{\beta',r})
   (w_{\beta,s}-v^{-4}w_{\beta',r})(w_{\beta,s}-v^{4}w_{\beta',r})(w_{\beta,s}-v^{2}w_{\beta',r})^{2}
   (w_{\beta,s}-v^{-2}w_{\beta',r})\big\}$,
for $\beta=[1,2,1,2,2]$ and $\beta'=[1,2,2,2]$,

\item
  $\prod^{1\leq r\leq d_{\beta'}}_{1\leq s\leq d_{\beta}}
   \big\{ (w_{\beta,s}-v^{-6}w_{\beta',r})(w_{\beta,s}-v^{-2}w_{\beta',r})\big\}$,
for $\beta=[1,2,1,2,2]$ and $\beta'=[2]$,

\item
  $\prod^{1\leq r\leq d_{\beta'}}_{1\leq s\leq d_{\beta}}
   \big\{ (w_{\beta,s}-v^{-6}w_{\beta',r})(w_{\beta,s}-v^{6}w_{\beta',r})(w_{\beta,s}-v^{-4}w_{\beta',r})
   (w_{\beta,s}-v^{4}w_{\beta',r})(w_{\beta,s}-v^{-2}w_{\beta',r})\big\}$,
for $\beta=[1,2,2]$ and $\beta'=[1,2,2,2]$,

\item
  $\prod^{1\leq r\leq d_{\beta'}}_{1\leq s\leq d_{\beta}}(w_{\beta,s}-v^{-4}w_{\beta',r})$,
for $\beta=[1,2,2]$ and $\beta'=[2]$,

\item
  $\prod^{1\leq r\leq d_{\beta'}}_{1\leq s\leq d_{\beta}}(w_{\beta,s}-v^{-6}w_{\beta',r})$,
for $\beta=[1,2,2,2]$ and $\beta'=[2]$.

\end{itemize}
This completes our proof.
\end{proof}

\begin{Prop}\label{vanishG}
Lemma {\rm \ref{vanish}}  is valid for type $G_{2}$ and specialization maps $\phi_{\unl{d}}$ as in \eqref{speG}.
\end{Prop}

\begin{proof}
Given $\unl{d},\unl{d}'\in \mathrm{KP}(\unl{k})$ with $\unl{d}'<\unl{d}$, and any $\sigma\in\mathfrak{S}_{\unl{k}}$,
we will show that $\phi_{\unl{d}'}(\sigma(F_{h}))=0$. Let $x^{(*,*)}_{*,*}$ be the above special splitting of the
variables $\{x_{i,l}\}_{i=1,2}^{1\leq l\leq k_i}$ for $\phi_{\unl{d}}$, see~(\ref{eq:part-symm},~\ref{eq:part-symm-2}),
and let $x'^{(*,*)}_{*,*}$ be any splitting of the variables $\{x_{i,l}\}_{i=1,2}^{1\leq l\leq k_i}$ for $\phi_{\unl{d}'}$,
see~\eqref{splitvariable}. We consider the following sets of $x^{(*,*)}_{*,*}$-variables:
\begin{equation}\label{Zbetas}
  Z^{>(\beta,s)}_{i}=\big\{x^{(\alpha,r)}_{i,*} \, \big| \, (\alpha,r)>(\beta,s) \big\}
  \qquad \forall\ (\beta,s)\in\Delta^{+}\times\BN, 1\leq i\leq 2.
\end{equation}

Let $\beta$ be the smallest positive root such that $d'_{\beta}<d_{\beta}$. Following Lemma \ref{Gs2}, we can assume
that for any $\alpha\leq \beta$ the property \eqref{fsunld} holds for the variables $x'^{(\alpha,*)}_{*,*}$, since
otherwise the $\phi_{\unl{d}'}$-specialization is zero. Thus, without loss of generality, we can assume $d'_{\alpha}=0$
for all $\alpha\leq \beta$. Then, we have the following case-by-case analysis:
\begin{itemize}[leftmargin=0.7cm]

\item
Case $\beta=[1]$. Let $\sigma(x^{(\beta,1)}_{1,1})=x'^{(\gamma,r)}_{1,t}$ for some $\gamma>\beta$. Since $F_{h}$ contains
the factor $\zeta_{1,2}(x^{(\beta,1)}_{1,1}/Z^{>(\beta,1)}_{2})$, and $x'^{(\gamma,r)}_{2,t}\in \sigma(Z^{>(\beta,1)}_{2})$,
we get $\phi_{\unl{d}'}(\sigma(F_{h}))=0$.

\item
Case $\beta=[1,2]$. Let $\sigma(x^{(\beta,1)}_{1,1})=x'^{(\gamma,r)}_{1,t}$ for some $\gamma>\beta$. Since $F_{h}$
contains the factor $\zeta_{1,2}(x^{(\beta,1)}_{1,1}/Z^{>(\beta,1)}_{2})$, we have $\phi_{\unl{d}'}(\sigma(F_{h}))=0$
unless $\sigma(x^{(\beta,1)}_{2,1})=x'^{(\gamma,r)}_{2,t}$. In this case, we again get $\phi_{\unl{d}'}(\sigma(F_{h}))=0$
as $F_{h}$ contains the factor $\zeta_{2,2}(x^{(\beta,1)}_{2,1}/Z^{>(\beta,1)}_{2})$ and
$x'^{(\gamma,r)}_{2,t+1}\in \sigma(Z^{>(\beta,1)}_{2})$.

\item
Case $\beta=[1,2,1,2,2]$. Let $\sigma(x^{(\beta,1)}_{1,1})=x'^{(\gamma_{1},r_{1})}_{1,t_{1}}$ for some $\gamma_{1}>\beta$.
Then, we get $\phi_{\unl{d}'}(\sigma(F_{h}))=0$ unless there is some $1\leq \ell_{1}<\ell_{2}\leq 3$
such that $\sigma(x^{(\beta,1)}_{2,\ell_{1}})=x'^{(\gamma_{1},r_{1})}_{2,t_{1}}$,
$\sigma(x^{(\beta,1)}_{2,\ell_{2}})=x'^{(\gamma_{1},r_{1})}_{2,t_{1}+1}$. Similarly, let
$\sigma(x^{(\beta,1)}_{1,2})=x'^{(\gamma_{2},r_{2})}_{1,t_{2}}$ for some $\gamma_{2}>\beta$, then
$\phi_{\unl{d}'}(\sigma(F_{h}))=0$ unless there is some $1\leq \ell_{3}<\ell_{4}\leq 3$ such that
$\sigma(x^{(\beta,1)}_{2,\ell_{3}})=x'^{(\gamma_{2},r_{2})}_{2,t_{2}}$,
$\sigma(x^{(\beta,1)}_{2,\ell_{4}})=x'^{(\gamma_{2},r_{2})}_{2,t_{2}+1}$. This implies that
$(\gamma_{1},r_{1})=(\gamma_{2},r_{2})$, $t_{1}=1$, $t_{2}=2$, thus contradicting $\gamma_{1}>\beta$.
Hence, $\phi_{\unl{d}'}(\sigma(F_{h}))=0$.

\item
Case $\beta=[1,2,2]$. Let $\sigma(x^{(\beta,1)}_{1,1})=x'^{(\gamma,r)}_{1,1}$ for some $\gamma>\beta$.
Then, we have $\phi_{\unl{d}'}(\sigma(F_{h}))=0$ unless $\sigma(x^{(\beta,1)}_{2,1})=x'^{(\gamma,r)}_{2,1}$
and $\sigma(x^{(\beta,1)}_{2,2})=x'^{(\gamma,r)}_{2,2}$. As $F_{h}$ contains the factor
$\zeta_{2,2}(x^{(\beta,1)}_{2,2}/Z^{>(\beta,1)}_{2})$ and $x'^{(\gamma,r)}_{2,3}\in \sigma(Z^{>(\beta,1)}_{2})$,
we thus again obtain $\phi_{\unl{d}'}(\sigma(F_{h}))=0$.

\item
Case $\beta=[1,2,2,2]$. This case is impossible with our assumptions $d'_{\leq \beta}=0$ and $d_{\beta}\ne 0$, as we have
  $d'_{[2]}\cdot [2] = \sum_{\alpha\in \Delta^+} d'_\alpha \cdot \alpha = \sum_{i=1}^2 k_i \alpha_i =
   \sum_{\alpha\in \Delta^+} d_\alpha \cdot \alpha = d_{[1,2,2,2]}\cdot [1,2,2,2] + d_{[2]}\cdot [2]$.

\end{itemize}
This completes our proof.
\end{proof}

\begin{Prop}\label{spanG}
Lemma {\rm \ref{span}}  is valid for type $G_{2}$ and specialization maps $\phi_{\unl{d}}$ as in \eqref{speG}.
\end{Prop}

\begin{proof}
The wheel conditions~\eqref{eq:wheel-G2} for $F \in S_{\unl{k}}$ guarantee that $\phi_{\unl{d}}(F)$
(which is a Laurent polynomial in the variables $\{w_{\beta,s}\}$) vanishes under specific
specializations $w_{\beta,s}=v^\#\cdot w_{\beta',s'}$. To evaluate the aforementioned powers $\#$ of $v$
and the orders of vanishing, let us view $\phi_{\unl{d}}$ as a step-by-step specialization in each
interval $[\beta]$, ordered in the decreasing order with respect to
%$(\beta,s)<(\beta',s')$, defined via $\beta<\beta'$ or $\beta=\beta', s<s'$.
\begin{equation*}\label{eq:order-rootandinteger}
  (\beta,s)<(\beta',s') \quad \mathrm{iff}\quad \beta<\beta' \ \ \mathrm{or}\ \ \beta=\beta' \ \mathrm{and}\ s<s' \,.
\end{equation*}
We note that this computation is local with respect to any fixed pair $(\beta,s)\leq (\beta',s')$. Consider
\begin{equation}
  \unl{d}_{1}=\{d_{\beta}=2,d_{\alpha}=0,\alpha\neq \beta\}\in\text{KP}(\unl{k}_{1}),
\end{equation}
\begin{equation}
  \unl{d}_{2}=\{d_{\beta}=d_{\beta'}=1,d_{\alpha}=0, \alpha\neq \beta,\beta'\}\in\text{KP}(\unl{k}_{2}).
\end{equation}
For any $F_{1}\in S_{\unl{k}_{1}}, F_{2}\in S_{\unl{k}_{2}}$, it thus suffices to show that $\phi_{\unl{d}_{1}}(F_{1})$
is divisible by $G_{\beta}$ if $\phi_{\unl{d}}(F_{1})=0$ for any $\unl{d}<\unl{d}_{1}$ and $\phi_{\unl{d}_{2}}(F_{2})$
is divisible by $G_{\beta,\beta'}$ if $\phi_{\unl{d}}(F_{2})=0$ for any $\unl{d}<\unl{d}_{2}$, cf.~\eqref{eq:Gfactors-G2}.

\medskip

For $\unl{d}_{1}=\{d_{\beta}=2,\ d_{\alpha\neq \beta}=0\}$, the nontrivial cases are
$\beta=[1,2,2],[1,2,2,2],[1,2,1,2,2]$.
\begin{itemize}[leftmargin=*]

\item
$\beta=[1,2,2]$. For $F_{1}$, under $\phi_{\unl{d}_{1}}$ the wheel condition $F_{1}=0$ once
$x^{(\beta,1)}_{1,1}=v^{6}x^{(\beta,2)}_{1,1}=v^{3}x^{(\beta,1)}_{2,1}$  becomes $\phi_{\unl{d}_{1}}(F_{1})=0$
once $w_{\beta,1}=v^{6}w_{\beta,2}$, thus we get the vanishing factor $(w_{\beta,1}-v^{6}w_{\beta,2})$;
the wheel condition $F_{1}=0$ once
  $x^{(\beta,1)}_{2,2}=v^{2}x^{(\beta,1)}_{2,1}=v^{4}x^{(\beta,2)}_{2,2}=v^{6}x^{(\beta,2)}_{2,1}=v^{3}x^{(\beta,2)}_{1,1}$
becomes $\phi_{\unl{d}_{1}}(F_{1})=0$ once $w_{\beta,1}=v^{4}w_{\beta,2}$, thus we get the vanishing factor
$(w_{\beta,1}-v^{4}w_{\beta,2})$. Since $\phi_{\unl{d}_{1}}(F_{1})$ is symmetric with respect to $w_{\beta,1}$ and
$w_{\beta,2}$, we also have the vanishing factor $(w_{\beta,2}-v^{6}w_{\beta,1})$ and $(w_{\beta,2}-v^{4}w_{\beta,1})$,
thus we get the vanishing factor $G_{\beta}$.

\item
$\beta=[1,2,2,2]$. Besides the vanishing factors appearing in $G_{[1,2,2]}$, the wheel condition $F_{1}=0$ once
$x^{(\beta,1)}_{2,3}=v^{2}x^{(\beta,1)}_{2,2}=v^{4}x^{(\beta,1)}_{2,1}=v^{6}x^{(\beta,2)}_{2,1}=v^{3}x^{(\beta,2)}_{1,1}$
becomes $\phi_{\unl{d}_{1}}(F_{1})=0$ once $w_{\beta,1}=v^{2}w_{\beta,2}$, thus we get the vanishing factor
$(w_{\beta,1}-v^{2}w_{\beta,2})(w_{\beta,2}-v^{2}w_{\beta,1})$, and we get all the vanishing factors in $G_{\beta}$.

\item
$\beta=[1,2,1,2,2]$. Besides the vanishing factors in $G_{[1,2,2,2]}$, the wheel conditions $F_{1}=0$ once
$x^{(\beta,1)}_{1,1}=v^{6}x^{(\beta,2)}_{1,2}=v^{3}x^{(\beta,1)}_{2,1}$ or
$x^{(\beta,1)}_{1,2}=v^{6}x^{(\beta,2)}_{1,2}=v^{3}x^{(\beta,1)}_{2,2}$ or
$x^{(\beta,1)}_{1,2}=v^{6}x^{(\beta,2)}_{1,1}=v^{3}x^{(\beta,1)}_{2,2}$ give us the rest vanishing factors in $G_{\beta}$.

\end{itemize}

For $\unl{d}_{2}=\{d_{\beta}=d_{\beta'}=1,\ d_{\alpha\neq \beta,\beta'}=0\}$, take $F_{2}\in S_{\unl{k}}$, then
\begin{itemize}[leftmargin=*]

\item
$(\beta,\beta')=([1],[1,2])$. Under $\phi_{\unl{d}_{2}}$ the wheel condition $F_{2}=0$ once
$x^{(\beta',1)}_{1,1}=v^{6}x^{(\beta,1)}_{1,1}=v^{3}x^{(\beta',1)}_{2,1}$ becomes $\phi_{\unl{d}_{2}}(F_{2})=0$
once $w_{\beta,1}=v^{-6}w_{\beta',1}$, thus we have the factor $G_{\beta,\beta'}$.

\item
$(\beta,\beta')=([1],[1,2,1,2,2])$. Besides the factor $G_{[1],[1,2]}$, the wheel condition
$x^{(\beta',1)}_{1,2}=v^{6}x^{(\beta,1)}_{1,1}=v^{3}x^{(\beta',1)}_{2,2}$ gives the vanishing factor
$w_{\beta,1}=v^{-4}w_{\beta',1}$. Let $\unl{d}_{3}=\{d_{[1,2]}=3, d_{\alpha\neq [1,2]}=0\}$, then $\unl{d}_{3}<\unl{d}_{2}$,
by $\phi_{\unl{d}_{3}}(F_{2})=0$ we know $\phi_{\unl{d}_{2}}(F_{2})=0$ once $w_{\beta,1}=v^{4}w_{\beta',1}$.

\item
$(\beta,\beta')=([1],[1,2,2])$. Besides the factor $G_{[1],[1,2]}$, let
$\unl{d}_{3}=\{d_{[1,2]}=2, d_{\alpha\neq [1,2]}=0\}$, then $\unl{d}_{3}<\unl{d}_{2}$, by $\phi_{\unl{d}_{3}}(F_{2})=0$
we know $\phi_{\unl{d}_{2}}(F_{2})=0$ once $w_{\beta,1}=v^{2}w_{\beta',1}$.

\item
$(\beta,\beta')=([1],[1,2,2,2])$. Besides the factor $G_{[1],[1,2]}$, let
$\unl{d}_{3}=\{d_{[1,2,1,2,2]}=1, d_{\alpha\neq [1,2,1,2,2]}=0\}$, then $\unl{d}_{3}<\unl{d}_{2}$,
by $\phi_{\unl{d}_{3}}(F_{2})=0$ we know $\phi_{\unl{d}_{2}}(F_{2})=0$ once $w_{\beta,1}=v^{2}w_{\beta',1}$.
Let $\unl{d}_{4}=\{d_{[1,2]}=1, d_{[1,2,2]}=1, d_{\alpha\neq [1,2],[1,2,2]}=0\}$, then $\unl{d}_{4}<\unl{d}_{2}$,
by $\phi_{\unl{d}_{4}}(F_{2})=0$ we know $\phi_{\unl{d}_{2}}(F_{2})=0$ once $w_{\beta,1}=v^{4}w_{\beta',1}$.

\item
$(\beta,\beta')=([1],[2])$. Let $\unl{d}_{3}=\{d_{[1,2]}=1, d_{\alpha\neq [1,2]}=0\}$, then $\unl{d}_{3}<\unl{d}_{2}$,
by $\phi_{\unl{d}_{3}}(F_{2})=0$ we know $\phi_{\unl{d}_{2}}(F_{2})=0$ once $w_{\beta,1}=w_{\beta',1}$.

\item
$(\beta,\beta')=([1,2],[1,2,1,2,2])$. The factors in $G_{[1,2],[1,2,1,2,2]}$ all appear in $G_{[1,2,1,2,2]}$,
and they appear due to similar wheel conditions.

\item
$(\beta,\beta')=([1,2],[1,2,2])$. Besides the factors in $G_{[1,2]}$, let
$\unl{d}_{3}=\{d_{[1,2,1,2,2]}=1, d_{\alpha\neq [1,2,1,2,2]}=0\}$, then $\unl{d}_{3}<\unl{d}_{2}$,
by $\phi_{\unl{d}_{3}}(F_{2})=0$ we know $\phi_{\unl{d}_{2}}(F_{2})=0$ once $w_{\beta,1}=v^{-2}w_{\beta',1}$.

\item
$(\beta,\beta')=([1,2],[1,2,2,2])$. Besides the factors in $G_{[1,2]}$, let
$\unl{d}_{3}=\{d_{[1,2,1,2,2]}=1, d_{[2]}=1, d_{\alpha\neq [1,2,1,2,2],[2]}=0\}$,
$\unl{d}_{4}=\{d_{[1,2,2]}=2, d_{\alpha\neq [1,2,2]}=0\}$, then $\unl{d}_{3},\unl{d}_{4}<\unl{d}_{2}$.
By $\phi_{\unl{d}_{3}}(F_{2})=0$ we know $\phi_{\unl{d}_{2}}(F_{2})=0$ once $w_{\beta,1}=v^{-2}w_{\beta',1}$,
and by $\phi_{\unl{d}_{4}}(F_{2})=0$ we know $\phi_{\unl{d}_{2}}(F_{2})=0$ once $w_{\beta,1}=v^{2}w_{\beta',1}$,
thus we get the vanishing factor $G_{[1,2],[1,2,2,2]}$.

\item
$(\beta,\beta')=([1,2],[2])$. Let $\unl{d}_{3}=\{d_{[1,2,2]}=1, d_{\alpha\neq [1,2,2]}=0\}$, then
$\unl{d}_{3}<\unl{d}_{2}$, by $\phi_{\unl{d}_{3}}(F_{2})=0$ we know $\phi_{\unl{d}_{2}}(F_{2})=0$ once
$w_{\beta,1}=v^{-2}w_{\beta',1}$.

\item
$(\beta,\beta')=([1,2,1,2,2],[1,2,2])$. The factors in $G_{[1,2,1,2,2],[1,2,2]}$ all appear in $G_{[1,2,1,2,2]}$,
and they appear due to similar wheel conditions.

\item
$(\beta,\beta')=([1,2,1,2,2],[1,2,2,2])$. Besides the factors appearing in $G_{[1,2,1,2,2]}$,
let $\unl{d}_{3}=\{d_{[1,2,2]}=3, d_{\alpha\neq [1,2,2]}=0\}$, then $\unl{d}_{3}<\unl{d}_{2}$,
by $\phi_{\unl{d}_{3}}(F_{2})=0$ we know $\phi_{\unl{d}_{2}}(F_{2})=0$ once $w_{\beta,1}=v^{2}w_{\beta',1}$.

\item
$(\beta,\beta')=([1,2,1,2,2],[2])$. Besides the factors appearing in $G_{[1,2,1,2,2]}$,
let $\unl{d}_{3}=\{d_{[1,2,2]}=2, d_{\alpha\neq [1,2,2]}=0\}$, then $\unl{d}_{3}<\unl{d}_{2}$,
by $\phi_{\unl{d}_{3}}(F_{2})=0$ we know $\phi_{\unl{d}_{2}}(F_{2})=0$ once $w_{\beta,1}=v^{-2}w_{\beta',1}$.

\item
$(\beta,\beta')=([1,2,2],[1,2,2,2])$. The factors in $G_{[1,2,2],[1,2,2,2]}$ all appear in $G_{[1,2,2,2]}$,
and they appear due to similar wheel conditions.

\item
$(\beta,\beta')=([1,2,2],[2])$. Let $\unl{d}_{3}=\{d_{[1,2,2,2]}=1, d_{\alpha\neq [1,2,2,2]}=0\}$, then
$\unl{d}_{3}<\unl{d}_{2}$, by $\phi_{\unl{d}_{3}}(F_{2})=0$ we know $\phi_{\unl{d}_{2}}(F_{2})=0$ once
$w_{\beta,1}=v^{-4}w_{\beta',1}$.

\item
$(\beta,\beta')=([1,2,2,2],[2])$. The wheel condition $F_{2}=0$ once
$x^{(\beta',1)}_{2,1}=v^{2}x^{(\beta,1)}_{2,3}=v^{4}x^{(\beta,1)}_{2,2}=v^{6}x^{(\beta,1)}_{2,1}=v^{3}x^{(\beta,1)}_{1,1}$
becomes $\phi_{\unl{d}_{2}}(F_{2})=0$ once $w_{\beta,1}=v^{-6}w_{\beta',1}$.

\end{itemize}
This completes our proof.
\end{proof}

%%%%%%%%%%%%%%%%%%%%%%%%%%%%%%%%%%%%%%%%%%%%%%%%%%%%%%%%%%%%%%%%%%%%%%%%%%%%%%%%%%%%%%%%%%%%%%%%%%%%%%%%%%%
%%%%%%%%%%%%%%%%%%%%%%%%%%%% end of updated Yue's original text over the field %%%%%%%%%%%%%%%%%%%%%%%%%%%%
%%%%%%%%%%%%%%%%%%%%%%%%%%%%%%%%%%%%%%%%%%%%%%%%%%%%%%%%%%%%%%%%%%%%%%%%%%%%%%%%%%%%%%%%%%%%%%%%%%%%%%%%%%%

Combining Propositions \ref{spekpG}--\ref{spanG}, we immediately obtain the shuffle algebra realization
and the PBWD theorem for $\qfg$:

\begin{Thm}\label{shufflePBWDG}
(a) $\Psi\colon \qfg \,\iso\, S$ of~\eqref{eq:Psi-homom} is a $\BQ(v)$-algebra isomorphism.

\medskip
\noindent
(b) For any choices of $s_k$ and $\lambda_k$ in the definition~(\ref{rootvector1},~\ref{rootvector2})
of quantum root vectors $E_{\beta,s}$, the ordered PBWD monomials $\{E_{h}\}_{h\in H}$ from \eqref{PBWDbases}
form a $\BQ(v)$-basis of $\qfg$.
\end{Thm}

   %%%%%%%%%%%%%%%%%%%%%%%%%%%%%%%%%%%%%%%%%%%%%%%%%%%%%%%%%%%%%%%%%%
   %%%%%%%%%%%%%%%%%%%%%%%%%%%%%%%%%%%%%%%%%%%%%%%%%%%%%%%%%%%%%%%%%%
   %%%%%%%%%%%%%%%%%%%%%%%%%%%%%%%%%%%%%%%%%%%%%%%%%%%%%%%%%%%%%%%%%%

\subsection{Integral form $\integralgl$ and its shuffle algebra realization}\label{lusg}

Let us consider the {\em divided powers}
\begin{equation}\label{dividedpowers}
  \mathbf{E}_{i,r}^{(k)}\coloneqq \frac{e_{i,r}^{k}}{[k]_{v_{i}}!} \quad \forall\ 1\leq i\leq 2,\  r\in\mathbb{Z},\ k\in\BN.
\end{equation}
Following \cite[\S 7.7]{Gro94}, we define the  integral form $\integralgl$ as the $\BZ[v,v^{-1}]$-subalgebra of $\qfg$
generated by $\{\mathbf{E}_{i,r}^{(k)}\}_{1\leq i\leq 2, r\in\BZ}^{k\in\BN}$. For any $(\beta,s)\in\Delta^{+}\times \BZ$,
we define the \emph{normalized divided powers} of the quantum root vectors from~\eqref{rvg1}--\eqref{rvg5} via:
\begin{equation}\label{looplusg}
  \tilde{\mathbf{E}}_{\beta,s}^{\pm,(k)}\coloneqq
  \begin{cases}
    \frac{(\tilde{E}_{\beta,s}^{\pm})^{k}}{[k]_{v_{\beta}}!} &\quad \text{if}\ \beta=[1],[2],[1,2]\\
    \frac{(\tilde{E}_{\beta,s}^{\pm})^{k}}{([2]_{v}!)^{k}[k]_{v_{\beta}}!} &\quad \text{if}\ \beta=[1,2,2]\\
    \frac{(\tilde{E}_{\beta,s}^{\pm})^{k}}{([3]_{v}!)^{k}[k]_{v_{\beta}}!} &\quad \text{if}\   \beta=[1,2,2,2],[1,2,1,2,2]
  \end{cases}.
\end{equation}
Similarly to \cite[Proposition 1.2]{Tsy22}, we have:

\begin{Prop}\label{integralrvg}
For any $\beta\in\Delta^{+}$, $s\in\BZ$, $k\in\BN$, we have $\tilde{\mathbf{E}}_{\beta,s}^{\pm,(k)}\in \integralgl$.
\end{Prop}

\begin{proof}
Let $\qgg$ be the ``positive subalgebra'' of the Drinfeld-Jimbo quantum group of $\fg_{2}$. Thus, $\qgg$ is the
$\BQ(v)$-algebra generated by $\{E_{1},E_{2}\}$ subject to the $v$-Serre relations:
\begin{equation}
  \sum_{k=0}^{1-a_{ij}} (-1)^{k} \left[\begin{matrix} 1-a_{ij}\\k\end{matrix}\right]_{v_{i}}
  E_{i}^{k}E_{j}E_{i}^{1-a_{ij}-k}=0,\quad i\neq j.
\end{equation}
Let $\qggi$ be the \emph{Lusztig integral form} defined as the $\BZ[v,v^{-1}]$-subalgebra of $\qgg$ generated
by the divided powers
  \[E_{i}^{(k)}\coloneqq \frac{E_{i}^{k}}{[k]_{v_{i}}!}\qquad \forall\ 1\leq i\leq 2,\  k\in\BN.\]

Recall our specific convex order~\eqref{lynorderg} on $\Delta^{+}$. Let $\{\hat{E}_{\beta}^{-}\}_{\beta\in\Delta^{+}}$
denote Lusztig's quantum root vectors of $\qgg$ associated to this convex order (defined through the use
of Lusztig's braid group action, see~\cite[\S37.1.3]{Lus-book}). According to~\cite[Theorem 6.6]{Lus90}, we have:
\begin{equation}
  \hat{\mathbf{E}}_{\beta}^{-,(k)}\coloneqq \frac{(\hat{E}^{-}_{\beta})^{k}}{[k]_{v_{\beta}}!}\in\qggi
  \qquad \forall\ \beta\in\Delta^{+},\ k\in\BN.\label{lusrvg}
\end{equation}
On the other hand, let us define another set of quantum root vectors $\{\tilde{E}^{-}_{\beta}\}_{\beta\in\Delta^{+}}$
in $\qgg$ using $v$-commutators similar to \eqref{rvg1}--\eqref{rvg5}:
\begin{equation}
\begin{aligned}
  & \tilde{E}^{-}_{[i]}\coloneqq E_{i},\ 1\leq i\leq 2;\quad  \tilde{E}^{-}_{[1,2]}\coloneqq [E_{1},E_{2}]_{v^{-3}};
    \quad \tilde{E}^{-}_{[1,2,2]}\coloneqq [[E_{1},E_{2}]_{v^{-3}},E_{2}]_{v^{-1}};\\
  & \tilde{E}^{-}_{[1,2,2,2]}\coloneqq [[[E_{1},E_{2}]_{v^{-3}},E_{2}]_{v^{-1}},E_{2}]_{v};\
    \tilde{E}^{-}_{[1,2,1,2,2]}\coloneqq [[E_{1},E_{2}]_{v^{-3}},[[E_{1},E_{2}]_{v^{-3}},E_{2}]_{v^{-1}}]_{v}.
\label{lynrvf}
\end{aligned}
\end{equation}
Due to \cite[Proposition 5.5.2]{LS91}, the quantum root vectors $\tilde{E}^{-}_{\beta}$ and $\hat{E}^{-}_{\beta}$
differ only by a scalar multiple for any $\beta\in\Delta^{+}$. These scalars are determined explicitly for Drinfeld-Jimbo
quantum groups of any simple Lie algebra $\fg$ in \cite[Theorem 4.2]{BKM14} (note that the parameter $q$ of~\cite{BKM14}
equals our $v^{-1}$). Specifically, in our case we have:
\begin{equation}
  \tilde{E}^{-}_{\beta}=
  \begin{cases}
    \hat{E}^{-}_{\beta} & \quad \text{if}\ \beta=[1],[2],[1,2]\\
    [2]_{v}!\hat{E}^{-}_{\beta} &\quad \text{if}\ \beta=[1,2,2]\\
    [3]_{v}!\hat{E}^{-}_{\beta} & \quad \text{if}\ \beta=[1,2,2,2],[1,2,1,2,2]
  \end{cases}.
\label{scalarluslyn}
\end{equation}

Let us now pass from the finite to the loop setup. First, we note that comparing the coefficients of
$z_{1}^{-s_{1}}\cdots z_{1-a_{ij}}^{-s_{1}}w^{-s_{2}}$ in \eqref{serreloop} for any $s_{1},s_{2}\in\BZ$, we obtain:
  \[\sum_{k=0}^{1-a_{ij}}(-1)^{k}\left[\begin{matrix} 1-a_{ij}\\k\end{matrix}\right]_{v_{i}}
    e_{i,s_{1}}^{k}e_{j,s_{2}}e_{i,s_{1}}^{1-a_{ij}-k}=0,\quad i\neq j.\]
Thus, the assignment $E_{1}\mapsto e_{1,s_{1}},E_{2}\mapsto e_{2,s_{2}}$ gives rise to an algebra homomorphism
  \[\eta_{s_{1},s_{2}}\colon \qgg\longrightarrow \qfg.\]
Clearly, we have $\eta_{s_{1},s_{2}}(\qggi)\subset \integralgl$.
Combining \eqref{lusrvg}--\eqref{scalarluslyn} with~\eqref{looplusg}, we thus~get:
\begin{equation*}
  \tilde{\mathbf{E}}^{-,(k)}_{\beta,s}=\eta_{s_{1},s_{2}}(\hat{\mathbf{E}}_{\beta}^{-,(k)})\in \integralgl.
\end{equation*}

To prove the other inclusions $\tilde{\mathbf{E}}^{+,(k)}_{\beta,s}\in \integralgl$, let us consider the convex order
on $\Delta^{+}$ opposite to \eqref{lynorderg}:
\begin{equation}
  [2]<[1,2,2,2]<[1,2,2]<[1,2,1,2,2]<[1,2]<[1].
\end{equation}
Let $\{\hat{E}^{+}_{\beta}\}_{\beta\in\Delta^{+}}$ denote the set of Lusztig's quantum root vectors associated to
that convex order, and define another set of quantum root vectors $\{\tilde{E}^{+}_{\beta}\}_{\beta\in\Delta^{+}}$
via $v$-commutators:
\begin{equation}
\begin{aligned}
  & \tilde{E}^{+}_{[i]}\coloneqq E_{i},\ 1\leq i\leq 2;\
    \tilde{E}^{+}_{[1,2]}\coloneqq [E_{2},E_{1}]_{v^{-3}}\doteq [E_{1},E_{2}]_{v^{3}};\\
  & \tilde{E}^{+}_{[1,2,2]}\coloneqq [E_{2},[E_{2},E_{1}]_{v^{-3}}]_{v^{-1}}\doteq [[E_{1},E_{2}]_{v^{3}},E_{2}]_{v};\\
  & \tilde{E}^{+}_{[1,2,2,2]}\coloneqq [E_{2},[E_{2},[E_{2},E_{1}]_{v^{-3}}]_{v^{-1}}]_{v}\doteq
    [[[E_{1},E_{2}]_{v^{3}},E_{2}]_{v},E_{2}]_{v^{-1}} ;\\
  & \tilde{E}^{+}_{[1,2,1,2,2]}\coloneqq [[E_{2},[E_{2},E_{1}]_{v^{-3}}]_{v^{-1}},[E_{2},E_{1}]_{v^{-3}}]_{v}\doteq
    [[E_{1},E_{2}]_{v^{3}},[[E_{1},E_{2}]_{v^{3}},E_{2}]_{v}]_{v^{-1}}.
\end{aligned}
\end{equation}
Then, due to~\cite[Theorem 6.6]{Lus90} and~\cite[Theorem 4.2]{BKM14} the analogues of~\eqref{lusrvg}
and~\eqref{scalarluslyn} with the superscript $-$ replaced by $+$ hold. Therefore, we likewise obtain:
  \[\tilde{\mathbf{E}}^{+,(k)}_{\beta,s}\doteq \eta_{s_{1},s_{2}}(\hat{\mathbf{E}}_{\beta}^{+,(k)})\in \integralgl.\]
This completes our proof.
\end{proof}

For any $\unl{k}\in\BN^{2}$, consider the $\BZ[v,v^{-1}]$-submodule $\mathbf{S}_{\unl{k}}$ of $S_{\unl{k}}$
consisting of rational functions $F$ satisfying the following two conditions:
\begin{enumerate}[leftmargin=1cm]

\item
If $f$ denotes the numerator of $F$ from \eqref{polecondition}, then
\begin{equation}
  f\in \BZ[v,v^{-1}][\{x_{i,r}^{\pm 1}\}_{1\leq i\leq 2}^{1\leq r\leq k_{i}}]^{\mathfrak{S}_{\underline{k}}}.
\label{li1G}
\end{equation}

\item
For any $\unl{d}\in\text{KP}(\underline{k})$, the specialization $\phi_{\unl{d}}(F)$ is divisible by the product
\begin{equation}
  \prod_{\beta\in\Delta^{+}}{\tilde{c}_{\beta}}^{d_{\beta}},
\label{li2G}
\end{equation}
where we define $\{\tilde{c}_{\beta}\}_{\beta\in\Delta^{+}}$ via $\{c_\beta\}_{\beta\in \Delta^+}$ of Lemma~\ref{Gs1}:
\begin{equation}
  \tilde{c}_{\beta}=
  \begin{cases}
    c_{\beta} & \quad \text{if}\  \beta=[1],[2],[1,2]\\
    \frac{c_{\beta}}{[2]_{v}!} & \quad \text{if}\ \beta=[1,2,2]\\
    \frac{c_{\beta}}{[3]_{v}!} & \quad \text{if}\ \beta=[1,2,2,2],[1,2,1,2,2]
  \end{cases}.
\label{ctildebeta}
\end{equation}
\end{enumerate}
We define $\mathbf{S}:=\bigoplus_{\unl{k}\in\BN^{2}}\mathbf{S}_{\unl{k}}$. Then, we have:

\begin{Prop}\label{lintegralG}
$\Psi(\integralgl) \subset \mathbf{S}$.
\end{Prop}

The proof is based on the following simple ``rank 1'' computation from \cite[Lemma 1.3]{Tsy22}:

\begin{Lem}\label{rank2}
For any $\ell\geq 1$, $r\in\BZ$, and $1\leq i\leq 2$, we have
\begin{equation}\label{eq:rank1-power-trig}
  \underbrace{x_{i,1}^{r}\star\cdots\star x_{i,1}^{r}}_{\ell\rm\ times}=
  v_{i}^{-\frac{\ell(\ell-1)}{2}}[\ell]_{v_{i}}!\cdot (x_{i,1}\cdots x_{i,\ell})^{r}.
\end{equation}
\end{Lem}

\begin{proofprop}
For any $m\in\BN, 1\leq i_{1},\dots,i_{m}\leq 2, r_{1},\dots,r_{m}\in\BZ, \ell_{1},\dots,\ell_{m}\in\BN$,~let
  \[F\coloneqq \Psi\big(\mathbf{E}^{(\ell_{1})}_{i_{1},r_{1}}\cdots \mathbf{E}^{(\ell_{m})}_{i_{m},r_{m}}\big),\]
and $f$ be the numerator of $F$ from \eqref{polecondition}. If a variable $x^{(*,*)}_{*,*}$ is plugged into
$\Psi(\mathbf{E}^{(\ell_{q})}_{i_{q},r_{q}})$, then we shall use the following notation:
\begin{equation}\label{eq:o-spot}
  o(x^{(*,*)}_{*,*})=q.
\end{equation}
Thus, $o(x^{(\beta,s)}_{i,t})=q$ means that in the corresponding summand from the symmetrization of
$\Psi(\mathbf{E}^{(\ell_{1})}_{i_{1},r_{1}})\cdots \Psi(\mathbf{E}^{(\ell_{m})}_{i_{m},r_{m}})\cdot (\mathrm{rational\ factor})$,
this $x$-variable is placed as an argument of $\Psi(\mathbf{E}^{(\ell_{q})}_{i_{q},r_{q}})$.

According to Lemma \ref{rank2}:
  \[\Psi(\mathbf{E}^{(\ell_{q})}_{i_{q},r_{q}})=
    v_{i_{q}}^{-\frac{\ell_{q}(\ell_{q}-1)}{2}}(x_{i_{q},1}\cdots x_{i_{q},\ell_{q}})^{r_{q}}
    \qquad \forall\ 1\leq q\leq m,\]
hence, the condition \eqref{li1G} holds. To verify the validity of the divisibility \eqref{li2G}, it suffices to show that
for any $\beta\in\Delta^{+}$ and $1\leq s\leq d_{\beta}$, the total contribution of $\phi_{\unl{d}}$-specializations of the
$\zeta$-factors between the variables $\{x^{(\beta,s)}_{i,t}\}_{i\in\beta}^{1\leq t\leq \nu_{\beta,i}}$ is a multiple of
$\tilde{c}_{\beta}$. This is obvious  for $\beta=[1],[2],[1,2]$. Let us now treat the remaining three cases:
\begin{itemize}[leftmargin=0.7cm]

\item
$\beta=[1,2,2]$.
In this case, the $\phi_{\unl{d}}$-specialization of the corresponding product of the $\zeta$-factors vanishes unless
  \[o(x^{(\beta,s)}_{1,1})\geq o(x^{(\beta,s)}_{2,1})\geq o(x^{(\beta,s)}_{2,2}).\]
On the other hand, $o(x^{(\beta,s)}_{i,t})\neq o(x^{(\beta,s)}_{i',t'})$ for $i\neq i'$.
Thus, we only need to deal with the case
  \[o(x^{(\beta,s)}_{1,1})> o(x^{(\beta,s)}_{2,1})\geq o(x^{(\beta,s)}_{2,2}).\]
In this case, the product of the $\zeta$-factors
  \[ \zeta_{2,1}(x^{(\beta,s)}_{2,1}/x^{(\beta,s)}_{1,1})\cdot\zeta_{2,1}(x^{(\beta,s)}_{2,2}/x^{(\beta,s)}_{1,1})\]
has a numerator
  \[(x^{(\beta,s)}_{2,1}-v^{3}x^{(\beta,s)}_{1,1})(x^{(\beta,s)}_{2,2}-v^{3}x^{(\beta,s)}_{1,1}),\]
which contributes precisely the factor $\langle 3\rangle_{v}\langle 2\rangle_{v}=\tilde{c}_{[1,2,2]}$ under $\phi_{\unl{d}}$.

\item
$\beta=[1,2,2,2]$.
In this case, the $\phi_{\unl{d}}$-specialization of the corresponding product of the $\zeta$-factors vanishes unless
  \[o(x^{(\beta,s)}_{1,1})> o(x^{(\beta,s)}_{2,1})\geq o(x^{(\beta,s)}_{2,2})\geq o(x^{(\beta,s)}_{2,3}).\]
In the latter case, the product of the $\zeta$-factors
  \[\zeta_{2,1}(x^{(\beta,s)}_{2,1}/x^{(\beta,s)}_{1,1})\cdot\zeta_{2,1}(x^{(\beta,s)}_{2,2}/x^{(\beta,s)}_{1,1})\cdot
   \zeta_{2,1}(x^{(\beta,s)}_{2,3}/x^{(\beta,s)}_{1,1})\]
has a numerator
  \[(x^{(\beta,s)}_{2,1}-v^{3}x^{(\beta,s)}_{1,1}) (x^{(\beta,s)}_{2,2}-v^{3}x^{(\beta,s)}_{1,1})
    (x^{(\beta,s)}_{2,3}-v^{3}x^{(\beta,s)}_{1,1}),\]
which contributes precisely the factor $\langle 3\rangle_{v}\langle 2\rangle_{v}\langle 1\rangle_{v}=\tilde{c}_{[1,2,2,2]}$
under $\phi_{\unl{d}}$.

\item
$\beta=[1,2,1,2,2]$.
In this case, the $\phi_{\unl{d}}$-specialization of the corresponding product of the $\zeta$-factors vanishes unless
\begin{equation}
  o(x^{(\beta,s)}_{1,1})> o(x^{(\beta,s)}_{2,1})\geq o(x^{(\beta,s)}_{2,2})\geq o(x^{(\beta,s)}_{2,3})\quad
  \text{and}\quad o(x^{(\beta,s)}_{1,2})>o(x^{(\beta,s)}_{2,2}),
\label{longestG}
\end{equation}
If $o(x^{(\beta,s)}_{1,2})>o(x^{(\beta,s)}_{2,1})$, then the product
  $\prod_{1\leq k\leq 2}^{1\leq l\leq 3}\zeta_{2,1}(x^{(\beta,s)}_{2,l}/x^{(\beta,s)}_{1,k})$
has a numerator
  \[\prod_{1\leq k\leq 2}^{1\leq l\leq 3}(x^{(\beta,s)}_{2,l}-v^{3}x^{(\beta,s)}_{1,k}),\]
which contributes precisely the factor
  $\langle 4\rangle_{v}\langle 3\rangle_{v}^{2}\langle 2\rangle_{v}^{2}\langle 1\rangle_{v}=\tilde{c}_{[1,2,1,2,2]}$
under $\phi_{\unl{d}}$. On the other hand, if $o(x^{(\beta,s)}_{1,2})<o(x^{(\beta,s)}_{2,1})$, then we have
  \[o(x^{(\beta,s)}_{1,1})> o(x^{(\beta,s)}_{2,1})>o(x^{(\beta,s)}_{1,2})> o(x^{(\beta,s)}_{2,2})\geq o(x^{(\beta,s)}_{2,3}),\]
in which case the $\phi_{\unl{d}}$-specialization of the $\zeta$-factors between the $x^{(\beta,s)}_{*,*}$-variables
contributes a total of $[3]_{v}\cdot \tilde{c}_{[1,2,1,2,2]}$.

\end{itemize}
This completes the verification of the divisibility \eqref{li2G}, thus concluding the proof.
\end{proofprop}

\vspace{5pt}
For any $h\in H$, define the ordered monomials (cf.~\eqref{PBWDbases})
\begin{equation}\label{eq:ordered-pbw-int-G}
  \tilde{\mathbf{E}}^{+}_{h} \ =
  \prod_{(\beta,s)\in\Delta^{+}\times\mathbb{Z}}\limits^{\rightarrow}\tilde{\mathbf{E}}_{\beta,s}^{+,(h(\beta,s))},\qquad
  \tilde{\mathbf{E}}^{-}_{h} \ =
  \prod_{(\beta,s)\in\Delta^{+}\times\mathbb{Z}}\limits^{\rightarrow}\tilde{\mathbf{E}}_{\beta,s}^{-,(h(\beta,s))}.
\end{equation}
According to Proposition~\ref{lintegralG}, we have $\tilde{\mathbf{E}}^{\pm}_{h}\in \mathbf{S}$.
For any $\epsilon\in\{\pm\}$, let $\mathbf{S}^{\epsilon}_{\unl{k}}$ be the $\BZ[v,v^{-1}]$-submodule of
$\mathbf{S}_{\unl{k}}$ spanned by $\{\Psi(\tilde{\mathbf{E}}^{\epsilon}_{h})\}_{h\in H_{\unl{k}}}$. Then,
the following analogue of Lemma~\ref{span}~holds:

\begin{Prop}\label{spanig}
For any $F\in \mathbf{S}_{\unl{k}}$ and $\unl{d}\in\mathrm{KP}(\unl{k})$, if $\phi_{\unl{d}'}(F)=0$ for all
$\unl{d}'\in \mathrm{KP}(\unl{k})$ such that $\unl{d}'<\unl{d}$, then there exists
$F_{\unl{d}}\in \mathbf{S}^{\epsilon}_{\unl{k}}$ such that $\phi_{\unl{d}}(F)=\phi_{\unl{d}}(F_{\unl{d}})$
and $\phi_{\unl{d}'}(F_{\unl{d}})=0$ for all $\unl{d}'<\unl{d}$.
\end{Prop}

\begin{proof}
Combining \eqref{looplusg}, \eqref{ctildebeta}, and Proposition \ref{spekpG}, we obtain:
\begin{equation}
  \phi_{\unl{d}}(\Psi(\tilde{\mathbf{E}}^{\epsilon}_{h}))\doteq
  \prod_{\beta,\beta'\in \Delta^+}^{\beta<\beta'}G_{\beta,\beta'} \cdot
  \prod_{\beta\in\Delta^{+}}({\tilde{c}_{\beta}}^{d_{\beta}}\cdot G_{\beta}) \ \cdot
  \prod_{\beta\in\Delta^{+}}\mathbf{P}_{\lambda_{h,\beta}},
\end{equation}
where
\begin{equation}
   \mathbf{P}_{\lambda_{h,\beta}}\coloneqq \frac{1}{\prod_{r\in\BZ}[h(\beta,r)]_{v_{\beta}}!}\cdot
   {\mathop{Sym}}_{\mathfrak{S}_{d_{\beta}}}
   \left(w_{\beta,1}^{r_{\beta}(h,1)}\cdots w_{\beta,d_{\beta}}^{r_{\beta}(h,d_{\beta})}
        \prod_{1\leq i<j\leq d_{\beta}}\frac{w_{\beta,i}-v_{\beta}^{-2}w_{\beta,j}}{w_{\beta,i}-w_{\beta,j}}\right),
\end{equation}
cf.~\eqref{eq:lambda-collection}. As $\{\mathbf{P}_{\lambda_{h,\beta}}\}_{h\in H_{\unl{k},\unl{d}}}$ form a $\BZ[v,v^{-1}]$-basis
of $\BZ[v,v^{-1}][\{w_{\beta,s}^{\pm 1}\}_{s=1}^{d_{\beta}}]^{\mathfrak{S}_{d_{\beta}}}$
(this is a ``rank~$1$'' computation, cf. \cite[Proposition 1.4, Lemma 1.14]{Tsy22}), the claimed result
follows from Proposition~\ref{spanG} and conditions~(\ref{li1G},~\ref{li2G}).
\end{proof}

Combining Propositions \ref{lintegralG} and \ref{spanig}, we obtain the following upgrade of Theorem \ref{shufflePBWDG}:

\begin{Thm}\label{srig}
(a) The $\BQ(v)$-algebra isomorphism $\Psi\colon \qfg \,\iso\, S$  of Theorem {\rm \ref{shufflePBWDG}(a)} gives rise
to a $\BZ[v,v^{-1}]$-algebra isomorphism $\Psi\colon \integralgl \,\iso\, \mathbf{S}$.

\medskip
\noindent
(b) For any choices of  $s_1,s_2$ in \eqref{rvg2}--\eqref{rvg5} and $\epsilon\in \{\pm\}$,
the ordered monomials $\{\tilde{\mathbf{E}}^{\epsilon}_{h}\}_{h\in H}$ of~\eqref{eq:ordered-pbw-int-G}
form a basis of the free $\BZ[v,v^{-1}]$-module $\integralgl$.
\end{Thm}

   %%%%%%%%%%%%%%%%%%%%%%%%%%%%%%%%%%%%%%%%%%%%%%%%%%%%%%%%%%%%%%%%%%
   %%%%%%%%%%%%%%%%%%%%%%%%%%%%%%%%%%%%%%%%%%%%%%%%%%%%%%%%%%%%%%%%%%
   %%%%%%%%%%%%%%%%%%%%%%%%%%%%%%%%%%%%%%%%%%%%%%%%%%%%%%%%%%%%%%%%%%

\section{Specialization maps for type $B_{n}$}\label{type B}

In this section, we define specialization maps for the shuffle algebras of type $B_{n}$ and verify their key properties.
This implies the shuffle algebra realizations and PBWD-type theorems for $\qfb$, as well as for its two integral forms
$\integralb$ and $\integralbl$. Using arguments similar to those from Section~\ref{tG} we establish the counterparts of
Lemmas~\ref{shuffleelement},~\ref{vanish},~\ref{span} for $B_{2}$ case, and then use induction to treat the general $B_n$ case.

   %%%%%%%%%%%%%%%%%%%%%%%%%%%%%%%%%%%%%%%%%%%%%%%%%%%%%%%%%%%%%%%%%%
   %%%%%%%%%%%%%%%%%%%%%%%%%%%%%%%%%%%%%%%%%%%%%%%%%%%%%%%%%%%%%%%%%%
   %%%%%%%%%%%%%%%%%%%%%%%%%%%%%%%%%%%%%%%%%%%%%%%%%%%%%%%%%%%%%%%%%%

\subsection{$\qfb$ and its shuffle algebra realization}

In type $B_{n}$, for any $F\in S_{\unl{k}}$ with $\unl{k}\in\mathbb{N}^{n}$, the wheel conditions are:
\begin{equation}
\begin{aligned}
  F(\{x_{i,r}\}_{1\leq i\leq n}^{1\leq r\leq k_{i}})=0 \quad \text{once} \quad
  & x_{i,1}=v^{4}x_{i,2}=v^{2}x_{i+1,1} \quad \text{for some} \quad 1\leq i\leq n-1,\\
  \text{or} \quad
  & x_{i,1}=v^{4}x_{i,2}=v^{2}x_{i-1,1} \quad \text{for some} \quad 2\leq i\leq n-1,\\
  \text{or} \quad
  & x_{n,1}=v^{2}x_{n,2}=v^{4}x_{n,3}=v^{2}x_{n-1,1}.
\end{aligned}
\end{equation}

For any $\unl{k}\in\mathbb{N}^{n}$ and $\unl{d}\in\text{KP}(\unl{k})$, the specialization map $\phi_{\unl{d}}$
as in~\eqref{speentireshuffle} is defined by the following specialization of the $x^{(*,*)}_{*,*}$-variables
(replacing~\eqref{eq:spec-A} for type $A_n$):
\begin{equation}
  x^{(\beta,s)}_{i,1}\mapsto v^{-2i}w_{\beta,s}, \quad x^{(\beta,s)}_{i,2}\mapsto v^{-4n+2i+2}w_{\beta,s} \qquad
  \forall\ \beta\in\Delta^{+},\ 1\leq s\leq d_\beta,\ i\in\beta.
\label{spe}
\end{equation}

\begin{Lem}\label{phirvns}
Consider the particular choices~\eqref{rvb1}--\eqref{rvb2} of quantum root vectors
  $\{\tilde{E}^{\pm}_{\beta,s}\}_{\beta\in\Delta^{+}}^{s\in\BZ}$.
Their images under $\Psi$ in the shuffle algebra $S$ of type $B_{n}$  are as follows:
\begin{itemize}[leftmargin=0.7cm]

\item
If $\beta=[i,j]$, then for any decomposition $s=s_{i}+\cdots+s_{j}$ used in \eqref{rvb1}, we have:
\begin{align}
  & \Psi(\tilde{E}^{+}_{[i,j],s})\doteq
    \frac{{\langle 2\rangle_{v}}^{|\beta|-1} \cdot x_{i,1}^{s_{i}+1}\cdots x_{j-1,1}^{s_{j-1}+1}x_{j,1}^{s_{j}}}
         {\prod_{\ell=i}^{j-1}(x_{\ell,1}-x_{\ell+1,1})},\\
  & \Psi(\tilde{E}^{-}_{[i,j],s})\doteq
    \frac{{\langle 2\rangle_{v}}^{|\beta|-1} \cdot x_{i,1}^{s_{i}}x_{i+1,1}^{s_{i+1}+1}\cdots x_{j,1}^{s_{j}+1}}
         {\prod_{\ell=i}^{j-1}(x_{\ell,1}-x_{\ell+1,1})}.
\end{align}

\item
If $\beta=[i,n,j]$, then for any decomposition $s=s_{i}+\cdots +s_{j-1}+2s_{j}+\cdots +2s_{n}$ used in \eqref{rvb2}, we have:
\begin{align}
  & \Psi(\tilde{E}^{+}_{[i,n,j],s})\doteq
    \frac{{\langle 2\rangle_{v}}^{|\beta|-1} \cdot g_{1}\cdot
          \prod_{\ell=j}^{n-1}(v^{4}x_{\ell,1}-x_{\ell,2})(v^{4}x_{\ell,2}-x_{\ell,1})}
         {(x_{i,1}-x_{i+1,1})\cdots (x_{j-1,1}-x_{j,1})(x_{j-1,1}-x_{j,2})\prod_{\ell=j}^{n-1}
          \prod_{1\leq r,t\leq 2}(x_{\ell,r}-x_{\ell+1,t})},\\
  & \Psi(\tilde{E}^{-}_{[i,n,j],s})\doteq
    \frac{{\langle 2\rangle_{v}}^{|\beta|-1} \cdot g_{2}\cdot
          \prod_{\ell=j}^{n-1}(v^{4}x_{\ell,1}-x_{\ell,2})(v^{4}x_{\ell,2}-x_{\ell,1})}
         {(x_{i,1}-x_{i+1,1})\cdots (x_{j-1,1}-x_{j,1})(x_{j-1,1}-x_{j,2})\prod_{\ell=j}^{n-1}
          \prod_{1\leq r,t\leq 2}(x_{\ell,r}-x_{\ell+1,t})},
\end{align}
where
\begin{equation}
\begin{aligned}
  & g_{1}=\prod_{\ell=i}^{j-2}x_{\ell,1}^{s_{\ell}+1} \cdot x_{j-1,1}^{s_{j-1}+2}(x_{j,1}x_{j,2})^{s_{j}} \cdot
     \prod_{\ell=j+1}^{n}(x_{\ell,1}x_{\ell,2})^{s_{\ell}+1},\\
  & g_{2}=x_{i,1}^{s_{i}}\prod_{\ell=i+1}^{j-1} x_{\ell,1}^{s_{\ell}+1}
    \prod_{\ell=j}^{n}(x_{\ell,1}x_{\ell,2})^{s_{\ell}+1}.
\end{aligned}
\end{equation}
\end{itemize}
\end{Lem}

\begin{proof}
Straightforward computation.
\end{proof}

For more general quantum root vectors $\{E_{\beta,s}\}_{\beta\in\Delta^{+}}^{s\in\BZ}$ defined by \eqref{rootvector1},
we have  the following counterpart of Lemma~\ref{Gs1}:

\begin{Lem}\label{phirv}
For any choices of $s_k$ and $\lambda_k$ in \eqref{rootvector1}, we have:
\begin{equation}\label{eq:B-general}
  \phi_{\beta}(\Psi(E_{\beta,s}))\doteq c_{\beta}\cdot w_{\beta,1}^{s+\kappa_{\beta}}
  \qquad \forall\ (\beta,s)\in\Delta^{+}\times \BZ,
\end{equation}
where $\{\kappa_{\beta}\}_{\beta\in\Delta^{+}}$ are explicitly given by
\begin{equation}
  \kappa_{\beta}=
  \begin{cases}
    |\beta|-1 &\quad\text{if}\  \beta=[i,j]\\
    |\beta|+2(n-j)-1 &\quad\text{if}\  \beta=[i,n,j]
  \end{cases}
\label{kappaB}
\end{equation}
and the constants $\{c_{\beta}\}_{\beta\in \Delta^{+}}$ are explicitly given by
\begin{equation}\label{eq:c-factor-B}
  c_{\beta}=
  \begin{cases}
    {\langle 2\rangle_{v}}^{|\beta|-1} &\quad\text{if}\  \beta=[i,j]\\
    {\langle 2\rangle_{v}}^{|\beta|-1} \cdot \prod_{\ell=j}^{n-1} \big\{(v^{-4n+4\ell-2}-1)(v^{-4n+4\ell+6}-1)\big\}
    &\quad\text{if}\  \beta=[i,n,j]
  \end{cases}.
\end{equation}
\end{Lem}

\begin{proof}
The proof is by induction on $n$, the base case $n=2$ being obvious. Let us now assume that the result holds for
types $B_{m}\ (m<n)$. In what follows, we shall use the following special case of Proposition {\rm \ref{vanishB}}:
for any positive roots $\alpha_{1}<\alpha_{2}$ such that $\alpha_{1}+\alpha_{2}$ is a root, we have
  \[\phi_{\alpha_{1}+\alpha_{2}}(\Psi(E_{\alpha_{1},s_{1}})\star \Psi(E_{\alpha_{2},s_{2}}))=0
    \qquad \forall\ s_{1},s_{2}\in\BZ. \]

If $\beta=[i,j]$ with $1\leq i\leq j\leq n$, then \eqref{eq:B-general} follows from the $A_{n}$ type, see
\cite[Lemma 3.14]{Tsy18}. If $\beta=[i,n,j]$ with $1<i<j\leq n$, then \eqref{eq:B-general} likewise follows from
the $B_{n-1}$ case. Therefore, it remains to treat the cases $\beta=[1,n,j]$ with $2\leq j\leq n$.

If $\beta=[1,n,n]$, then
  \[E_{\beta,s}=
    [[[\cdots[[e_{1,s_{1}},e_{2,s_{2}}]_{\lambda_{1}},e_{3,s_{3}}]_{\lambda_{2}},\cdots,
    e_{n-1,s_{n-1}}]_{\lambda_{n-2}},e_{n,s_{n}}]_{\lambda_{n-1}},e_{n,s_{n+1}}]_{\lambda_{n}}\]
with $s=s_{1}+\cdots+s_{n+1}$. Consider $\alpha=[1,n]$, $r=s_{1}+\cdots+s_{n-1}+s_{n}$, and
  \[E_{\alpha,r}=
    [[\cdots[[e_{1,s_{1}},e_{2,s_{2}}]_{\lambda_{1}},e_{3,s_{3}}]_{\lambda_{2}},\cdots,
    e_{n-1,s_{n-1}}]_{\lambda_{n-2}},e_{n,s_{n}}]_{\lambda_{n-1}}.\]
Then, we have:
\begin{equation}\label{eq:simplification-1}
\begin{aligned}
  \phi_{\beta}(\Psi(E_{\beta,s}))
  & = \phi_{\beta}(\Psi(E_{\alpha,r})\star\Psi(e_{n,s_{n+1}})-\lambda_{n}\Psi(e_{n,s_{n+1}})\star\Psi(E_{\alpha,r}))\\
  & \doteq \phi_{\beta}(\Psi(e_{n,s_{n+1}})\star\Psi(E_{\alpha,r}))\\
  & \doteq \phi_{\beta}\Big({\mathop{Sym}}_{\mathfrak{S}_2}
    \Big(\Psi(E_{\alpha,r})(x^{(\beta,1)}_{n,2})^{s_{n+1}}
    \zeta_{n,n-1}(x^{(\beta,1)}_{n,2}/x^{(\beta,1)}_{n-1,1})\zeta_{n,n}(x^{(\beta,1)}_{n,2}/x^{(\beta,1)}_{n,1})\Big)\Big)\\
  & \doteq \phi_{\beta}\left(\Psi(E_{\alpha,r})(x^{(\beta,1)}_{n,2})^{s_{n+1}}
    \frac{(x^{(\beta,1)}_{n,2}-v^{2}x^{(\beta,1)}_{n-1,1})(x^{(\beta,1)}_{n,2}-v^{-2}x^{(\beta,1)}_{n,1})}
         {(x^{(\beta,1)}_{n,2}-x^{(\beta,1)}_{n-1,1})(x^{(\beta,1)}_{n,1}-x^{(\beta,1)}_{n,2})}\right)\\
  & \doteq {\langle 2\rangle_{v}}^{|\beta|-1} \cdot w_{\beta,1}^{s+|\beta|-1},
\end{aligned}
\end{equation}
where we plug the variables $\{x^{(\beta,1)}_{i,1}\}_{1\leq i\leq n}$ into $\Psi(E_{\alpha,r})$,
the variable $x^{(\beta,1)}_{n,2}$ into $\Psi(e_{n,s_{n+1}})$, $\mathop{Sym}_{\mathfrak{S}_2}$ denotes the symmetrization
with respect to the variables $\{x^{(\beta,1)}_{n,1}, x^{(\beta,1)}_{n,2}\}$, and the last line follows by applying
the validity of~\eqref{eq:B-general} for $(\alpha,r)$ established above. We also note that the second equality
in~\eqref{eq:simplification-1} used the vanishing $\phi_{\beta}(\Psi(E_{\alpha,r})\star\Psi(e_{n,s_{n+1}}))=0$, due to
\begin{equation}\label{eq:simple-vanishing-1}
  \phi_{\beta}(\zeta_{n,n}(x^{(\beta,1)}_{n,1}/x^{(\beta,1)}_{n,2}))=0=\phi_{\beta}(\zeta_{n-1,n}(x^{(\beta,1)}_{n-1,1}/x^{(\beta,1)}_{n,1})) \,.
\end{equation}

If $\beta=[1,n,j]$ with $2\leq j\leq n-1$, then
  \[E_{\beta,s}=
    [[\cdots[[\cdots[e_{1,s_{1}},e_{2,s_{2}}]_{\lambda_{1}},\cdots, e_{n,s_{n}}]_{\lambda_{n-1}},
    e_{n,s_{n+1}}]_{\lambda_{n}},\cdots,e_{j+1,s_{2n-j}}]_{\lambda_{2n-j-1}},e_{j,s_{2n-j+1}}]_{\lambda_{2n-j}}\]
with $s=s_{1}+\cdots +s_{2n-j+1}$. Consider $\alpha=[1,n,j+1]$, $r=s_{1}+\cdots+s_{2n-j}$, and
  \[E_{\alpha,r}=[\cdots[[\cdots[e_{1,s_{1}},e_{2,s_{2}}]_{\lambda_{1}},\cdots,
    e_{n,s_{n}}]_{\lambda_{n-1}},e_{n,s_{n+1}}]_{\lambda_{n}},\cdots,e_{j+1,s_{2n-j}}]_{\lambda_{2n-j-1}}.\]
Note that
  $\phi_{\beta}(\zeta_{j+1,j}(x^{(\beta,1)}_{j+1,2}/x^{(\beta,1)}_{j,2}))=0=\phi_{\beta}(\zeta_{j,j+1}(x^{(\beta,1)}_{j,1}/x^{(\beta,1)}_{j+1,1}))$,
cf.~\eqref{eq:simple-vanishing-1}. Thus, we have:
\begin{equation}\label{eq:simplification-2}
\begin{aligned}
  \phi_{\beta}(\Psi(E_{\beta,s}))
  & = \phi_{\beta}(\Psi(E_{\alpha,r})\star \Psi(e_{j,s_{2n-j+1}})-\lambda_{2n-j}\Psi(e_{j,s_{2n-j+1}})\star\Psi(E_{\alpha,r}))\\
  & \doteq \phi_{\beta}(\Psi(e_{j,s_{2n-j+1}})\star\Psi(E_{\alpha,r}))\\
  & \doteq \phi_{\beta}\Big( {\mathop{Sym}}_{\mathfrak{S}_2}
    \Big(\Psi(E_{\alpha,r})(x^{(\beta,1)}_{j,2})^{s_{2n-j+1}}
    \zeta_{j,j}(x^{(\beta,1)}_{j,2}/x^{(\beta,1)}_{j,1})\times \\
  & \qquad \ \ \ \ \zeta_{j,j-1}(x^{(\beta,1)}_{j,2}/x^{(\beta,1)}_{j-1,1})
    \zeta_{j,j+1}(x^{(\beta,1)}_{j,2}/x^{(\beta,1)}_{j+1,1})\zeta_{j,j+1}
    (x^{(\beta,1)}_{j,2}/x^{(\beta,1)}_{j+1,2})\Big)\Big)\\
  & \doteq \phi_{\beta}\left(
    \frac{\Psi(E_{\alpha,r})(x^{(\beta,1)}_{j,2})^{s_{2n-j+1}}(x^{(\beta,1)}_{j,2}-v^{-4}x^{(\beta,1)}_{j,1})}
         {x^{(\beta,1)}_{j,1}-x^{(\beta,1)}_{j,2}}\times\right. \\
  & \qquad\ \ \ \
    \left.
    \frac{(x^{(\beta,1)}_{j,2}-v^{2}x^{(\beta,1)}_{j-1,1})(x^{(\beta,1)}_{j,2}-v^{2}x^{(\beta,1)}_{j+1,1})
          (x^{(\beta,1)}_{j,2}-v^{2}x^{(\beta,1)}_{j+1,2})}
         {(x^{(\beta,1)}_{j,2}-x^{(\beta,1)}_{j-1,1})(x^{(\beta,1)}_{j,2}-x^{(\beta,1)}_{j+1,1})
          (x^{(\beta,1)}_{j,2}-x^{(\beta,1)}_{j+1,2})}\right)\\
  & \doteq {\langle 2\rangle_{v}}^{|\beta|-1} \cdot
    \prod_{\ell=j}^{n-1}\big\{(v^{-4n+4\ell-2}-1)(v^{-4n+4\ell+6}-1)\big\} \cdot w_{\beta,1}^{s+|\beta|+2(n-j)-1},
\end{aligned}
\end{equation}
where we plug the variables $\{x^{(\beta,1)}_{i,t}, x^{(\beta,1)}_{j,1}\}^{1\leq t\leq \nu_{\beta,i}}_{1\leq i\leq n, i\neq j}$
into $\Psi(E_{\alpha,r})$, the variable $x^{(\beta,1)}_{j,2}$ into $\Psi(e_{j,s_{2n-j+1}})$, $\mathop{Sym}_{\mathfrak{S}_2}$
denotes the symmetrization with respect to the variables $\{x^{(\beta,1)}_{j,1}, x^{(\beta,1)}_{j,2}\}$, and the last line
follows by applying the induction hypothesis for $\phi_{\alpha}(\Psi(E_{\alpha,r}))$.
\end{proof}

Let us now generalize the above lemma by computing $\phi_{\unl{d}}(\Psi(E_{h}))$ for any $h\in H_{\unl{k},\unl{d}}$.
Similarly to type $G_{2}$, we choose a special splitting such that the variables in $\Psi(E_{\beta,r_{\beta}(h,s)})$ are
taken to be the group $\{x^{(\beta,s)}_{i,t}\}^{1\leq t\leq \nu_{\beta,i}}_{1\leq i\leq n}$, and under $\phi_{\unl{d}}$
they are specialized as in \eqref{spe}. For each $1\leq i\leq n$, we define the set $X_{i}$ as in \eqref{variablei},
and the total order on $X_{i}$ as in \eqref{ordervariable}.

For any $\unl{d}\in\text{KP}(\unl{k})$, we define the subset $\text{Sh}_{\unl{d}}\subset \mathfrak{S}_{\unl{k}}$
of ``$\unl{d}$-shuffle permutations'' as in \eqref{eq:d-perm-G}:
\begin{equation}\label{unldshuffle}
  \text{Sh}_{\unl{d}}=
  \Big\{\sigma\in \mathfrak{S}_{\unl{k}} \, \Big| \, \sigma(x^{(\beta,s)}_{i,1})<\sigma(x^{(\beta,s)}_{i,2})
        \quad \forall\ \beta\in\Delta^{+},1\leq s\leq d_{\beta},2\leq i\leq n\Big\}.
\end{equation}
Then from \eqref{shuffleproduct}, we have:
\begin{equation}
  \Psi(E_{h})\doteq
  \sum_{\sigma\in\text{Sh}_{\unl{d}}} \sigma\big(F_{h}(\{x^{(*,*)}_{*,*}\})\big)=
  \sum_{\sigma\in\text{Sh}_{\unl{d}}} F_{h}\big(\{\sigma(x^{(*,*)}_{*,*})\}\big),
\end{equation}
where
\begin{equation}
  F_{h}\coloneqq
  \prod_{\substack{\beta\in\Delta^{+}\\1\leq s\leq d_{\beta}}}\Psi(E_{\beta,r_{\beta}(h,s)})
  \prod^{(\alpha,p)<(\beta,q)}_{\substack{\alpha,\beta\in\Delta^{+}\\1\leq p\leq d_{\alpha},1\leq q\leq d_{\beta}}}
  \prod_{1\leq i,j\leq n}^{a_{ij\neq 0}}\prod_{1\leq \ell\leq \nu_{\alpha,i}}^{1\leq r\leq \nu_{\beta,j}}
  \frac{x^{(\alpha,p)}_{i,\ell}-v_{i}^{-a_{ij}}x^{(\beta,q)}_{j,r}}{x^{(\alpha,p)}_{i,\ell}-x^{(\beta,q)}_{j,r}}.
\end{equation}

%%%%%%%%%%%%%%%%%%%%%%%%%%%%%%%%%%%%%%%%%%%%%%%%%%%%%%%%%%%%%%%%%%%%%%%%%%%%%%%%%%%%%%%%%%%%
%%%%%%%%%%%%%%%%%%%%%%%%%%%% Yue's original text over the field %%%%%%%%%%%%%%%%%%%%%%%%%%%%
%%%%%%%%%%%%%%%%%%%%%%%%%%%%%%%%%%%%%%%%%%%%%%%%%%%%%%%%%%%%%%%%%%%%%%%%%%%%%%%%%%%%%%%%%%%%

Let us consider the elements of $\fS_{\unl{k}}$ satisfying \eqref{fsunld}, which form a subgroup isomorphic to
$\fS_{\unl{d}}$ (we shall denote this subgroup by $\fS_{\unl{d}}$). Then, similarly to Lemma \ref{Gs2}, we have:

\begin{Lem}\label{bstep3}
$\phi_{\unl{d}}(\sigma(F_{h}))=0$ for $\sigma\notin \fS_{\unl{d}}$.
\end{Lem}

\begin{proof}
The proof is by induction on $n$. Arguing alike in the proof of Lemma \ref{Gs2}, the result is clear for $B_{2}$-type.
Let $\Delta^{+}_{1}=\{[1,j]\}_{1\leq j\leq n}\cup\{[1,n,j]\}_{1<j\leq n}$. It suffices to show that
$\phi_{\unl{d}}(\sigma(F_{h}))\neq 0$ only if \eqref{fsunld} holds for every $\beta\in\Delta_{1}^{+}$, as for
$\beta>[1,n,2]$ we can apply the induction hypothesis for $B_{n-1}$-type. We shall now prove that \eqref{fsunld}
holds for any $x^{(\beta,s_{1})}_{*,*}$, assuming it holds for any $x^{(\beta',s')}_{*,*}$ with
$(\beta',s')<(\beta,s_{1})$. Similar to \eqref{Zbetas}, we define the following sets of $x^{(*,*)}_{*,*}$-variables:
\begin{equation}
  Z^{>(\beta,s)}_{i}=\big\{x^{(\alpha,r)}_{i,*} \, \big| \, (\alpha,r)>(\beta,s) \big\}
  \qquad \forall\ (\beta,s)\in\Delta^{+}\times\BN, 1\leq i\leq n.
\end{equation}
\begin{itemize}[leftmargin=0.7cm]

\item
Case 1: $\beta=[1,j]$ with $1\leq j\leq n-1$. Let $\sigma(x^{(\beta,s_{1})}_{1,1})=x^{(\gamma,r)}_{1,1}$ with
$\gamma\geq \beta$. Since $F_{h}$ contains the factor $\zeta_{1,2}(x^{(\beta,s_{1})}_{1,1}/Z^{>(\beta,s_{1})}_{2})$,
we have $\phi_{\unl{d}}(\sigma(F_{h}))=0$ unless $\sigma(x^{(\beta,s_{1})}_{2,1})=x^{(\gamma,r)}_{2,1}$. Proceeding
as this, we get  $\phi_{\unl{d}}(\sigma(F_{h}))=0$ unless $\sigma(x^{(\beta,s_{1})}_{\ell,1})=x^{(\gamma,r)}_{\ell,1}$
for any $\ell\leq j$. If $\gamma>\beta$, then we get $\phi_{\unl{d}}(\sigma(F_{h}))=0$, since $F_{h}$ contains
the factor $\zeta_{j,j+1}(x^{(\beta,s_{1})}_{j,1}/Z^{>(\beta,s_{1})}_{j+1})$ and
$x^{(\gamma,r)}_{j+1,1}\in \sigma(Z^{>(\beta,s_{1})}_{j+1})$, hence, a contradiction.
Thus $\gamma=\beta$, and \eqref{fsunld} holds for any $x^{(\beta,s_{1})}_{*,*}$.

\item
Case 2: $\beta=[1,n]$. Let $\sigma(x^{(\beta,s_{1})}_{1,1})=x^{(\gamma,r)}_{1,1}$ with $\gamma\geq \beta$.
Then, similar to above, we get $\phi_{\unl{d}}(\sigma(F_{h}))=0$ unless
$\sigma(x^{(\beta,s_{1})}_{\ell,1})=x^{(\gamma,r)}_{\ell,1}$ for any $\ell\leq n$. If $\gamma>\beta$, then
we get $\phi_{\unl{d}}(\sigma(F_{h}))=0$, since $F_{h}$ contains the factor
$\zeta_{n,n}(x^{(\beta,s_{1})}_{n,1}/Z^{>(\beta,s_{1})}_{n})$ and $x^{(\gamma,r)}_{n,2}\in \sigma(Z^{>(\beta,s_{1})}_{n})$,
a contradiction. Thus $\gamma=\beta$, and \eqref{fsunld} holds for any $x^{(\beta,s_{1})}_{*,*}$.

\item
Case 3: $\beta=[1,n,j]$ with $2\leq j\leq n$. Let $\sigma(x^{(\beta,s_{1})}_{1,1})=x^{(\gamma,r)}_{1,1}$
with $\gamma\geq \beta$. Then, $\phi_{\unl{d}}(\sigma(F_{h}))=0$ unless there exist
$\{t_{\ell,1},t_{\ell,2}\}_{\ell=j}^{n}$ such that
\begin{equation}
\begin{aligned}
  & \sigma(x^{(\beta,s_{1})}_{\ell,1})=x^{(\gamma,r)}_{\ell,1} \qquad \text{if}\  1\leq \ell\leq j-1, \\
  & \sigma(x^{(\beta,s_{1})}_{\ell,t_{\ell,1}})=x^{(\gamma,r)}_{\ell,1}, \quad
    \sigma(x^{(\beta,s_{1})}_{\ell,t_{\ell,2}})=x^{(\gamma,r)}_{\ell,2} \qquad \text{if}\ j\leq \ell\leq n.
\end{aligned}
\end{equation}
If $t_{\ell,2}<t_{\ell,1}$ for some $\ell\in \{j,j+1,\ldots,n\}$, then
$x^{(\beta,s_{1})}_{\ell,t_{\ell,2}}<x^{(\beta,s_{1})}_{\ell,t_{\ell,1}}$, and the condition
$\sigma\in\text{Sh}_{\unl{d}}$ implies
  \[x^{(\gamma,r)}_{\ell,1}=\sigma(x^{(\beta,s_{1})}_{\ell,t_{\ell,1}})
    <\sigma(x^{(\beta,s_{1})}_{\ell,t_{\ell,2}})=x^{(\gamma,r)}_{\ell,2},\]
a contradiction. Hence, $t_{\ell,1}<t_{\ell,2}$ for any $j\leq \ell\leq n$, so that $t_{\ell,1}=1, t_{\ell,2}=2$
for any $j\leq \ell\leq n$. If $\gamma>\beta$, then we get $\phi_{\unl{d}}(\sigma(F_{h}))=0$, since $F_{h}$
contains the factor $\zeta_{j,j-1}(x^{(\beta,s_{1})}_{j,2}/Z^{>(\beta,s_{1})}_{j-1})$ and
$x^{(\gamma,r)}_{j-1,2}\in \sigma(Z^{>(\beta,s_{1})}_{j-1})$, a contradiction.
Thus $\gamma=\beta$, and \eqref{fsunld} holds for any $x^{(\beta,s_{1})}_{*,*}$.

\end{itemize}
This completes our proof.
\end{proof}

Combining Lemmas \ref{phirv} and \ref{bstep3}, we obtain the following analogue of Lemma \ref{shuffleelement}
for type~$B_{n}$:

\begin{Prop}\label{spekpB}
For any $h\in H_{\unl{k},\unl{d}}$, we have
\begin{equation}
  \phi_{\unl{d}}(\Psi(E_{h}))\doteq
  \prod_{\beta,\beta'\in \Delta^+}^{\beta<\beta'}G_{\beta,\beta'}\cdot
  \prod_{\beta\in\Delta^{+}}(c_{\beta}^{d_{\beta}}\cdot G_{\beta})\cdot
  \prod_{\beta\in\Delta^{+}} P_{\lambda_{h,\beta}},\label{speehb}
\end{equation}
where the factors $\{P_{\lambda_{h,\beta}}\}_{\beta\in\Delta^{+}}$ are given by \eqref{hlp}, the constants
$\{c_{\beta}\}_{\beta\in\Delta^{+}}$ are as in Lemma~{\rm \ref{phirv}}, and the terms $G_{\beta,\beta'},G_{\beta}$
are products of linear factors $w_{\beta,s}$ and $w_{\beta,s}-v^{\BZ}w_{\beta',s'}$ which are independent of
$h\in H_{\unl{k},\unl{d}}$ and are $\mathfrak{S}_{\unl{d}}$-symmetric (the factors $G_{\beta}$ are explicitly
computed in~\eqref{eqformulagbeta}).
\end{Prop}

\begin{Prop}\label{vanishB}
Lemma {\rm \ref{vanish}} is valid for type $B_{n}$ and specialization maps $\phi_{\unl{d}}$ as in \eqref{spe}.
\end{Prop}

\begin{proof}
The proof is by induction on $n$ with the base case of $B_{2}$-type being clear, cf.~the proof of Lemma \ref{bstep3}.
Given $\unl{d},\unl{d}'\in \mathrm{KP}(\unl{k})$ with $\unl{d}'<\unl{d}$, let $\beta\in \Delta^+$ be the smallest root
such that $d'_{\beta}<d_{\beta}$. We can assume that $\beta\in\Delta^{+}_{1}$, as otherwise the induction hypothesis for
$B_{n-1}$ type will apply.
Similarly to the proof of Proposition \ref{vanishG}, without loss of generality, we can assume $d'_{\alpha}=0$ for all
$\alpha\leq \beta$. Let $x^{(*,*)}_{*,*}$ be the above special splitting of the variables for $\phi_{\unl{d}}$, and
$x'^{(*,*)}_{*,*}$ be any splitting of the variables for $\phi_{\unl{d}'}$. Then, we have the following case-by-case analysis:
\begin{itemize}[leftmargin=0.7cm]

\item
Case 1: $\beta=[1,j]$ with $1\leq j\leq n-1$. Since $0=d'_{\beta}<d_{\beta}$, we have
$\sigma(x^{(\beta,1)}_{1,1})=x'^{(\gamma,r)}_{1,1}$ for some $\gamma>\beta$. Since $F_{h}$ contains the $\zeta$-factor
$\zeta_{1,2}(x^{(\beta,1)}_{1,1}/x^{(\alpha,s)}_{2,t})$ for any $(\alpha,s)>(\beta,1)$ and $1\leq t\leq 2$, we have
$\phi_{\unl{d}'}(\sigma(F_{h}))=0$ unless $\sigma(x^{(\beta,1)}_{2,1})=x'^{(\gamma,r)}_{2,1}$. Proceeding as this, we get
$\phi_{\unl{d}'}(\sigma(F_{h}))=0$ unless $\sigma(x^{(\beta,1)}_{i,1})=x'^{(\gamma,r)}_{i,1}$ for any $1\leq i\leq j$.
In the latter case, we still obtain  $\phi_{\unl{d}'}(\sigma(F_{h}))= 0$, since $F_{h}$ contains the factor
$\zeta_{j,j+1}(x^{(\beta,1)}_{j,1}/x'^{(\gamma,r)}_{j+1,1})$.

\item
Case 2: $\beta=[1,n]$. Let $\sigma(x^{(\beta,1)}_{1,1})=x'^{(\gamma,r)}_{1,1}$ for some $\gamma>\beta$. Arguing as above,
we have $\phi_{\unl{d}'}(\sigma(F_{h}))=0$ unless $\sigma(x^{(\beta,1)}_{i,1})=x'^{(\gamma,r)}_{i,1}$ for any $1\leq i\leq n$.
If $(\beta',s')$ is such that $\sigma(x^{(\beta',s')}_{n,t})=x'^{(\gamma,r)}_{n,2}$, then we have $(\beta',s')>(\beta,1)$.
Therefore, we still get $\phi_{\unl{d}'}(\sigma(F_{h}))= 0$, since $F_{h}$ contains the factor
$\zeta_{n,n}(x^{(\beta,1)}_{n,1}/x^{(\beta',s')}_{n,t})$.

\item
Case 3: $\beta=[1,n,j]$ with $2<j\leq n$. Similar to above, $\phi_{\unl{d}'}(\sigma(F_{h}))= 0$ unless there is
$\gamma>\beta$ such that $\sigma(x^{(\beta,1)}_{i,1})=x^{(\gamma,r)}_{i,1}$ for $1\leq i\leq n$,
$\sigma(x^{(\beta,1)}_{i,2})=x^{(\gamma,r)}_{i,2}$ for $j\leq i\leq n$. In the latter case,
we again obtain $\phi_{\unl{d}'}(\sigma(F_{h}))=0$, since $F_{h}$ contains the factor
$\zeta_{j,j-1}(x^{(\beta,1)}_{j,2}/x'^{(\gamma,r)}_{j-1,2})$.

\item
Case 4: $\beta=[1,n,2]$. This case is impossible as $0=d'_\alpha=d_\alpha$ for $\alpha<\beta$ and $0=d'_\beta<d_\beta$
(indeed, $k_1=\sum_{\alpha\leq \beta} d'_\alpha<\sum_{\alpha\leq \beta}d_\alpha=k_1$ as
$\unl{d},\unl{d}'\in\mathrm{KP}(\unl{k})$, a contradiction).

\end{itemize}
This completes our proof.
\end{proof}

\begin{Prop}\label{spanB}
Lemma {\rm \ref{span}}  is valid for type $B_{n}$ and specialization maps $\phi_{\unl{d}}$ as in \eqref{spe}.
\end{Prop}

\begin{proof}
Similarly to the proof of Proposition \ref{spanG}, we only need to prove the lemma to be true for any pair of roots
$(\beta\leq \beta')$. First let $\unl{d}_{1}=\{d_{\beta}=2,d_{\alpha}=0,\forall\, \alpha\neq\beta\}\in\text{KP}(\unl{k}_{1})$
and $F_{1}\in S_{\unl{k}_{1}}$. We can assume $\beta=[1,n,j]$ for some $2\leq j\leq n$. If $j>2$, let $\gamma=[2,n,j]$,
and $\unl{d}_{3}=\{d_{\gamma}=2,d_{[1]}=2,d_{\alpha}=0,\forall\, \alpha\neq\gamma,[1]\}$. By induction we know $\phi_{\unl{d}_{3}}(F_{1})$
has the factor $G_{\gamma}$ (here we change the variable $w_{\gamma,s}$ in $G_{\gamma}$ by $w_{\beta,s}$). Moreover, we have
\begin{equation}
  G_{\beta}=(w_{\beta,1}-v^{-4}w_{\beta,2})(w_{\beta,1}-v^{4}w_{\beta,2})\cdot G_{\gamma}. 
\label{trivial}
\end{equation}
The wheel condition $F_{1}=0$ once $x^{(\beta,1)}_{1,1}=v^{4}x^{(\beta,2)}_{1,1}=v^{2}x^{(\beta,1)}_{2,1}$ or
$x^{(\beta,2)}_{1,1}=v^{4}x^{(\beta,1)}_{1,1}=v^{2}x^{(\beta,2)}_{2,1}$ becomes $\phi_{\unl{d}_{1}}(F_{1})=0$ once
$w_{\beta,1}=v^{4}w_{\beta,2}$ or $w_{\beta,1}=v^{-4}w_{\beta,2}$, thus giving us the required vanishing factors. 
If $j=2$, let $\gamma=[1,n,3]$ (here we assume $n>3$), then
\begin{equation}
  G_{\beta}=
  (w_{\beta,1}-v^{\pm (4n-6)}w_{\beta,2})(w_{\beta,1}-v^{\pm (4n-14)}w_{\beta,2})(w_{\beta,1}-v^{\pm 4}w_{\beta,2})
  \cdot G_{\gamma}.
\label{trivial1}
\end{equation}
The wheel condition $F_{1}=0$ once $x^{(\beta,1)}_{2,2}=v^{4}x^{(\beta,2)}_{2,1}=v^{2}x^{(\beta,2)}_{1,1}$, or
$x^{(\beta,1)}_{2,1}=v^{4}x^{(\beta,2)}_{2,2}=v^{2}x^{(\beta,1)}_{3,1}$, or
$x^{(\beta,1)}_{2,2}=v^{4}x^{(\beta,2)}_{2,2}=v^{2}x^{(\beta,2)}_{3,2}$ give us the vanishing factors above.
Thus the lemma is true for pairs $(\beta,\beta)$.

Now for any pair $(\beta<\beta')$, let
  $\unl{d}_{2}=\{d_{\beta}=d_{\beta'}=1, d_{\alpha\neq \beta,\beta'}=0\}\in\text{KP}(\unl{k}_{2})$
and $F_{2}\in S_{\unl{k}_{2}}$,  we prove $\phi_{\unl{d}_{2}}(F_{2})$ has the vanishing factor $G_{\beta,\beta'}$
if $\phi_{\unl{d}}(F_{2})=0$ for any $\unl{d}<\unl{d}_{2}$ . By induction we know it is true for any $(\beta<\beta')$
with $[2]\leq \beta$. And by results on type $A_{n}$, we know it is true for any $(\beta<\beta')$ with
$\beta=[1,j],\beta'=[i',j']$ for some $1\leq i',j,j'\leq n$.

\begin{itemize}[leftmargin=0.7cm]

\item
$\beta=[1,j]$ for some $1\leq j\leq n$, $\beta'=[i,n,n]$ for some $1\leq i\leq n-1$.

Let $\gamma=[i,n]$. If $j<n-1$ then $G_{\beta,\beta'}=G_{\beta,\gamma}$ and $\phi_{\unl{d}_{2}}(F_{2})$
has the factor $G_{\beta,\beta'}$ for the same reason. If $j=n-1$, then
\begin{equation}
  G_{\beta,\beta'}=(w_{\beta,1}-v^{2}w_{\beta',1})\cdot G_{\beta,\gamma},
\end{equation}
and by induction we know $\phi_{\unl{d}_{2}}(F_{2})$ has the factor $G_{\beta,\gamma}$.
Set $\unl{d}_{3}=\{d_{[1,n]}=2, d_{\alpha\neq [1,n]}=0\}$ if $i=1$, or
$\unl{d}_{3}=\{d_{[1,n,n]}=1, d_{[i,n-1]}=1, d_{\alpha\neq [1,n,n],[i,n-1]}=0\}$ if $i\neq 1$.
Then for each case, we have
$\unl{d}_{3}<\unl{d}_{2}$ and $\phi_{\unl{d}_{3}}(F_{2})=0$ implies that $\phi_{\unl{d}_{2}}(F_{2})=0$ once
$w_{\beta,1}=v^{2}w_{\beta',1}$. If $j=n$ and $i=1$, then
\begin{equation}
  G_{\beta,\beta'}=(w_{\beta,1}-v^{-2}w_{\beta',1})\cdot G_{\beta,\beta},
\end{equation}
where $G_{\beta,\beta}$ is obtained by changing the variable $w_{\beta,2}$ in $G_{\beta}$ by $w_{\beta',1}$.
By induction we know $\phi_{\unl{d}_{2}}(F_{2})$ has the factor $G_{\beta,\beta}$. Also the wheel condition $F_{2}=0$
once $x^{(\beta',1)}_{n,2}=v^{2}x^{(\beta',1)}_{n,1}=v^{4}x^{(\beta,1)}_{n,1}=v^{2}x^{(\beta,1)}_{n-1,1}$ becomes
$\phi_{\unl{d}_{2}}(F_{2})=0$ once $w_{\beta,1}=v^{-2}w_{\beta',1}$. If $j=n$ and $i>1$, then
\begin{equation}
  G_{\beta,\beta'}=(w_{\beta,1}-w_{\beta',1})\cdot G_{\beta,\gamma}.
\end{equation}
Let $\unl{d}_{3}=\{d_{[1,n,n]}=1, d_{[i,n]}=1, d_{\alpha\neq [1,n,n],[i,n]}=0\}$, then $\unl{d}_{3}<\unl{d}_{2}$ and
$\phi_{\unl{d}_{3}}(F_{2})=0$ implies that $\phi_{\unl{d}_{2}}(F_{2})=0$ once $w_{\beta,1}=w_{\beta',1}$.

\item
$\beta=[1,j]$ for some $1\leq j\leq n$, $\beta'=[i,n,\ell]$ for some $1\leq i<\ell\leq n$.

We already prove $\ell=n$ case, so we can use induction. If $j\leq \ell-2$, then $G_{\beta,\beta'}=G_{\beta,[i,\ell-1]}$
and $\phi_{\unl{d}_{2}}(F_{2})$ has the factor $G_{\beta,\beta'}$ for the same reason. If $j\geq \ell+1$, then
\begin{equation}
  G_{\beta,\beta'}=(w_{\beta,1}-v^{(4\ell-4n-2)}w_{\beta',1})(w_{\beta,1}-v^{(4\ell-4n+6)}w_{\beta',1})\cdot G_{\beta,[i,n,\ell+1]}.
\label{trivial2}
\end{equation}
By induction we know $\phi_{\unl{d}_{2}}(F_{2})$ has the factor $G_{\beta,[i,n,\ell+1]}$, and the wheel conditions $F_{2}=0$
once $x^{(\beta',1)}_{\ell,2}=v^{4}x^{(\beta,1)}_{\ell,1}=v^{2}x^{(\beta,1)}_{\ell-1,1}$ or
$x^{(\beta,1)}_{\ell,1}=v^{4}x^{(\beta',1)}_{\ell,2}=v^{2}x^{(\beta,1)}_{\ell+1,1}$ give the additional vanishing factors.
If $j=\ell-1$, then
\begin{equation}
  G_{\beta,\beta'}=(w_{\beta,1}-v^{(4\ell-4n+2)}w_{\beta',1})\cdot G_{\beta,[i,n,\ell+1]}.
\end{equation}
By induction we know $\phi_{\unl{d}_{2}}(F_{2})$ has the factor $G_{\beta,[i,n,\ell+1]}$.
Let $\unl{d}_{3}=\{d_{[1,\ell]}=d_{[i,n,\ell+1]}=1,d_{\alpha\neq [1,\ell],[i,n,\ell+1]}=0\}$, then $\unl{d}_{3}<\unl{d}_{2}$
and $\phi_{\unl{d}_{3}}(F_{2})=0$ implies that $\phi_{\unl{d}_{2}}(F_{2})=0$ once $w_{\beta,1}=v^{(4\ell-4n+2)}w_{\beta',1}$.
If $j=\ell$, then
\begin{equation}
  G_{\beta,\beta'}=(w_{\beta,1}-v^{(4\ell-4n-2)}w_{\beta',1})\cdot G_{\beta,[i,n,\ell+1]},
\end{equation}
and the wheel condition $F_{2}=0$ once $x^{(\beta',1)}_{\ell,2}=v^{4}x^{(\beta,1)}_{\ell,1}=v^{2}x^{(\beta,1)}_{\ell-1,1}$ gives
the additional vanishing factor.

\item
$\beta=[1,n,j]$ for some $2\leq j\leq n$ and $\beta'=[i,\ell]$ for some $2\leq i\leq \ell\leq n$.

We only need to prove $\beta=[1,n,n]$ case, other cases can be proved by induction as above. If $\ell\leq n-2$,
then $G_{\beta,\beta'}=G_{[1,n-1],\beta'}$ and the vanishing factors appear for the same wheel conditions.
If $\ell=n-1$, then
\begin{equation}
  G_{\beta,\beta'}=(w_{\beta,1}-v^{2}w_{\beta',1})\cdot G_{[1,n],\beta'}.
\end{equation}
Let $\unl{d}_{3}=\{d_{[1,n,n-1]}=d_{[i,n-2]}=1,d_{\alpha\neq [1,n,n-1],[i,n-2]}=0\}$, then $\unl{d}_{3}<\unl{d}_{2}$
and  $\phi_{\unl{d}_{3}}(F_{2})=0$ implies that $\phi_{\unl{d}_{2}}(F_{2})=0$ once $w_{\beta,1}=v^{2}w_{\beta',1}$.
If $\ell=n$, then
\begin{equation}
  G_{\beta,\beta'}=(w_{\beta,1}-v^{-4}w_{\beta',1})\cdot G_{\beta,[i,n-1]},
\end{equation}
and the wheel condition $F_{2}=0$ once
$x^{(\beta',1)}_{n,1}=v^2x^{(\beta,1)}_{n,2}=v^{4}x^{(\beta,1)}_{n,1}=v^2x^{(\beta,1)}_{n-1,1}$
becomes $\phi_{\unl{d}_{2}}(F_{2})=0$ once $w_{\beta,1}=v^{-4}w_{\beta',1}$.

\item
$\beta=[1,n,j]$ for some $2\leq j\leq n$, $\beta'=[i,n,\ell]$ for some $1\leq i<\ell\leq n$.

If $j\geq 3$ and $i>2$, then $G_{\beta,\beta'}=G_{[2,n,j],\beta'}$ and they appear for the same reason.

Let us consider now $j\geq 3$ and $i=2$. Then if $\ell< j$, we have
\begin{equation}
  G_{\beta,\beta'}=(w_{\beta,1}-w_{\beta',1})\cdot G_{[2,n,j],\beta'},
\end{equation}
Let $\unl{d}_{3}=\{d_{[1,n,\ell]}=d_{[2,n,j]}=1,d_{\alpha\neq [1,n,\ell],[2,n,j]}=0\}$, then $\unl{d}_{3}<\unl{d}_{2}$ and
$\phi_{\unl{d}_{3}}(F_{2})=0$ implies that $\phi_{\unl{d}_{2}}(F_{2})=0$ once $w_{\beta,1}=w_{\beta',1}$. If $\ell\geq j+2$,
then
\begin{equation}
  G_{\beta,\beta'}=(w_{\beta,1}-v^{4n-4j+2}w_{\beta',1})(w_{\beta,1}-v^{4n-4j-6}w_{\beta',1})\cdot G_{[1,n,j+1],\beta'},
\label{fff}
\end{equation}
and the wheel condition $x^{(\beta,1)}_{j,2}=v^{4}x^{(\beta',1)}_{j,1}=v^{2}x^{(\beta',1)}_{j-1,1}$ or
$x^{(\beta',1)}_{j,1}=v^{4}x^{(\beta,1)}_{j,2}=v^{2}x^{(\beta',1)}_{j+1,1}$ gives the vanishing factors.
If $\ell=j+1$, then
\begin{equation}
  G_{\beta,\beta'}=(w_{\beta,1}-v^{4n-4j+2}w_{\beta',1})(w_{\beta,1}-v^{4n-4j-6}w_{\beta',1})(w_{\beta,1}-v^{4}w_{\beta',1})\cdot
  G_{[1,n,j+1],\beta'},
\end{equation}
then the wheel condition for \eqref{fff} and $x^{(\beta,1)}_{j+1,2}=v^{4}x^{(\beta',1)}_{j+1,2}=v^{2}x^{(\beta,1)}_{j,2}$
give the vanishing factors. If $\ell=j$, then
\begin{equation}
  G_{\beta,\beta'}=(w_{\beta,1}-v^{-4}w_{\beta',1})\cdot G_{\beta',\beta'},
\end{equation}
and the wheel condition $x^{(\beta',1)}_{2,1}=v^{4}x^{(\beta,1)}_{2,1}=v^{2}x^{(\beta,1)}_{1,1}$ give the vanishing factor.

If $j\geq 3$ and $i=1$, then $\ell<j$. If $\ell\geq 3$, we have
\begin{equation}
  G_{\beta,\beta'}=(w_{\beta,1}-v^{-4}w_{\beta',1})(w_{\beta,1}-v^{4}w_{\beta',1})\cdot G_{[2,n,j],[2,n,\ell]},
\end{equation}
and the wheel conditions for \eqref{trivial} give the additional factors. If $\ell=2$, then for $j>3$, we can use induction.
If $\ell=2, j=3$, then
\begin{equation}
  G_{\beta,\beta'}=
  (w_{\beta,1}-v^{-4}w_{\beta',1})(w_{\beta,1}-v^{-4n+6}w_{\beta',1})(w_{\beta,1}-v^{-4n+14}w_{\beta',1})\cdot G_{\beta,\beta}.
\label{trivial3}
\end{equation}
The wheel conditions for \eqref{trivial2} and the wheel condition
  $x^{(\beta',1)}_{3,2}=v^{4}x^{(\beta,1)}_{3,2}=v^{2}x^{(\beta',1)}_{2,2}$
give the additional factors. Now let $\beta=[1,n,2]$, then $i\geq 2$. If $\ell>3$  then we can use induction.
If $\ell=3$, then $\beta'=[2,n,3]$, and similar wheel conditions for \eqref{trivial3} apply.

\end{itemize}
This completes our proof.
\end{proof}

Using formulas \eqref{trivial}--\eqref{trivial1}, we obtain the following explicit formulas for the factors $G_{\beta}$
(which shall be used in Subsection \ref{rttb}):

\begin{Cor}\label{formulagbeta}
The factors $\{G_{\beta}\}_{\beta\in\Delta^{+}}$ featuring in \eqref{speehb} are explicitly given by:
\begin{equation}\label{eqformulagbeta}
\begin{aligned}
  & G_{\beta} \ =
    \prod_{s=1}^{d_{\beta}}w_{\beta,s}^{j-i}\prod_{1\leq s\neq s'\leq d_{\beta}}(w_{\beta,s}-v^{4}w_{\beta,s'})^{j-i} \qquad
    \text{if}\ \beta=[i,j],\\
  & G_{\beta} \ =
    \prod_{s=1}^{d_{\beta}}w_{\beta,s}^{4n-i-3j+1}
    \prod_{1\leq s\neq s'\leq d_{\beta}} \big\{(w_{\beta,s}-v^{4}w_{\beta,s'})^{2n-i-j}(w_{\beta,s}-v^{2}w_{\beta,s'})\big\}\times &\\
  & \quad \ \prod_{1\leq s\neq s'\leq d_{\beta}} \prod_{\ell=j}^{n-1}
    \big\{(w_{\beta,s}-v^{4n-4\ell+2}w_{\beta,s'})(w_{\beta,s}-v^{4n-4\ell-6}w_{\beta,s'})\big\} \qquad
    \text{if}\ \beta=[i,n,j].
\end{aligned}
\end{equation}
\end{Cor}

%%%%%%%%%%%%%%%%%%%%%%%%%%%%%%%%%%%%%%%%%%%%%%%%%%%%%%%%%%%%%%%%%%%%%%%%%%%%%%%%%%%%%%%%%%%%
%%%%%%%%%%%%%%%%%%%%%%%%%%%% end of Yue's original text over the field %%%%%%%%%%%%%%%%%%%%%%%%%%%%
%%%%%%%%%%%%%%%%%%%%%%%%%%%%%%%%%%%%%%%%%%%%%%%%%%%%%%%%%%%%%%%%%%%%%%%%%%%%%%%%%%%%%%%%%%%%

Combining Propositions \ref{spekpB}--\ref{spanB}, we immediately obtain the shuffle algebra realization and
the PBWD theorem for $\qfb$:

\begin{Thm}\label{shufflePBWDB}
(a) $\Psi\colon \qfb \,\iso\, S$ of~\eqref{eq:Psi-homom} is a $\BQ(v)$-algebra isomorphism.

\medskip
\noindent
(b) For any choices of $s_k$ and $\lambda_k$ in the definition~\eqref{rootvector1} of quantum root vectors $E_{\beta,s}$,
the ordered PBWD monomials $\{E_{h}\}_{h\in H}$ from \eqref{PBWDbases} form a $\BQ(v)$-basis of $\qfb$.
\end{Thm}

   %%%%%%%%%%%%%%%%%%%%%%%%%%%%%%%%%%%%%%%%%%%%%%%%%%%%%%%%%%%%%%%%%%
   %%%%%%%%%%%%%%%%%%%%%%%%%%%%%%%%%%%%%%%%%%%%%%%%%%%%%%%%%%%%%%%%%%
   %%%%%%%%%%%%%%%%%%%%%%%%%%%%%%%%%%%%%%%%%%%%%%%%%%%%%%%%%%%%%%%%%%

\subsection{RTT integral form $\integralb$ and its shuffle algebra realization}\label{rttb}

For $\epsilon\in\{\pm \}$, we define
\begin{equation}\label{eq:rtt-vectors}
  \tilde{\mathcal{E}}^{\epsilon}_{\beta,s}\coloneqq \langle 2\rangle_{v}\cdot \tilde{E}^{\epsilon}_{\beta,s}
  \qquad \forall\ (\beta,s)\in\Delta^{+}\times \BZ,
\end{equation}
cf.~\eqref{rvb1}--\eqref{rvb2}. Similarly to \eqref{PBWDbases}, we also consider the ordered monomials
\begin{equation}
  \tilde{\mathcal{E}}^{\epsilon}_{h} \ =
  \prod_{(\beta,s)\in\Delta^{+}\times\mathbb{Z}}\limits^{\rightarrow}
  (\tilde{\mathcal{E}}^{\epsilon}_{\beta,s})^{h(\beta,s)}\qquad \forall\ h\in H,
\label{pbwdbasesrtt}
\end{equation}
with the arrow $\rightarrow$ over the product sign indicating the order \eqref{orderbetas} on $\Delta^{+}\times\BZ$.

\begin{Rk}
The reason for the extra factor $\langle 2\rangle_{v}$ in the definition~\eqref{eq:rtt-vectors} is explained in
the Appendix~\ref{sec:app}, see Corollary~\ref{cor:Dr-via-rtt-roots} and Proposition~\ref{prop:RTT-vs-Drinfeld integral forms}.
\end{Rk}

We define the RTT integral form $\integralb$ as the $\BZ[v,v^{-1}]$-subalgebra of $\qfb$ generated by
$\{\tilde{\mathcal{E}}^{\epsilon}_{\beta,s}\}_{\beta\in\Delta^{+}}^{s\in\BZ}$. We note that the above definition depends on
the choices of quantum root vectors in \eqref{rvb1}--\eqref{rvb2} as well as of $\epsilon\in\{\pm\}$. The main goal of the present
subsection is to prove the following theorem by simultaneously establishing the shuffle algebra realization of $\integralb$,
which is of independent interest:

\begin{Thm}\label{PBWDintegralb}
(a) $\integralb$ is independent of the choice of quantum root vectors
$\{\tilde{\mathcal{E}}^{\epsilon}_{\beta,s}\}_{\beta\in\Delta^{+}}^{s\in \BZ}$.

\medskip
\noindent
(b) For any choices of $s_k$ in~\eqref{rvb1}--\eqref{rvb2} and $\epsilon\in \{\pm\}$, the ordered monomials
$\{\tilde{\mathcal{E}}^{\epsilon}_{h}\}_{h\in H}$ of~\eqref{pbwdbasesrtt} form a basis of
the free $\BZ[v,v^{-1}]$-module $\integralb$.
\end{Thm}

For any $\unl{k}\in\BN^{n}$, consider the $\BZ[v,v^{-1}]$-submodule $\tilde{\mathcal{S}}_{\unl{k}}$ of $S_{\unl{k}}$ consisting
of rational functions $F$ satisfying the following two conditions:
\begin{enumerate}[leftmargin=1cm]

\item
If $f$ denotes the numerator of $F$ from \eqref{polecondition}, then
\begin{equation}
  f\in {\langle 2\rangle_{v}}^{|\unl{k}|}\cdot
  \BZ[v,v^{-1}][\{x_{i,r}^{\pm 1}\}_{1\leq i\leq n}^{1\leq r\leq k_{i}}]^{\mathfrak{S}_{\underline{k}}},
\label{rttconstant1}
\end{equation}
where $|\unl{k}|=|(k_1,\ldots,k_n)|:=k_1+\dots+k_n$.

\item
For any $\unl{d}\in\text{KP}(\underline{k})$, the specialization $\phi_{\unl{d}}(f\cdot {\langle 2\rangle_{v}}^{-|\unl{k}|})$
is divisible by
\begin{equation}
  \prod_{\beta=[i,n,j]\in \Delta^{+}}\prod_{\ell=j}^{n-1}
  \big\{(v^{-4n+4\ell-2}-1)^{d_{\beta}}(v^{-4n+4\ell+6}-1)^{d_{\beta}}\big\}.
\label{rttconstant2}
\end{equation}

\end{enumerate}
 We define $\tilde{\mathcal{S}}:=\bigoplus_{\unl{k}\in\BN^{n}}\tilde{\mathcal{S}}_{\unl{k}}$. Then, we have:

\begin{Prop}\label{goodB}
$\Psi(\integralb) \subset \tilde{\mathcal{S}}$.
\end{Prop}

\begin{proof}
For any $\epsilon\in\{\pm\}$, $m\in\BN$, $\beta_{1},\dots,\beta_{m}\in\Delta^{+}$, $r_{1},\dots,r_{m}\in\BZ$, let
  \[F\coloneqq
    \Psi\big(\tilde{\mathcal{E}}^{\epsilon}_{\beta_{1},r_{1}}\cdots \tilde{\mathcal{E}}^{\epsilon}_{\beta_{m},r_{m}}\big),\]
and $f$ be the numerator of $F$ from \eqref{polecondition}. We set $\unl{k}=\sum_{q=1}^{m} \beta_{q}$. Similarly
to~\eqref{eq:o-spot}, if a variable $x^{(*,*)}_{*,*}$ is plugged into $\Psi(\tilde{\mathcal{E}}^{\epsilon}_{\beta_{q},r_{q}})$
for some $1\leq q\leq m$, then we shall use the notation
  $$o(x^{(*,*)}_{*,*})=q.$$
Due to Lemma \ref{phirvns}, $f$ is divisible by ${\langle 2\rangle_{v}}^{|\unl{k}|}$, hence, the condition \eqref{rttconstant1}
holds. Now for any $\unl{d}\in\text{KP}(\unl{k})$, consider each summand from the symmetrization featuring in $f$.
Pick any $\beta=[i,n,j]$ with $1\leq i<j\leq n$ such that $d_{\beta}\neq 0$. For any $1\leq s\leq d_{\beta}$ and
$j\leq \ell\leq n$, it suffices to show that the contribution of the $\phi_{\unl{d}}$-specializations of $\zeta$-factors
between the variables
  \[\big\{x^{(\beta,s)}_{\ell-1,1} \,,\, x^{(\beta,s)}_{\ell,1} \,,\, \dots \,,\, x^{(\beta,s)}_{n,1},x^{(\beta,s)}_{n,2}
    \,,\, \dots \,,\, x^{(\beta,s)}_{\ell+1,2} \,,\, x^{(\beta,s)}_{\ell,2}\big\}\]
is divisible by $\prod_{t=\ell}^{n-1} \{(v^{-4n+4t-2}-1)(v^{-4n+4t+6}-1)\}$.
The proof is by induction on $\ell$, where the base step $\ell=n$ is vacuous.

We first note that this $\phi_{\unl{d}}$-specialization vanishes unless
  \[o(x^{(\beta,s)}_{\ell-1,1})\geq o(x^{(\beta,s)}_{\ell,1})\geq \cdots\geq o(x^{(\beta,s)}_{n,1})\geq
    o(x^{(\beta,s)}_{n,2})\geq\cdots\geq o(x^{(\beta,s)}_{\ell+1,2})\geq o(x^{(\beta,s)}_{\ell,2}).\]
If $o(x^{(\beta,s)}_{\ell,1})=o(x^{(\beta,s)}_{\ell,2})$, then due to Lemma~\ref{phirvns},
this $\phi_{\unl{d}}$-specialization contains the required factor
  \[ \prod_{t=\ell}^{n-1} \{(v^{-4n+4t-2}-1)(v^{-4n+4t+6}-1)\}. \]
If $o(x^{(\beta,s)}_{\ell,1})>o(x^{(\beta,s)}_{\ell,2})$, then we have the following two cases to consider:
\begin{itemize}[leftmargin=0.7cm]

\item
Case 1: $o(x^{(\beta,s)}_{\ell+1,1})=o(x^{(\beta,s)}_{\ell,2})$. Then we have
 \[o(x^{(\beta,s)}_{\ell+1,1})=o(x^{(\beta,s)}_{\ell+2,1})=\cdots=o(x^{(\beta,s)}_{\ell+1,2})= o(x^{(\beta,s)}_{\ell,2}).\]
According  to Lemma \ref{phirvns}, the corresponding $\phi_{\unl{d}}$-specialization is divisible by
  \[\prod_{t=\ell+1}^{n-1} \{(v^{-4n+4t-2}-1)(v^{-4n+4t+6}-1)\}.\]
On the other hand, the $\phi_{\unl{d}}$-specialization of the product of $\zeta$-factors
  \[\zeta_{\ell,\ell-1}(x^{(\beta,s)}_{\ell,2}/x^{(\beta,s)}_{\ell-1,1})
    \zeta_{\ell,\ell}(x^{(\beta,s)}_{\ell,2}/x^{(\beta,s)}_{\ell,1})
    \zeta_{\ell+1,\ell}(x^{(\beta,s)}_{\ell+1,2}/x^{(\beta,s)}_{\ell,1})\]
contributes the remaining required factor $(v^{-4n+4\ell-2}-1)(v^{-4n+4\ell+6}-1)$.

\item
Case 2: $o(x^{(\beta,s)}_{\ell+1,1})>o(x^{(\beta,s)}_{\ell,2})$.
Then the $\phi_{\unl{d}}$-specialization of the product of $\zeta$-factors
  \[\zeta_{\ell,\ell-1}(x^{(\beta,s)}_{\ell,2}/x^{(\beta,s)}_{\ell-1,1})
    \zeta_{\ell,\ell}(x^{(\beta,s)}_{\ell,2}/x^{(\beta,s)}_{\ell,1})
    \zeta_{\ell,\ell+1}(x^{(\beta,s)}_{\ell,2}/x^{(\beta,s)}_{\ell+1,1})\]
contributes the factor $(v^{-4n+4\ell-2}-1)(v^{-4n+4\ell+6}-1)$.
Considering the contribution of the $\phi_{\unl{d}}$-specializations of $\zeta$-factors between the variables
  \[\big\{x^{(\beta,s)}_{\ell,1} \,,\, x^{(\beta,s)}_{\ell+1,1} \,,\, \dots \,,\, x^{(\beta,s)}_{n,1},x^{(\beta,s)}_{n,2}
    \,,\, \dots \,,\, x^{(\beta,s)}_{\ell+2,2} \,,\, x^{(\beta,s)}_{\ell+1,2}\big\}\]
and using the induction hypothesis, we get the remaining required factors
  \[\prod_{t=\ell+1}^{n-1} \{(v^{-4n+4t-2}-1)(v^{-4n+4t+6}-1)\} .\]

\end{itemize}
This completes our proof.
\end{proof}

We shall now introduce a certain refinement of $\tilde{\mathcal{S}}$ in order to describe the image $\Psi(\integralb)$.
Pick any $F\in\tilde{\mathcal{S}}_{\unl{k}}$ and $\unl{d}\in\mathrm{KP}(\unl{k})$. First, according to
\eqref{rttconstant1} and \eqref{rttconstant2}, $\phi_{\unl{d}}(F)$ is divisible by
\begin{equation}
  A_{\unl{d}}\coloneqq {\langle 2\rangle_{v}}^{|\unl{k}|} \cdot
  \prod_{\beta=[i,n,j]\in \Delta^{+}}\prod_{\ell=j}^{n-1}
  \big\{(v^{-4n+4\ell-2}-1)^{d_{\beta}}(v^{-4n+4\ell+6}-1)^{d_{\beta}}\big\}.
\end{equation}
Second, following Corollary \ref{formulagbeta} (based solely on the wheel conditions), the specialization
$\phi_{\unl{d}}(F)$ is also divisible by the product
  \[B_{\unl{d}}\coloneqq \prod_{\beta\in\Delta^{+}}G_{\beta},\]
with $G_{\beta}$ computed explicitly in \eqref{eqformulagbeta}.
Combining these two observations, we can now define the following {\em reduced specialization map}
\begin{equation}
  \xi_{\unl{d}}\colon \tilde{\mathcal{S}} \longrightarrow
  \BZ[v,v^{-1}][\{w_{\beta,s}^{\pm 1}\}_{\beta\in\Delta^{+}}^{1\leq s\leq d_{\beta}}]^{\mathfrak{S}_{\unl{d}}}
  \qquad \text{via} \qquad \xi_{\unl{d}}(F)\coloneqq \frac{\phi_{\unl{d}}(F)}{A_{\unl{d}}B_{\unl{d}}}.
\label{reducedspe}
\end{equation}

Let us introduce another type of specialization maps. Pick any collection of positive integers
$\unl{t}=\{t_{\beta,r}\}_{\beta\in\Delta^{+}}^{1\leq r\leq \ell_{\beta}}\ (\ell_{\beta}\in\BN)$ satisfying
\begin{equation}
  d_{\beta}=\sum_{r=1}^{\ell_{\beta}}t_{\beta,r} \qquad \forall\ \beta\in\Delta^{+}.
\label{verticalpartition}
\end{equation}
For any $\beta\in\Delta^{+}$, we split the variables $\{w_{\beta,s}\}_{s=1}^{d_{\beta}}$ into $\ell_{\beta}$ groups
of size $t_{\beta,r}$ each $(1\leq r\leq \ell_{\beta})$ and specialize the variables in the $r$-th group to
\begin{equation}
  v_{\beta}^{-2}z_{\beta,r},\ v_{\beta}^{-4}z_{\beta,r},\ \dots\ ,\ v_{\beta}^{-2t_{\beta,r}}z_{\beta,r}.
\label{verticalspemap}
\end{equation}
For any
  $g\in \BZ[v,v^{-1}][\{w_{\beta,s}^{\pm 1}\}_{\beta\in\Delta^{+}}^{1\leq s\leq d_{\beta}}]^{\mathfrak{S}_{\unl{d}}}$,
we define $\varpi_{\unl{t}}(g)$ as the corresponding specialization of~$g$. This gives rise to the {\em vertical specialization map}
\begin{equation}
  \varpi_{\unl{t}}\colon
  \BZ[v,v^{-1}][\{w_{\beta,s}^{\pm 1}\}_{\beta\in\Delta^{+}}^{1\leq s\leq d_{\beta}}]^{\mathfrak{S}_{\unl{d}}} \longrightarrow
  \BZ[v,v^{-1}][\{z_{\beta,r}^{\pm 1}\}_{\beta\in\Delta^{+}}^{1\leq r\leq \ell_{\beta}}].
\label{verticalspe}
\end{equation}

Finally, given any $\unl{d}\in\mathrm{KP}(\unl{k})$ and a collection of positive integers
$\unl{t}=\{t_{\beta,r}\}_{\beta\in\Delta^{+}}^{1\leq r\leq \ell_{\beta}}$ satisfying \eqref{verticalpartition},
we combine \eqref{reducedspe} and \eqref{verticalspe} to define the \textbf{cross specialization map}
\begin{equation}
  \Upsilon_{\unl{d},\unl{t}}\colon \tilde{\mathcal{S}} \longrightarrow
  \BZ[v,v^{-1}][\{z_{\beta,r}^{\pm 1}\}_{\beta\in\Delta^{+}}^{1\leq r\leq \ell_{\beta}}]
  \quad \text{via} \quad
  \Upsilon_{\unl{d},\unl{t}}(F)\coloneqq\varpi_{\unl{t}}(\xi_{\unl{d}}(F)).
\label{crossspe}
\end{equation}
Similarly to~\cite[Definition 3.37]{Tsy19}, we introduce:

\begin{Def}\label{def:rtt-integral-shuffle}
$F\in\tilde{\mathcal{S}}_{\unl{k}}$ is \textbf{integral} if $\Upsilon_{\unl{d},\unl{t}}(F)$ is divisible by
$\prod_{\beta\in\Delta^{+}}^{1\leq r\leq \ell_{\beta}}[t_{\beta,r}]_{v_{\beta}}!$ for any $\unl{d}\in\mathrm{KP}(\unl{k})$
and $\unl{t}=\{t_{\beta,r}\}_{\beta\in\Delta^{+}}^{1\leq r\leq \ell_{\beta}}$ satisfying \eqref{verticalpartition}.
\end{Def}

Let $\mathcal{S}\subset \tilde{\mathcal{S}}$ denote the $\BZ[v,v^{-1}]$-submodule of all integral elements. Then, we have:

\begin{Prop}\label{integralbsubset}
$\Psi(\integralb)\subset \mathcal{S}$.
\end{Prop}

\begin{proof}
For any $\epsilon\in\{\pm\}$, $m\in\BN$, $\beta_{1},\dots,\beta_{m}\in\Delta^{+}$, $r_{1},\dots,r_{m}\in\BZ$, let
  \[F\coloneqq
    \Psi\big(\tilde{\mathcal{E}}^{\epsilon}_{\beta_{1},r_{1}}\cdots \tilde{\mathcal{E}}^{\epsilon}_{\beta_{m},r_{m}}\big).\]
For any $\beta\in\Delta^{+}$ and $1\leq r\leq \ell_{\beta}$, we need to show that under $\Upsilon_{\unl{d},\unl{t}}$,
the contribution of the $\zeta$-factors between the variables $x^{(*,*)}_{*,*}$ that got specialized to $v^{?}z_{\beta,r}$
is divisible by $[t_{\beta,r}]_{v_{\beta}}!$. For $\beta=[i,j]$ with $1\leq i\leq j\leq n$, this follows from
\cite[Lemma 3.51]{Tsy18} (note that $v_{[i,j]}=v_{j}$).

It remains to treat the case $\beta=[i,n,j]$ with $1\leq i<j\leq n$. We note that $v_{\beta}=v^{2}$ for this~$\beta$.
Without loss of generality, we may assume that under the $\phi_{\unl{d}}$-specialization:
  \[x^{(\beta,s)}_{i,1}\mapsto v^{-2i}w_{\beta,s} \,,\, \dots \,,\, x^{(\beta,s)}_{n,1}\mapsto v^{-2n}w_{\beta,s} \,,\,
    x^{(\beta,s)}_{n,2}\mapsto v^{-2n+2}w_{\beta,s} \,,\, \dots \,,\, x^{(\beta,s)}_{j,2}\mapsto  v^{-4n+2j+2}w_{\beta,s}\]
for $1\leq s\leq t_{\beta,r}$, while under the $\varpi_{\unl{t}}$-specialization:
  \[w_{\beta,1}\mapsto v^{-4}z_{\beta,r} \,,\, \dots \,,\, w_{\beta,t_{\beta,r}}\mapsto v^{-4t_{\beta,r}}z_{\beta,r}.\]

Fix any $1\leq s\neq s'\leq t_{\beta,r}$. First, let us consider the relative position of the variables
\begin{equation}\label{jcase}
  \big\{x^{(\beta,s)}_{j+1,2} \,,\, x^{(\beta,s)}_{j,2} \,,\, x^{(\beta,s')}_{j+1,2} \,,\, x^{(\beta,s')}_{j,2}\big\},
\end{equation}
that is, compare their $o$-values with $o(x^{(*,*)}_{*,*})$ defined now as in the proof of Proposition \ref{goodB}.
We can assume that
  \[o(x^{(\beta,s)}_{j+1,2})\geq o(x^{(\beta,s)}_{j,2}) \quad \mathrm{and} \quad
    o(x^{(\beta,s')}_{j+1,2})\geq o(x^{(\beta,s')}_{j,2}),\]
as otherwise the corresponding term is specialized to zero under $\phi_{\unl{d}}$. If
$o(x^{(\beta,s)}_{j,2})=o(x^{(\beta,s')}_{j,2})$ then, according to Lemma \ref{phirvns}, the corresponding term is
divisible by
  \[(v^{4}x^{(\beta,s)}_{j,2}-x^{(\beta,s')}_{j,2})(v^{4}x^{(\beta,s')}_{j,2}-x^{(\beta,s)}_{j,2}),\]
whose $\phi_{\unl{d}}$-specialization contributes the factor
$(w_{\beta,s}-v^{4}w_{\beta,s'})(w_{\beta,s'}-v^{4}w_{\beta,s})$. If $o(x^{(\beta,s)}_{j,2})\neq o(x^{(\beta,s')}_{j,2})$,
then without loss of generality, we can assume $o(x^{(\beta,s)}_{j,2})> o(x^{(\beta,s')}_{j,2})$. Then:
\begin{itemize}[leftmargin=0.7cm]

\item
If $o(x^{(\beta,s')}_{j+1,2})> o(x^{(\beta,s)}_{j,2})$, then the $\phi_{\unl{d}}$-specialization of the product
of $\zeta$-factors
  \[\zeta_{j,j+1}(x^{(\beta,s)}_{j,2}/x^{(\beta,s')}_{j+1,2}) \cdot
    \zeta_{j,j+1}(x^{(\beta,s')}_{j,2}/x^{(\beta,s)}_{j+1,2})\]
contributes the factor $(w_{\beta,s}-v^{4}w_{\beta,s'})(w_{\beta,s'}-v^{4}w_{\beta,s})$.

\item
If $o(x^{(\beta,s)}_{j,2})\geq o(x^{(\beta,s')}_{j+1,2})$, then the $\phi_{\unl{d}}$-specialization of the product
of $\zeta$-factors
  \[\zeta_{j,j}(x^{(\beta,s')}_{j,2}/x^{(\beta,s)}_{j,2})\cdot
    \zeta_{j,j+1}(x^{(\beta,s')}_{j,2}/x^{(\beta,s)}_{j+1,2})\]
contributes the factor $(w_{\beta,s}-v^{4}w_{\beta,s'})(w_{\beta,s'}-v^{4}w_{\beta,s})$.

\end{itemize}
Similarly, considering the $\phi_{\unl{d}}$-specialization of the $\zeta$-factors arising from the following quadruples
\begin{align*}
  \big\{x^{(\beta,s)}_{j+2,2} \,,\, x^{(\beta,s)}_{j+1,2} \,,\, x^{(\beta,s')}_{j+2,2} \,,\, x^{(\beta,s')}_{j+1,2}\big\}
    \,,\, \dots \,,\,
  \big\{x^{(\beta,s)}_{n,2} \,,\, x^{(\beta,s)}_{n-1,2} \,,\, x^{(\beta,s')}_{n,2},x^{(\beta,s')}_{n-1,2}\big\}, \\
  \big\{x^{(\beta,s)}_{n-2,1} \,,\, x^{(\beta,s)}_{n-1,1},x^{(\beta,s')}_{n-2,1} \,,\, x^{(\beta,s')}_{n-1,1}\big\}
    \,,\, \dots \,,\,
  \big\{x^{(\beta,s)}_{i,1} \,,\, x^{(\beta,s)}_{i+1,1} \,,\, x^{(\beta,s')}_{i,1},x^{(\beta,s')}_{i+1,1}\big\},
\end{align*}
along with the contribution of the tuple \eqref{jcase} considered above, they produce a total contribution of the factor
  \[(w_{\beta,s}-v^{4}w_{\beta,s'})^{2n-i-j-1}(w_{\beta,s'}-v^{4}w_{\beta,s})^{2n-i-j-1}.\]

Second, let us consider the relative position of the variables
\begin{equation}\label{ncase}
  \big\{x^{(\beta,s)}_{n-1,1} \,,\, x^{(\beta,s)}_{n,1} \,,\, x^{(\beta,s)}_{n,2} \,,\,
    x^{(\beta,s')}_{n-1,1} \,,\, x^{(\beta,s')}_{n,1} \,,\, x^{(\beta,s')}_{n,2}\big\}.
\end{equation}
We can assume that
\begin{equation}\label{assumptionintegral}
  o(x^{(\beta,s)}_{n-1,1})\geq o(x^{(\beta,s)}_{n,1})\geq o(x^{(\beta,s)}_{n,2})
  \quad \mathrm{and} \quad
  o(x^{(\beta,s')}_{n-1,1})\geq o(x^{(\beta,s')}_{n,1})\geq o(x^{(\beta,s')}_{n,2}),
\end{equation}
as otherwise the corresponding term is specialized to zero under $\phi_{\unl{d}}$. Then:
\begin{itemize}[leftmargin=0.7cm]

\item
If $o(x^{(\beta,s)}_{n,2})<o(x^{(\beta,s')}_{n,1})$ or $o(x^{(\beta,s')}_{n,1})<o(x^{(\beta,s)}_{n-1,1})$, then
the $\phi_{\unl{d}}$-specialization of $\zeta$-factor
  \[\zeta_{n,n}(x^{(\beta,s)}_{n,2}/x^{(\beta,s')}_{n,1}) \quad {\rm or} \quad
    \zeta_{n,n-1}(x^{(\beta,s')}_{n,1}/x^{(\beta,s)}_{n-1,1})\]
respectively, contributes the factor $w_{\beta,s}-v^{-4}w_{\beta,s'}$. Otherwise,
$o(x^{(\beta,s)}_{n,2})\geq o(x^{(\beta,s')}_{n,1})\geq o(x^{(\beta,s)}_{n-1,1})$,
which together with~\eqref{assumptionintegral} implies:
  \[o(x^{(\beta,s')}_{n,1})=o(x^{(\beta,s)}_{n-1,1})=o(x^{(\beta,s)}_{n,1})=o(x^{(\beta,s)}_{n,2}).\]
But there are at most two variables $x^{(*,*)}_{n,*}$ plugged into each
$\Psi(\tilde{\mathcal{E}}^{\epsilon}_{\beta_{q},r_{q}})$, a contradiction.

\item
If $o(x^{(\beta,s)}_{n,2})<o(x^{(\beta,s')}_{n,2})$, or $o(x^{(\beta,s)}_{n,1})<o(x^{(\beta,s')}_{n,1})$,
or $o(x^{(\beta,s')}_{n,2})<o(x^{(\beta,s)}_{n-1,1})$, then the $\phi_{\unl{d}}$-specialization of $\zeta$-factor
  \[\zeta_{n,n}(x^{(\beta,s)}_{n,2}/x^{(\beta,s')}_{n,2}) \quad {\rm or} \quad
    \zeta_{n,n}(x^{(\beta,s)}_{n,1}/x^{(\beta,s')}_{n,1}) \quad {\rm or} \quad
    \zeta_{n,n-1}(x^{(\beta,s')}_{n,2}/x^{(\beta,s)}_{n-1,1})\]
respectively, contributes the factor $w_{\beta,s}-v^{-2}w_{\beta,s'}$. Otherwise, if $o(x^{(\beta,s)}_{n,2})\geq o(x^{(\beta,s')}_{n,2})$,
$o(x^{(\beta,s)}_{n,1})\geq o(x^{(\beta,s')}_{n,1})$, and $o(x^{(\beta,s')}_{n,2})\geq o(x^{(\beta,s)}_{n-1,1})$, then
according to \eqref{assumptionintegral} we have
  \[o(x^{(\beta,s)}_{n-1,1})= o(x^{(\beta,s)}_{n,1})=o(x^{(\beta,s)}_{n,2})= o(x^{(\beta,s')}_{n,2}).\]
The latter can not occur for the same reason as above, hence, a contradiction.

\end{itemize}
Swapping the roles of $s$ and $s'$ in the above two bullets, we thus conclude that
the $\phi_{\unl{d}}$-specialization of the $\zeta$-factors arising from \eqref{ncase} contributes the factor
  \[(w_{\beta,s}-v^{4}w_{\beta,s'})(w_{\beta,s'}-v^{4}w_{\beta,s})
    (w_{\beta,s}-v^{2}w_{\beta,s'})(w_{\beta,s'}-v^{2}w_{\beta,s}).\]

Finally, for any $j\leq \ell<n$, let us consider the relative position of the variables
  \[\big\{x^{(\beta,s)}_{\ell,2}  \,,\, x^{(\beta,s')}_{\ell-1,1} \,,\,
    x^{(\beta,s')}_{\ell,1} \,,\, x^{(\beta,s')}_{\ell+1,1}\big\}.\]
We can assume that $o(x^{(\beta,s')}_{\ell-1,1})\geq o(x^{(\beta,s')}_{\ell,1})\geq o(x^{(\beta,s')}_{\ell+1,1})$,
as otherwise the corresponding term is specialized to zero under $\phi_{\unl{d}}$. If
$o(x^{(\beta,s)}_{\ell,2})=o(x^{(\beta,s')}_{\ell,1})$, then the corresponding term has the factor
  \[(x^{(\beta,s)}_{\ell,2}-v^{4}x^{(\beta,s')}_{\ell,1})(x^{(\beta,s)}_{\ell,2}-v^{-4}x^{(\beta,s')}_{\ell,1})\]
as in Lemma \ref{phirvns}, and so its $\phi_{\unl{d}}$-specialization produces the factor
  \[(w_{\beta,s}-v^{4n-4\ell+2}w_{\beta,s'})(w_{\beta,s}-v^{4n-4\ell-6}w_{\beta,s'}).\]
If $o(x^{(\beta,s)}_{\ell,2})<o(x^{(\beta,s')}_{\ell,1})$ or $o(x^{(\beta,s)}_{\ell,2})>o(x^{(\beta,s')}_{\ell,1})$,
then the $\phi_{\unl{d}}$-specialization of the products
  \[\zeta_{\ell,\ell}(x^{(\beta,s)}_{\ell,2}/x^{(\beta,s')}_{\ell,1})\cdot
    \zeta_{\ell,\ell-1}(x^{(\beta,s)}_{\ell,2}/x^{(\beta,s')}_{\ell-1,1}) \quad\text{or}\quad
    \zeta_{\ell+1,\ell}(x^{(\beta,s')}_{\ell+1,1}/x^{(\beta,s)}_{\ell,2}) \cdot
    \zeta_{\ell,\ell}(x^{(\beta,s')}_{\ell,1}/x^{(\beta,s)}_{\ell,2})\]
respectively, contributes the required factor
  \[(w_{\beta,s}-v^{4n-4\ell+2}w_{\beta,s'})(w_{\beta,s}-v^{4n-4\ell-6}w_{\beta,s'}).\]

All these contributions overall produce exactly the factor $G_{\beta}$ from Corollary \ref{formulagbeta}.
However, we have not used yet the factors $\zeta_{ii}(x^{(\beta,s)}_{i,1}/x^{(\beta,s')}_{i,1})$.
We can now appeal to the ``rank $1$'' computation of~\cite[Lemma 3.46]{Tsy18} to deduce
the required divisibility by $[t_{\beta,r}]_{v^{2}}!$.
\end{proof}

Finally, combining Propositions \ref{vanishB}, \ref{spanB}, and \ref{integralbsubset}, we obtain
the following upgrade of Theorem~\ref{shufflePBWDB} (we note that divisibility~(\ref{rttconstant1},~\ref{rttconstant2})
is precisely matching the constants of~\eqref{eq:c-factor-B}, while the divisibility condition from
Definition~\ref{def:rtt-integral-shuffle} is precisely matching the ``rank 1'' formula~\eqref{eq:rank1-power-trig}):

\begin{Thm}\label{rttthmb}
(a) The $\BQ(v)$-algebra isomorphism $\Psi\colon \qfb \,\iso\, S$ of Theorem {\rm \ref{shufflePBWDB}(a)}
gives rise to a $\BZ[v,v^{-1}]$-algebra isomorphism $\Psi\colon \integralb \,\iso\, \mathcal{S}$.

\medskip
\noindent
(b) Theorem {\rm \ref{PBWDintegralb}} holds.
\end{Thm}

   %%%%%%%%%%%%%%%%%%%%%%%%%%%%%%%%%%%%%%%%%%%%%%%%%%%%%%%%%%%%%%%%%%
   %%%%%%%%%%%%%%%%%%%%%%%%%%%%%%%%%%%%%%%%%%%%%%%%%%%%%%%%%%%%%%%%%%
   %%%%%%%%%%%%%%%%%%%%%%%%%%%%%%%%%%%%%%%%%%%%%%%%%%%%%%%%%%%%%%%%%%

\subsection{Integral form $\integralbl$ and its shuffle algebra realization}\label{lusb}

Similarly to \eqref{dividedpowers}, we consider the {\em divided powers}
\begin{equation}
  \mathbf{E}_{i,r}^{(k)}\coloneqq \frac{e_{i,r}^{k}}{[k]_{v_{i}}!} \quad \forall\ 1\leq i\leq n,\  r\in\mathbb{Z},\ k\in\BN.
\end{equation}
Likewise, we define the integral form $\integralbl$ as the $\BZ[v,v^{-1}]$-subalgebra of $\qfb$ generated by
$\{\mathbf{E}_{i,r}^{(k)}\}_{1\leq i\leq n, r\in\BZ}^{k\in\BN}$. For any $(\beta,s)\in\Delta^{+}\times \BZ$,
we define the {\em normalized divided powers} of the quantum root vectors from \eqref{rvb1}--\eqref{rvb2} via:
\begin{equation}
  \tilde{\mathbf{E}}_{\beta,s}^{\pm,(k)}\coloneqq
  \begin{cases}
    \frac{(\tilde{E}^{\pm}_{\beta,s})^{k}}{[k]_{v_{\beta}}!} & \text{if}\ \beta=[i,j] \ \mathrm{with} \ 1\leq j\leq n \\
    \frac{(\tilde{E}^{\pm}_{\beta,s})^{k}}{([2]_{v}!)^{k}[k]_{v_{\beta}}!} & \text{if} \ \beta=[i,n,j] \ \mathrm{with} \ 1\leq i<j\leq n
  \end{cases}.
\end{equation}
Similarly to Proposition \ref{integralrvg}, we obtain:

\begin{Prop}\label{integralrvb}
For any $\beta\in\Delta^{+}$, $s\in \BZ$, $k\in\BN$, we have $\tilde{\mathbf{E}}_{\beta,s}^{\pm,(k)}\in \integralbl$.
\end{Prop}

\begin{proof}
The proof is similar to that of Proposition \ref{integralrvg}. Let $\qgb$ be the ``positive subalgebra'' of the Drinfeld-Jimbo
quantum group of type $B_{n}$. Thus, $\qgb$ is the $\BQ(v)$-algebra generated by $\{E_{i}\}_{i=1}^{n}$ subject to the $v$-Serre
relations. Let $\qgbi$ be its \emph{Lusztig integral form}, defined as the $\BZ[v,v^{-1}]$-subalgebra of $\qgb$ generated by
the divided powers
  \[E_{i}^{(k)}\coloneqq \frac{E_{i}^{k}}{[k]_{v_{i}}!} \qquad  \forall\ 1\leq i\leq n, k\in\BN.\]

Recall our specific convex order \eqref{lynorderb} on $\Delta^{+}$. Let $\{\hat{E}_{\beta}^{-}\}_{\beta\in\Delta^{+}}$
denote Lusztig's quantum root vectors of $\qgb$ associated to this convex order. We also define
$\{\tilde{E}^{-}_{\beta}\}_{\beta\in\Delta^{+}}$ as the following iterated $v$-commutators similar to~\eqref{rvb1}--\eqref{rvb2}:
\begin{equation}
\begin{aligned}
  & \tilde{E}^{-}_{[i,j]}\coloneqq [\cdots[[E_{i},E_{i+1}]_{v^{-2}},E_{i+2}]_{v^{-2}},\cdots,E_{j}]_{v^{-2}},\\
  & \tilde{E}^{-}_{[i,n,j]}\coloneqq
    [\cdots[[[\cdots[E_{i},E_{i+1}]_{v^{-2}},\cdots, E_{n}]_{v^{-2}},E_{n}],E_{n-1}]_{v^{-2}},\cdots,E_{j}]_{v^{-2}}.
\end{aligned}
\end{equation}
Then, according to \cite[Proposition 5.5.2]{LS91} and \cite[Theorem 4.2]{BKM14}, we have:
\begin{equation}
  \tilde{E}^{-}_{\beta}=
  \begin{cases}
    \hat{E}^{-}_{\beta} & \text{if}\  \beta=[i,j] \ \mathrm{with}\ 1\leq i\leq j\leq n \\
    [2]_{v}!\hat{E}^{-}_{\beta} & \text{if}\  \beta=[i,n,j] \ \mathrm{with}\ 1\leq i<j\leq n
  \end{cases}.
\end{equation}

To pass from the finite to the loop setup, we note that for any $\unl{s}=(s_{1},\dots,s_{n})\in\BZ^{n}$, the assignment
$E_{1}\mapsto e_{1,s_{1}},\dots,E_{n}\mapsto e_{n,s_{n}}$ gives rise to an algebra homomorphism
  \[\eta_{\unl{s}}\colon \qgb\longrightarrow \qfb,\]
such that $\eta_{\unl{s}}(\qgbi)\subset \integralbl$.
As $\tilde{\mathbf{E}}^{-,(k)}_{\beta,s}=\eta_{\unl{s}}(\hat{\mathbf{E}}_{\beta}^{-,(k)})$ and
$\hat{\mathbf{E}}_{\beta}^{-,(k)}\in \qgbi$ by \cite[Theorem 6.6]{Lus90}, we get
$\tilde{\mathbf{E}}^{-,(k)}_{\beta,s}\in\integralbl$. Using similar arguments and the convex order on $\Delta^+$
opposite to~\eqref{lynorderb}, we also obtain $\tilde{\mathbf{E}}^{+,(k)}_{\beta,s}\in\integralbl$.
This completes our proof.
\end{proof}

For any $\unl{k}\in \BN^n$, consider the $\BZ[v,v^{-1}]$-submodule $\mathbf{S}_{\unl{k}}$ of $S_{\unl{k}}$ consisting of
rational functions $F$ satisfying the following two conditions:
\begin{enumerate}[leftmargin=1cm]

\item
If $f$ denotes the numerator of $F$ from \eqref{polecondition}, then
\begin{equation}
  f\in \BZ[v,v^{-1}][\{x_{i,r}^{\pm 1}\}_{1\leq i\leq n}^{1\leq r\leq k_{i}}]^{\mathfrak{S}_{\underline{k}}}.
\label{li1}
\end{equation}

\item
For any $\unl{d}\in\text{KP}(\underline{k})$, the specialization $\phi_{\unl{d}}(F)$ is divisible by the product
\begin{equation}
\begin{aligned}
  \prod_{\beta=[i,j]\in\Delta^{+}} {\langle 2\rangle_{v}}^{d_{\beta}(|\beta|-1)}\ \cdot
  & \prod_{\beta=[i,n,j]\in \Delta^{+}} \big\{{\langle 2\rangle_{v}}^{d_{\beta}(|\beta|-2)}\langle 1\rangle_{v}\big\} \times\\
  & \prod_{\beta=[i,n,j]\in \Delta^{+}}\prod_{\ell=j}^{n-1}
    \big\{(v^{-4(n-\ell)-2}-1)^{d_{\beta}}(v^{-4(n-\ell)+6}-1)^{d_{\beta}}\big\}.
\label{li2}
\end{aligned}
\end{equation}

\end{enumerate}

\begin{Rk}
We note that this definition is much simpler than that of $\mathcal{S}_{\unl{k}}$ in
Definition~\ref{def:rtt-integral-shuffle}.
\end{Rk}

We define $\mathbf{S}\coloneqq \bigoplus_{\unl{k}\in\BN^{n}}\mathbf{S}_{\unl{k}}$.
Then, similarly to Proposition \ref{lintegralG}, we have:

\begin{Prop}\label{lintegralB}
$\Psi(\integralbl) \subset \mathbf{S}$.
\end{Prop}

\begin{proof}
For any $m\in\BN$, $1\leq i_{1},\dots,i_{m}\leq n$, $r_{1},\dots,r_{m}\in\BZ$, $\ell_{1},\dots,\ell_{m}\in\BN$, let
  \[F\coloneqq
    \Psi\big(\mathbf{E}^{(\ell_{1})}_{i_{1},r_{1}}\cdots \mathbf{E}^{(\ell_{m})}_{i_{m},r_{m}}\big),\]
and $f$ be the numerator of $F$ from \eqref{polecondition}. According to Lemma \ref{rank2}:
  \[\Psi(\mathbf{E}^{(\ell_{q})}_{i_{q},r_{q}})=v_{i_{q}}^{-\frac{\ell_{q}(\ell_{q}-1)}{2}}(x_{i_{q},1}\cdots
    x_{i_{q},\ell_{q}})^{r_{q}}  \qquad \forall\ 1\leq q\leq m,\]
hence, the condition \eqref{li1} holds. To verify the validity of the divisibility \eqref{li2}, it suffices to show
(see Lemma~\ref{Gs1} for $\beta=[i,j]$) that for any $\beta=[i,n,j]$ ($1\leq i<j\leq n$), and $1\leq s\leq d_{\beta}$, the total contribution of
$\phi_{\unl{d}}$-specializations of the $\zeta$-factors between the variables
$\{x^{(\beta,s)}_{i,t}\}_{i\in\beta}^{1\leq t\leq \nu_{\beta,i}}$ is divisible by
\begin{equation}
  {\langle 2\rangle_{v}}^{|\beta|-2}\langle 1\rangle_{v} \cdot
   \prod_{\ell=j}^{n-1} \big\{(v^{-4(n-\ell)-2}-1)(v^{-4(n-\ell)+6}-1)\big\}.
\label{neqcase}
\end{equation}

We shall now use the notation $o(x^{(*,*)}_{*,*})$ defined as in \eqref{eq:o-spot}.
The $\phi_{\unl{d}}$-specialization of the corresponding product of $\zeta$-factors vanishes unless
  \[o(x^{(\beta,s)}_{i,1})\geq o(x^{(\beta,s)}_{i+1,1})\geq \cdots \geq o(x^{(\beta,s)}_{n,1})\geq
    o(x^{(\beta,s)}_{n,2})\geq\cdots\geq o(x^{(\beta,s)}_{j+1,2})\geq o(x^{(\beta,s)}_{j,2}),\]
Since the equality above can occur only in a single spot, due to $o(x^{(\beta,s)}_{i,t})\ne o(x^{(\beta,s)}_{i',t'})$
for $i\ne i'$, we need to treat the following two cases:
\begin{enumerate}

\item
  $o(x^{(\beta,s)}_{i,1})> o(x^{(\beta,s)}_{i+1,1})> \cdots > o(x^{(\beta,s)}_{n,1})>
   o(x^{(\beta,s)}_{n,2})>\cdots>o(x^{(\beta,s)}_{j+1,2})> o(x^{(\beta,s)}_{j,2})$.

\item
  $o(x^{(\beta,s)}_{i,1})> o(x^{(\beta,s)}_{i+1,1})> \cdots > o(x^{(\beta,s)}_{n,1})=
   o(x^{(\beta,s)}_{n,2})>\cdots>o(x^{(\beta,s)}_{j+1,2})> o(x^{(\beta,s)}_{j,2})$.

\end{enumerate}

In the first case, the $\phi_{\unl{d}}$-specialization of each $\zeta$-factor
  \[\zeta_{i+1,i}(x^{(\beta,s)}_{i+1,1}/x^{(\beta,s)}_{i,1}),\dots,
    \zeta_{n,n-1}(x^{(\beta,s)}_{n,1}/x^{(\beta,s)}_{n-1,1}), \zeta_{n-1,n}(x^{(\beta,s)}_{n-1,2}/x^{(\beta,s)}_{n,2}),
    \dots,\zeta_{j,j+1}(x^{(\beta,s)}_{j,2}/x^{(\beta,s)}_{j+1,2})\]
as well as of the product
  \[\zeta_{n,n}(x^{(\beta,s)}_{n,2}/x^{(\beta,s)}_{n,1})\cdot \zeta_{n,n-1}(x^{(\beta,s)}_{n,2}/x^{(\beta,s)}_{n-1,1})\]
contributes a multiple of $\langle 2\rangle_{v}$, thus totalling ${\langle 2\rangle_{v}}^{|\beta|-1}$.
On the other hand, for any $j\leq \ell\leq n-1$, the $\phi_{\unl{d}}$-specialization of the product of $\zeta$-factors
  \[\zeta_{\ell,\ell-1}(x^{(\beta,s)}_{\ell,2}/x^{(\beta,s)}_{\ell-1,1})\cdot
    \zeta_{\ell,\ell}(x^{(\beta,s)}_{\ell,2}/x^{(\beta,s)}_{\ell,1})\cdot
    \zeta_{\ell,\ell+1}(x^{(\beta,s)}_{\ell,2}/x^{(\beta,s)}_{\ell+1,1})\]
contributes precisely the required factor $(v^{-4(n-\ell)-2}-1)(v^{-4(n-\ell)+6}-1)$.

For the second case, the only difference is that we replace the product of two $\zeta$-factors
  $\zeta_{n,n}(x^{(\beta,s)}_{n,2}/x^{(\beta,s)}_{n,1})\cdot \zeta_{n,n-1}(x^{(\beta,s)}_{n,2}/x^{(\beta,s)}_{n-1,1})$
with a single $\zeta_{n,n-1}(x^{(\beta,s)}_{n,2}/x^{(\beta,s)}_{n-1,1})$, so that the first contribution of
${\langle 2\rangle_{v}}^{|\beta|-1}$ is now getting replaced with
${\langle 2\rangle_{v}}^{|\beta|-2} \cdot \langle 1\rangle_{v}$.

This completes our verification of the divisibility \eqref{li2}, thus concluding the proof.
\end{proof}

For any $h\in H$, define the ordered monomials (cf.~\eqref{PBWDbases})
\begin{equation}\label{eq:B-Lus-pbwd}
  \tilde{\mathbf{E}}^{+}_{h}\ =\prod_{(\beta,s)\in\Delta^{+}\times\mathbb{Z}}\limits^{\rightarrow}
  \tilde{\mathbf{E}}_{\beta,s}^{+,(h(\beta,s))},\qquad
  \tilde{\mathbf{E}}^{-}_{h}\ =\prod_{(\beta,s)\in\Delta^{+}\times\mathbb{Z}}\limits^{\rightarrow}
  \tilde{\mathbf{E}}_{\beta,s}^{-,(h(\beta,s))}.
\end{equation}
For any $\epsilon\in\{\pm\}$, let $\mathbf{S}^{\epsilon}_{\unl{k}}$ be the $\BZ[v,v^{-1}]$-submodule
of $\mathbf{S}_{\unl{k}}$ spanned by $\{\Psi(\tilde{\mathbf{E}}^{\epsilon}_{h})\}_{h\in H_{\unl{k}}}$.
Then, the following analogue of Lemma \ref{span} holds:

\begin{Prop}\label{spanib}
For any $F\in \mathbf{S}_{\unl{k}}$ and $\unl{d}\in\mathrm{KP}(\unl{k})$, if $\phi_{\unl{d}'}(F)=0$ for all
$\unl{d}'\in \mathrm{KP}(\unl{k})$ such that $\unl{d}'<\unl{d}$, then there exists
$F_{\unl{d}}\in \mathbf{S}^{\epsilon}_{\unl{k}}$ such that $\phi_{\unl{d}}(F)=\phi_{\unl{d}}(F_{\unl{d}})$
and $\phi_{\unl{d}'}(F_{\unl{d}})=0$ for all $\unl{d}'<\unl{d}$.
\end{Prop}

The proof of this result is completely analogous to that of Proposition~\ref{spanig}. Combining
Propositions \ref{lintegralB} and~\ref{spanib}, we obtain the following upgrade of Theorem \ref{shufflePBWDB}:

\begin{Thm}\label{lusthmb}
(a) The $\BQ(v)$-algebra isomorphism $\Psi\colon \qfb \,\iso\, S$ of Theorem~\ref{shufflePBWDB}(a) gives rise to a $\BZ[v,v^{-1}]$-algebra
isomorphism $\Psi\colon \integralbl \,\iso\, \mathbf{S}$.

\medskip
\noindent
(b) For any choices of $s_k$ in \eqref{rvb1}--\eqref{rvb2} and $\epsilon\in \{\pm\}$, the ordered monomials
$\{\tilde{\mathbf{E}}^{\epsilon}_{h}\}_{h\in H}$ of~\eqref{eq:B-Lus-pbwd} form a basis of
the free $\BZ[v,v^{-1}]$-module $\integralbl$.
\end{Thm}

   %%%%%%%%%%%%%%%%%%%%%%%%%%%%%%%%%%%%%%%%%%%%%%%%%%%%%%%%%%%%%%%%%%
   %%%%%%%%%%%%%%%%%%%%%%%%%%%%%%%%%%%%%%%%%%%%%%%%%%%%%%%%%%%%%%%%%%
   %%%%%%%%%%%%%%%%%%%%%%%%%%%%%%%%%%%%%%%%%%%%%%%%%%%%%%%%%%%%%%%%%%

\section{Yangian counterpart}\label{yangian}

In this section, we generalize the results of Sections \ref{tG}--\ref{type B} to the Yangian case, thus establishing
shuffle algebra realizations of Yangians  and their Drinfeld-Gavarini duals in types $B_{n}$ and $G_{2}$. This should be
viewed as the ``rational vs trigonometric'' counterpart, where we replace factors $\frac{z}{w}-v^k$ by $z-w-\frac{k}{2}\hbar$.
In particular, $\zeta_{i,j}(z)$ of~\eqref{eq:zeta} will be replaced by $\hzeta_{i,j}(z)=1+\frac{(\alpha_{i},\alpha_{j})\cdot \hbar}{2z}$.

   %%%%%%%%%%%%%%%%%%%%%%%%%%%%%%%%%%%%%%%%%%%%%%%%%%%%%%%%%%%%%%%%%%
   %%%%%%%%%%%%%%%%%%%%%%%%%%%%%%%%%%%%%%%%%%%%%%%%%%%%%%%%%%%%%%%%%%
   %%%%%%%%%%%%%%%%%%%%%%%%%%%%%%%%%%%%%%%%%%%%%%%%%%%%%%%%%%%%%%%%%%

\subsection{The Yangian $Y^{>}_{\hbar}(\mathfrak{g})$ and its shuffle algebra realization}

We still use the notations from Section \ref{pre}. Let $\mathfrak{g}$ be a finite dimensional simple Lie algebra of type
$B_{n}$ or $G_{2}$. Following \cite{Dri88}, the \textbf{``positive subalgebra'' of the Yangian of $\fg$} in the new Drinfeld
realization, denoted by $\Yangian$, is the $\BQ[\hbar]$-algebra generated by $\{\mathsf{x}_{i,r}\}_{i\in I}^{r\in\mathbb{N}}$
subject to the following defining relations:
\begin{equation}
  [\sx_{i,r+1},\sx_{j,s}]-[\sx_{i,r},\sx_{j,s+1}]=\frac{d_{i}a_{ij}\hbar}{2}(\sx_{i,r}\sx_{j,s}+\sx_{j,s}\sx_{i,r})
  \qquad \forall\ i,j\in I, r,s\in\BN,
\end{equation}
\begin{equation}
  \mathop{Sym}_{s_{1},\dots,s_{1-a_{ij}}}[\sx_{i,s_{1}},[\sx_{i,s_{2}},\cdots,[\sx_{i,s_{1-a_{ij}}},\sx_{j,r}]\cdots]]=0
  \qquad \forall\ i\neq j, s_{1},\dots,s_{1-a_{ij}},r\in \BN.
\end{equation}
Analogously to \eqref{rootvector1}--\eqref{rootvector2}, let us now define the root vectors
$\{\sX_{\beta,s}\}_{\beta\in\Delta^{+}}^{s\in\BN}$ of $\Yangian$:
\begin{itemize}[leftmargin=0.7cm]

\item
$B_{n}$-type.

\noindent
For any $\beta=[i_{1},\dots,i_{\ell}]\in\Delta^{+}$ from~\eqref{eq:SL-words} and $s\in\BN$, choose a decomposition
$s=s_{1}+\cdots+s_{\ell}$ with $s_{1}, \dots,s_{\ell}\in\BN$. Then, we define
\begin{equation}
  \sX_{\beta,s}\coloneqq [\cdots[[\sx_{i_{1},s_{1}},\sx_{i_{2},s_{2}}],\sx_{i_{3},s_{3}}],\cdots,\sx_{i_{\ell},s_{\ell}}].
\label{ygrv1}
\end{equation}

\medskip

\item
$G_{2}$-type.

\noindent
For $\beta=[i_{1},\dots,i_{\ell}]\neq [1,2,1,2,2]$, $s\in \BN$, the elements $\sX_{\beta,s}$ are defined exactly
as in~\eqref{ygrv1}.

\noindent
For $\beta=[1,2,1,2,2]$ and $s\in \BN$, we choose a decomposition $s=s_{1}+\cdots+s_{5}$ with $s_{1},\dots,s_{5}\in\BN$,
and define
\begin{equation}
  \sX_{\beta,s}\coloneqq [[\sx_{1,s_{1}},\sx_{2,s_{2}}],[[\sx_{1,s_{3}},\sx_{2,s_{4}}],\sx_{2,s_{5}}]].
\label{ygrv2}
\end{equation}

\end{itemize}
We will also need the following specific choices $\{\tilde{\sX}_{\beta,s}\}_{\beta\in\Delta^{+}}^{s\in\BN}$
of~\eqref{ygrv1}--\eqref{ygrv2}:
\begin{itemize}[leftmargin=0.7cm]

\item
$B_{n}$-type.
\begin{align}
  & \tilde{\sX}_{[i,j],s}\coloneqq [\cdots[[\sx_{i,s},\sx_{i+1,0}],\sx_{i+2,0}],\cdots,\sx_{j,0}],\quad 1\leq i\leq j\leq n,\\
  & \tilde{\sX}_{[i,n,j],s}\coloneqq
    [\cdots[[[\cdots[\sx_{i,s},\sx_{i+1,0}],\cdots, \sx_{n,0}],\sx_{n,0}],\sx_{n-1,0}],\cdots,\sx_{j,0}],\quad 1\leq i<j\leq n.
\end{align}

\medskip

\item
$G_{2}$-type.
\begin{align}
  & \tilde{\sX}_{[i],s}\coloneqq \sx_{i,s},\quad 1\leq i\leq 2,\\
  & \tilde{\sX}_{[1,2],s}\coloneqq [\sx_{1,s},\sx_{2,0}],\\
  & \tilde{\sX}_{[1,2,2],s}\coloneqq [[\sx_{1,s},\sx_{2,0}],\sx_{2,0}],\\
  & \tilde{\sX}_{[1,2,2,2],s}\coloneqq [[[\sx_{1,s},\sx_{2,0}],\sx_{2,0}],\sx_{2,0}],\\
  & \tilde{\sX}_{[1,2,1,2,2],s}\coloneqq [[\sx_{1,s},\sx_{2,0}],[[\sx_{1,0},\sx_{2,0}],\sx_{2,0}]].
\end{align}

\end{itemize}

Let $\sH$ denote the set of all functions $h\colon \Delta^{+}\times\BN\rightarrow \BN$ with finite support.
For any $h\in\sH$, we consider the ordered monomials
\begin{equation}\label{eq:pbwd-yangian}
  \sX_{h}\, =\prod_{(\beta,s)\in\Delta^{+}\times\mathbb{N}}\limits^{\rightarrow}\sX_{\beta,s}^{h(\beta,s)}
  \qquad \text{and} \qquad
  \tilde{\sX}_{h}\, =\prod_{(\beta,s)\in\Delta^{+}\times\mathbb{N}}\limits^{\rightarrow}\tilde{\sX}_{\beta,s}^{h(\beta,s)}.
\end{equation}

\noindent
Then, similarly to \cite{Lev93} (cf. \cite[Theorem B.3]{FT19}), we have:

\begin{Thm}\label{yangianbasis}
The elements $\{\sX_{h}\}_{h\in \sH}$ form a basis of the free $\BQ[\hbar]$-module $\Yangian$.
\end{Thm}

\begin{proof}
Comparing $\tilde{\sX}_{\beta,s}$ to the root vectors $e_{\beta}^{(s)}$ used in \cite[(A.11)]{FT19}, we see that
the only difference is in the root vectors $\tilde{\sX}_{[1,2,1,2,2],s}$ in $G_{2}$-type. However, the two key properties
(B.1) and (B.2) of \cite[Appendix B]{FT19} still hold for the new root vectors, cf.~Remark~\ref{rk:FT-generaling-1}.
Hence, the proof of \cite[Theorem B.2]{FT19} and thus of \cite[Theorem B.3]{FT19} still goes through. This proves that
$\{\tilde{\sX}_{h}\}_{h\in \sH}$ form a basis of the free $\BQ[\hbar]$-module $\Yangian$. The proof of the fact that
the ordered monomials $\{\sX_{h}\}_{h\in \sH}$ in more general root vectors of \eqref{ygrv1}--\eqref{ygrv2} also provide
a basis will be derived from the shuffle algebra realization of $\Yangian$, see Theorem \ref{yangshuffle}.
\end{proof}

We define the shuffle algebra $(\bar\BW,\star)$ analogously to the shuffle algebra $(S,\star)$ of Section \ref{pre}
with the following modifications:
\begin{enumerate}

\item
All rational functions $F\in \bar\BW$ are defined over $\BQ[\hbar]$.

\item
The matrix $(\hzeta_{i,j}(z))_{i,j\in I}$ is defined via
\begin{equation}\label{eq:zeta-rational}
  \hzeta_{i,j}(z)=1+\frac{(\alpha_{i},\alpha_{j})\cdot \hbar}{2z}.
\end{equation}

\item
(\emph{pole conditions}) $F\in \bar\BW_{\unl{k}}$ has the form
\begin{equation}
  F=\frac{f(\{x_{i,r}\}_{i\in I}^{1\leq r\leq k_{i}})}
         {\prod_{i<j}^{a_{ij}\neq 0}\prod_{1\leq r\leq k_{i}}^{1\leq s\leq k_{j}}(x_{i,r}-x_{j,s})},
\label{poleyg}
\end{equation}
where $f\in \BQ[\hbar][\{x_{i,r}\}_{i\in I}^{1\leq r\leq k_{i}}]^{\mathfrak{S}_{\underline{k}}}$ and $<$ is an arbitrary order on $I$.

\item
(\emph{wheel conditions}) Let $f$ be the numerator of $F\in \bar\BW_{\unl{k}}$ from \eqref{poleyg}, then
\begin{equation}
   f(\{x_{i,r}\}_{i\in I}^{1\leq r\leq k_{i}})=0\  \text{once}\
   x_{i,s_{1}}=x_{i,s_{2}}+d_{i}\hbar=\cdots=x_{i,s_{1-a_{ij}}}-d_{i}a_{ij}\hbar=x_{j,r}-\frac{d_{i}a_{ij}}{2}\hbar
\end{equation}
for any $ i\neq j$ such that $a_{ij}\neq 0$, pairwise distinct $1\leq s_{1},\dots,s_{1-a_{ij}}\leq k_{i}$, and $1\leq r\leq k_{j}$.

\item
The shuffle product is defined like \eqref{shuffleproduct}, but $\zeta_{i,j}(\frac{x_{i,r}}{x_{j,s}})$
are replaced by $\hzeta_{i,j}(x_{i,r}-x_{j,s})$.
\end{enumerate}

Similarly to Proposition \ref{morphism}, we have:

\begin{Prop}
The assignment $\sx_{i,r}\mapsto x_{i,1}^{r}\in \bar\BW_{\mathbf{1}_i}\ (i\in I, r\in\BN)$ gives rise to a $\BQ[\hbar]$-algebra homomorphism
\begin{equation}\label{eq:Psi-homom-rat}
  \Psi\colon \Yangian \longrightarrow \bar\BW.
\end{equation}
\end{Prop}

Let us adapt our key tool of {\em specialization maps} to the Yangian setup. In type $G_{2}$, the specialization maps
$\phi_{\unl{d}}$ are defined via the following specialization of the $x^{(*,*)}_{*,*}$-variables,~cf.~\eqref{speG}:
\begin{equation}
\begin{aligned}
  & x^{(\beta,s)}_{1,t}\mapsto w_{\beta,s}+t\hbar,\quad 1\leq t\leq 2,\\
  & x^{(\beta,s)}_{2,t}\mapsto w_{\beta,s}-\tfrac{3}{2}\hbar+t\hbar, \quad 1\leq t\leq 3.
\end{aligned}
\end{equation}
In type $B_{n}$, the specialization maps $\phi_{\unl{d}}$ are defined via the following specialization
of the $x^{(*,*)}_{*,*}$-variables, cf.~\eqref{spe}:
\begin{equation}
   x^{(\beta,s)}_{i,1}\mapsto w_{\beta,s}-i \hbar, \quad x^{(\beta,s)}_{i,2}\mapsto w_{\beta,s}-(2n-i-1)\hbar \qquad
   \forall\ \beta\in\Delta^{+},\ 1\leq s\leq d_\beta,\ i\in\beta.
\end{equation}

In what follows, we will use the notation $\circeq$ to denote an equality up to $\BQ^{\times}$ (cf.~\eqref{eq:trig-const}):
\begin{equation}\label{eq:rat-const}
  A\circeq B \quad  \text{if} \quad  A=c\cdot B \quad \text{for some}\  c\in \BQ^{\times}.
\end{equation}

\begin{Rk}
We emphasize that $c$ in~\eqref{eq:rat-const} is a nonzero element of $\BQ$ rather than $\BQ[\hbar]$.
Most importantly, the appearance of such constants occurs from the ``rank 1'' computations. Explicitly,
the formula~\eqref{eq:rank1-power-trig} in the shuffle algebra $S$ is now replaced by the equality
\begin{equation}\label{eq:rank1-power-rat}
  \underbrace{x_{i,1}^{r}\star\cdots\star x_{i,1}^{r}}_{\ell\rm\ times}= \ell!\cdot (x_{i,1}\cdots x_{i,\ell})^{r}
\end{equation}
in the shuffle algebra $\bar\BW$ for any $i\in I, r\geq 0, \ell\geq 1$, see~\cite[Lemma 6.22]{Tsy18}.
Thus, the product of quantum integers in the trigonometric case is now replaced by the product of integers.
\end{Rk}

We have the following straightforward analogues of Lemmas \ref{imageEtildeg} and \ref{phirvns}:

\begin{Lem}\label{rootphiG}
For type $G_{2}$, we have:
\begin{align}
  & \Psi(\tilde{\sX}_{[i],s})\circeq x_{i,1}^{s},\quad 1\leq i\leq 2,\\
  &\Psi(\tilde{\sX}_{[1,2],s})\circeq \frac{\hbar x_{1,1}^{s}}{x_{1,1}-x_{2,1}},\\
  & \Psi(\tilde{\sX}_{[1,2,2],s})\circeq \frac{\hbar^{2} x_{1,1}^{s}}{(x_{1,1}-x_{2,1})(x_{1,1}-x_{2,2})},\\
  & \Psi(\tilde{\sX}_{[1,2,2,2],s})\circeq \frac{\hbar^{3} x_{1,1}^{s}}{(x_{1,1}-x_{2,1})(x_{1,1}-x_{2,2})(x_{1,1}-x_{2,3})},\\
  & \Psi(\tilde{\sX}_{[1,2,1,2,2],s})\circeq
    \frac{\hbar^{4} g(\{x_{1,r},x_{2,t}\}_{1\leq r\leq 2}^{1\leq t\leq 3})}
         {\prod_{1\leq r\leq 2}^{1\leq t\leq 3}(x_{1,r}-x_{2,t})},
\end{align}
where $g\in \BQ[\hbar][\{x_{1,r},x_{2,t}\}_{1\leq r\leq 2}^{1\leq t\leq 3}]^{\fS_{2}\times\fS_{3}}$
(we do not really need an explicit formula for this $g$).
\end{Lem}

\begin{Lem}\label{rootphiB}
For type $B_{n}$, we have:
\begin{align}
  & \Psi(\tilde{\sX}_{[i,j],s})\circeq \frac{\hbar^{j-i}x_{i,1}^{s}}{(x_{i,1}-x_{i+1,1})\cdots (x_{j-1,1}-x_{j,1})},\\
  & \Psi(\tilde{\sX}_{[i,n,j],s})\circeq
    \frac{\hbar^{2n-i-j+1}x_{i,1}^{s}\prod_{\ell=j}^{n-1}
           (2\hbar+x_{\ell,1}-x_{\ell,2})(2\hbar-x_{\ell,1}+x_{\ell,2})}
         {(x_{i,1}-x_{i+1,1})\cdots (x_{j-1,1}-x_{j,1})(x_{j-1,1}-x_{j,2})
          \prod_{\ell=j}^{n-1}\prod_{1\leq r,t\leq 2}(x_{\ell,r}-x_{\ell+1,t})}.
\end{align}
\end{Lem}

As follows from the above two lemmas, the images $\Psi(\tilde{\sX}_{\beta,s})$ are divisible by $\hbar^{|\beta|-1}$.
This is actually true for any root vectors $\sX_{\beta,s}$ defined in~\eqref{ygrv1}--\eqref{ygrv2}:

\begin{Lem}\label{conrv}
For any $\beta\in\Delta^{+}$ and $s\in\BN$, $\Psi(\sX_{\beta,s})$ is divisible by $\hbar^{|\beta|-1}$.
\end{Lem}

\begin{proof}
The proof is similar to that of Lemma \ref{phirv}, and follows immediately from the equality
  \[\hzeta_{i,j}(z)-\hzeta_{j,i}(-z)=\frac{(\alpha_{i},\alpha_{j})}{z}\hbar .\qedhere\]
\end{proof}

Furthermore, invoking the constants $\kappa_{\beta}$ given by \eqref{kappaG} in type $G_{2}$ and by \eqref{kappaB}
in type $B_{n}$, we have the following straightforward analogue of Lemmas \ref{Gs1} and \ref{phirv}:

\begin{Lem}\label{phiyangianrs}
In both types $G_{2}$ and $B_{n}$, we have:
\begin{equation}
  \phi_{\beta}(\Psi(\sX_{\beta,s}))\circeq
  \hbar^{\kappa_{\beta}}\cdot p_{\beta,s}(w_{\beta,1}) \qquad \forall\ (\beta,s)\in\Delta^{+}\times \BN,
\end{equation}
where $p_{\beta,s}(w)\in \BQ[\hbar][w]$ is a monic degree $s$ polynomial in $w$ over $\BQ[\hbar]$.
\end{Lem}

Recall that $\sH$ denotes the set of all functions $h\colon \Delta^+\times \BN\to \BN$ with finite support.
For any $\unl{k}\in \BN^I$ and $\unl{d}\in \mathrm{KP}(\unl{k})$, we define the subsets $\sH_{\unl{k}}$,
$\sH_{\unl{k},\unl{d}}$ of $\sH$ similarly to \eqref{hunlkunld}, but with $h\in H$ been replaced by $h\in \sH$.
Using Lemma \ref{phiyangianrs} and arguing exactly as in Sections \ref{tG}--\ref{type B}, we obtain the following
analogues of Lemmas \ref{shuffleelement} and \ref{vanish} for the Yangians of types $G_{2}$ and $B_{n}$:

\begin{Lem}\label{shuffleeleyang}
For  any $h\in\sH_{\unl{k},\unl{d}}$, we have
\begin{equation}
  \phi_{\unl{d}}(\Psi(\sX_{h}))\circeq
  \hbar^{\sum_{\beta\in\Delta^{+}}d_{\beta}\kappa_{\beta}}\cdot
  \prod_{\beta,\beta'\in \Delta^+}^{\beta<\beta'}\hat{G}_{\beta,\beta'}\cdot \prod_{\beta\in\Delta^{+}}\hat{G}_{\beta}\cdot
  \prod_{\beta\in\Delta^{+}}\hat{P}_{\lambda_{h,\beta}},
\end{equation}
where $\hat{G}_{\beta,\beta'},\hat{G}_{\beta}$ are \underline{independent of $h\in \sH_{\unl{k},\unl{d}}$} and are
rational counterparts of $G_{\beta,\beta'},G_{\beta}$ from Lemma {\rm \ref{shuffleelement}}
(obtained by replacing factors $(x-v^{t}y)$ with $(x-y-\frac{t}{2}\hbar)$), while
\begin{equation}\label{hlp-rat}
  \hat{P}_{\lambda_{h,\beta}}={\mathop{Sym}}_{\mathfrak{S}_{d_{\beta}}}
  \left(\prod_{s=1}^{d_{\beta}}p_{\beta,r_{\beta}(h,s)}(w_{\beta,s})
  \prod_{1\leq s<r\leq d_{\beta}}\Big(1+\frac{(\beta,\beta)\cdot \hbar}{2(w_{\beta,s}-w_{\beta,r})}\Big)\right).
\end{equation}
\end{Lem}

This again features a ``rank $1$ reduction'': each $\hat{P}_{\lambda_{h,\beta}}$ from~\eqref{hlp-rat} can be viewed as
the shuffle product $p_{\beta,r_{\beta}(h,1)}(x)\star\cdots \star p_{\beta,r_{\beta}(h,d_\beta)}(x)$ in the $A_1$-type
shuffle algebra $\bar\BW$, evaluated at $\{w_{\beta,s}\}_{s=1}^{d_\beta}$.
The following result is a Yangian counterpart of Lemma~\ref{vanish} for types $G_2$ and $B_n$:

\begin{Lem}\label{vanishyg}
For  any $h\in\sH_{\unl{k},\unl{d}}$ and $\unl{d}'<\unl{d}$, we have $\phi_{\unl{d}'}(\Psi(\sX_{h}))=0$.
\end{Lem}

Combining Theorem \ref{yangianbasis} and Lemmas \ref{shuffleeleyang}--\ref{vanishyg}, we obtain
(cf.~\cite[Proposition 6.16]{Tsy18}):

\begin{Prop}\label{injyang}
The homomorphism $\Psi$ of~\eqref{eq:Psi-homom-rat} is injective.
\end{Prop}

Following \cite[Definition 3.27]{Tsy19}, we introduce:

\begin{Def}
$F\in \bar\BW_{\unl{k}}$ is \textbf{good} if $\phi_{\unl{d}}(F)$ is divisible by
$\hbar^{\sum_{\beta\in \Delta^{+}}d_{\beta}\kappa_{\beta}}$ for any $\unl{d}\in\text{\rm KP}(\unl{k})$.
\end{Def}

Let $\BW_{\unl{k}}$ be the $\BQ[\hbar]$-submodule of all good elements in $\bar\BW_{\unl{k}}$, and set
$\BW:=\bigoplus_{\unl{k}\in\BN^{I}}\BW_{\unl{k}}$.

\begin{Prop}\label{subsetyang}
$\Psi(\Yangian)\subset \BW$.
\end{Prop}

\begin{proof}
For any $m\in\BN$, $i_{1},\dots,i_{m}\in I$, $r_{1},\dots,r_{m}\in\BN$, set
$F\coloneqq \Psi(\sx_{i_{1},r_{1}}\cdots \sx_{i_{m},r_{m}})$, and let $f$ be the numerator of $F$ from \eqref{poleyg}.
Set $\unl{k}=\sum_{\ell=1}^{m}\alpha_{i_{\ell}}\in\BN^{I}$, and choose any $\unl{d}\in\text{KP}(\unl{k})$.
It suffices to show that the $\phi_{\unl{d}}$-specialization of each summand in the symmetrization from $f$
is divisible by $\hbar^{\sum_{\beta\in\Delta^{+}}d_{\beta}\kappa_{\beta}}$. Similarly to~\eqref{eq:o-spot},
if a variable $x^{(*,*)}_{*,*}$ is plugged into $\Psi(\sx_{i_{q},r_{q}})$ for some $1\leq q\leq m$, then
we shall use the notation $o(x^{(*,*)}_{*,*})=q$.

In type $B_{n}$, the case of $\beta=[i,j]$ is analogous to $A_{n}$-type. Thus, it remains to treat the case of
$\beta=[i,n,j]$ with $d_{\beta}\neq 0$. For any $1\leq s\leq d_{\beta}$, the $\phi_{\unl{d}}$-specialization of
the corresponding summand vanishes unless
  \[o(x^{(\beta,s)}_{i,1})> o(x^{(\beta,s)}_{i+1,1})> \cdots > o(x^{(\beta,s)}_{n,1})> o(x^{(\beta,s)}_{n,2})>
    \cdots>o(x^{(\beta,s)}_{j+1,2})> o(x^{(\beta,s)}_{j,2}).\]
In the latter case, the $\phi_{\unl{d}}$-specialization of the $\hzeta$-factors between pairs of the
$x^{(\beta,s)}_{*,*}$-variables contributes precisely the required factor $\hbar^{\kappa_{\beta}}$.

In type $G_{2}$, the only nontrivial check is for $\beta=[1,2,1,2,2]$ case. Suppose $d_{\beta}\neq 0$.
For any $1\leq s\leq d_{\beta}$, the  $\phi_{\unl{d}}$-specialization of the corresponding summand vanishes unless
  \[o(x^{(\beta,s)}_{1,1})> o(x^{(\beta,s)}_{2,1})>o(x^{(\beta,s)}_{2,2})> o(x^{(\beta,s)}_{2,3})
    \quad \text{and} \quad o(x^{(\beta,s)}_{1,2})>o(x^{(\beta,s)}_{2,2}).\]
In the latter case, the $\phi_{\unl{d}}$-specialization of the $\hzeta$-factors between pairs of the
$x^{(\beta,s)}_{*,*}$-variables contributes precisely the required factor $\hbar^6=\hbar^{\kappa_{\beta}}$ as well.
\end{proof}

Let $\BW'_{\unl{k}}$ be the $\BQ[\hbar]$-submodule of $\BW_{\unl{k}}$ spanned by
$\{\Psi(\sX_{h})\}_{h\in \sH_{\unl{k}}}$. Then, the following Yangian counterpart of Lemma \ref{span}
holds true in types $G_{2}$ and $B_{n}$ (cf. Propositions~\ref{spanG} and~\ref{spanB}):

\begin{Prop}\label{spanyang}
For any $F\in \BW_{\unl{k}}$ and $\unl{d}\in \text{\rm KP}(\unl{k})$, if $\phi_{\unl{d}'}(F)=0$
for all $\unl{d}'\in \text{\rm KP}(\unl{k})$ such that $\unl{d}'<\unl{d}$, then there exists
$F_{\unl{d}}\in \BW'_{\unl{k}}$ such that $\phi_{\unl{d}}(F)=\phi_{\unl{d}}(F_{\unl{d}})$
and $\phi_{\unl{d}'}(F_{\unl{d}})=0$ for all $\unl{d}'<\unl{d}$.
\end{Prop}

\begin{proof}
Since $F$ is good, the specialization $\phi_{\unl{d}}(F)$ is divisible by
$\hbar^{\sum_{\beta\in\Delta^{+}}d_{\beta}\kappa_{\beta}}$. On the other hand, arguing as in the proofs
of Propositions~\ref{spanG} and~\ref{spanB} (which utilized only wheel conditions), we conclude that $\phi_{\unl{d}}(F)$
is also divisible by $\prod_{\beta<\beta'}\hat{G}_{\beta,\beta'}\cdot \prod_{\beta\in\Delta^{+}}\hat{G}_{\beta}$.
Therefore, we have:
\begin{equation}\label{wheelspeyg}
  \phi_{\unl{d}}(F)=\hbar^{\sum_{\beta\in\Delta^{+}}d_{\beta}\kappa_{\beta}}\cdot
  \prod_{\beta,\beta'\in \Delta^+}^{\beta<\beta'}\hat{G}_{\beta,\beta'}\cdot \prod_{\beta\in\Delta^{+}}\hat{G}_{\beta}\cdot G
\end{equation}
for some symmetric polynomial
  $G\in \BQ[\hbar][\{w_{\beta,s}\}_{\beta\in\Delta^{+}}^{1\leq s\leq d_{\beta}}]^{\fS_{\unl{d}}}$.
Combining~\eqref{wheelspeyg} with Lemma~\ref{shuffleeleyang}
(and the ``rank $1$'' counterpart from~\cite[end of proof of Lemma 3.18]{Tsy19}),
we see that there is a linear combination $F_{\unl{d}}=\sum_{h\in \sH_{\unl{k},\unl{d}}}c_{h}\sX_{h}$ such that
$\phi_{\unl{d}}(F)=\phi_{\unl{d}}(F_{\unl{d}})$. On the other hand, the equality $\phi_{\unl{d}'}(F_{\unl{d}})=0$
for all $\unl{d}'<\unl{d}$ follows from Lemma \ref{vanishyg}.
\end{proof}

Combining Propositions \ref{subsetyang}--\ref{spanyang}, we immediately obtain the shuffle algebra realization
and the PBW theorem for $\Yangian$ in types $G_{2}$ and $B_{n}$:

\begin{Thm}\label{yangshuffle}
(a) The $\BQ[\hbar]$-algebra embedding $\Psi\colon \Yangian \to \bar\BW$ of Proposition~\ref{injyang} gives rise
to a $\BQ[\hbar]$-algebra isomorphism $\Psi\colon \Yangian \,\iso\, \BW$.

\medskip
\noindent
(b) The ordered monomials $\{\sX_{h}\}_{h\in\sH}$ of~\eqref{eq:pbwd-yangian} form a basis of the free $\BQ[\hbar]$-module $\Yangian$,
cf.~Theorem~{\rm\ref{yangianbasis}}.
\end{Thm}

\begin{proof}
According to Lemma~\ref{vanishyg} and Proposition~\ref{injyang}, $\{\Psi(\sX_{h})\}_{h\in\sH}\subset \BW$ are
linearly independent. On the other hand, by iterated application of Proposition \ref{spanyang}, we also get that
$\{\Psi(\sX_{h})\}_{h\in\sH}$ span $\BW$ over $\BQ[\hbar]$, cf.~\cite[Proposition 1.6]{Tsy22}. Thus,
$\{\Psi(\sX_{h})\}_{h\in\sH}$ form a basis of $\BW$. Combining this with the injectivity of the homomorphism
$\Psi\colon \Yangian \to \BW$ from Proposition~\ref{injyang} (which uses the validity of Theorem~\ref{yangshuffle}(b)
only for the particular choices $\tilde{\sX}_{\beta,s}$, see Theorem~\ref{yangianbasis}), we immediately obtain
both parts of Theorem~\ref{yangshuffle}.
\end{proof}

\begin{Rk}\label{rk:FT-generaling-1}
We note that Theorem~\ref{yangshuffle}(b) can be proved directly as Theorem B.3 in~\cite{FT19}. The proof of the latter
relied only on (B.1) and (B.2) of \emph{loc.cit}. The former of these holds true in our setup without any changes:
  $\Delta(X_{\beta,s})=X_{\beta,s}\otimes 1+1\otimes X_{\beta,s}+\mathrm{lower\ order\ terms}$, where $\Delta$ is
the coproduct on the Yangian and we use the standard  filtration on the Yangian. On the other hand, (B.2) fails on the nose,
but what we really need is the property stated after (B.2) which does still hold:
``for any PBW monomial $y$, the expression $\tau_a(y)$ is polynomial in the variable $a$, has a maximal degree of $a$
equal to the filtered degree of $y$, and the coefficient of this leading power of $a$ equals $\bar{y}$ which is obtained
from $y$ by replacing all $\sx_{i,r}$ with $\sx_{i,0}$''.
\end{Rk}

   %%%%%%%%%%%%%%%%%%%%%%%%%%%%%%%%%%%%%%%%%%%%%%%%%%%%%%%%%%%%%%%%%%
   %%%%%%%%%%%%%%%%%%%%%%%%%%%%%%%%%%%%%%%%%%%%%%%%%%%%%%%%%%%%%%%%%%
   %%%%%%%%%%%%%%%%%%%%%%%%%%%%%%%%%%%%%%%%%%%%%%%%%%%%%%%%%%%%%%%%%%

\subsection{The Drinfeld-Gavarini dual $\dgdual$ and its shuffle algebra realization}
\label{DGdual}

For any $(\beta,s)\in\Delta^{+}\times\BN$, define $\bar{\sX}_{\beta,s}\in Y_{\hbar}^{>}(\mathfrak{g})$ via
\begin{equation}
  \bar{\sX}_{\beta,s}\coloneqq \hbar\cdot \sX_{\beta,s}.
\end{equation}
We define $\dgdual$, the \emph{``positive subalgebra'' of the Drinfeld-Gavarini dual}, as the  $\BQ[\hbar]$-subalgebra
of $\Yangian$ generated by $\{ \bar{\sX}_{\beta,s}\}_{\beta\in\Delta^{+}}^{s\in\BN}$. For any $h\in \sH$, define
the ordered monomial (cf.~\eqref{eq:pbwd-yangian}):
\begin{equation}
  \bar{\sX}_{h}\coloneqq \prod_{(\beta,s)\in\Delta^{+}\times\mathbb{N}}\limits^{\rightarrow}\bar{\sX}_{\beta,s}^{h(\beta,s)}.
\end{equation}
Similarly to \cite[Theorem A.7]{FT19}, we obtain:

\begin{Thm}\label{thm:Gavarini}
(a) $\dgdual$ is independent of the choice of root vectors $\sX_{\beta,s}$
in \eqref{ygrv1}--\eqref{ygrv2}.

\medskip
\noindent
(b) For any choices of $s_k$ in \eqref{ygrv1}--\eqref{ygrv2}, the ordered monomials $\{\bar{\sX}_{h}\}_{h\in \sH}$
form a basis of the free $\BQ[\hbar]$-module $\dgdual$.
\end{Thm}

\begin{proof}
The arguments of \cite[Appendix A]{FT19} apply directly to the particular choice
$\{\tilde{\sX}_{\beta,s}\}_{\beta\in\Delta^{+}}^{s\in\BN}$. The general case can be derived
from the shuffle realization of $\dgdual$, see Theorem~\ref{dgdualshuffle}.
\end{proof}

\begin{Rk}
We note that Theorem~\ref{thm:Gavarini} can be proved directly as Theorem A.7 in~\cite{FT19}. The proof of
the latter relied only on the validity of properties (As1, As2, As3) from \emph{loc.cit}. In the present setup:
(As1) is obvious, (As2) is established in Theorem~\ref{yangshuffle}(b), (As3) is verified precisely as
in~\cite[Lemma A.6]{FT19} since all our $\sX_{\beta,s}$'s are still iterated commutators of~$\sx_{i,r}$'s.
\end{Rk}

Following~\cite[Definition 3.8]{Tsy19}, we introduce:

\begin{Def}
$F\in \bar\BW_{\unl{k}}$ is \textbf{integral} if $F$ is divisible by $\hbar^{|\unl{k}|}$ and $\phi_{\unl{d}}(F)$
is divisible by $\hbar^{\sum_{\beta\in \Delta^{+}}d_{\beta}(\kappa_{\beta}+1)}$ for any $\unl{d}\in \mathrm{KP}(\unl{k})$.
\end{Def}

Let $\mathbf{W}_{\unl{k}}\subset \bar\BW_{\unl{k}}$ be the $\BQ[\hbar]$-submodule of all integral elements, and set
$\mathbf{W}:=\bigoplus_{\unl{k}\in\BN^{I}}\mathbf{W}_{\unl{k}}$. Then,  following Lemmas \ref{conrv}--\ref{phiyangianrs}
and the proof of Proposition \ref{subsetyang}, we obtain:

\begin{Prop}
$\Psi(\dgdual)\subset \mathbf{W}$.
\end{Prop}

Finally, we have the following upgrade of Theorem \ref{yangshuffle}:

\begin{Thm}\label{dgdualshuffle}
The $\BQ[\hbar]$-algebra isomorphism $\Psi\colon \Yangian \,\iso\, \BW$ of Theorem {\rm \ref{yangshuffle}(a)} gives rise to
a $\BQ[\hbar]$-algebra isomorphism $\Psi\colon \dgdual \,\iso\, \mathbf{W}$.
\end{Thm}

\begin{proof}
It remains to prove the opposite inclusion $\mathbf{W}\subseteq \Psi(\dgdual)$. To this end, it suffices to show
that for any $F\in \mathbf{W}_{\unl{k}}$ and $\unl{d}\in \text{\rm KP}(\unl{k})$, if $\phi_{\unl{d}'}(F)=0$ for all
$\unl{d}'\in \text{\rm KP}(\unl{k})$ such that $\unl{d}'<\unl{d}$, then there exists
$F_{\unl{d}}\in \mathbf{W}_{\unl{k}}\cap \Psi(\dgdual)$ such that $\phi_{\unl{d}}(F)=\phi_{\unl{d}}(F_{\unl{d}})$
and $\phi_{\unl{d}'}(F_{\unl{d}})=0$ for $\unl{d}'<\unl{d}$, cf.~(2) after Lemma~\ref{span}.
The proof of this result is analogous to the proof of Proposition \ref{spanyang}, except that the factor
$\hbar^{\sum_{\beta\in\Delta^{+}}d_{\beta}\kappa_{\beta}}$ in \eqref{wheelspeyg} is getting replaced by
$\hbar^{\sum_{\beta\in \Delta^{+}}d_{\beta}(\kappa_{\beta}+1)}$.
\end{proof}

   %%%%%%%%%%%%%%%%%%%%%%%%%%%%%%%%%%%%%%%%%%%%%%%%%%%%%%%%%%%%%%%%%%
   %%%%%%%%%%%%%%%%%%%%%% Appendix %%%%%%%%%%%%%%%%%%%%%%%%%%%%%%%%%%
   %%%%%%%%%%%%%%%%%%%%%%%%%%%%%%%%%%%%%%%%%%%%%%%%%%%%%%%%%%%%%%%%%%

\appendix
\section{The RTT realization in type $B_n$}\label{sec:app}

In this section, we recall the RTT realization of $U_v(L{\mathfrak{o}}_{2n+1})$, established in~\cite{JLM20},
and use it to explain the natural origin and the name of the integral form $\integralb$ from Subsection~\ref{rttb}.

   %%%%%%%%%%%%%%%%%%%%%%%%%%%%%%%%%%%%%%%%%%%%%%%%%%%%%%%%%%%%%%%%%%
   %%%%%%%%%%%%%%%%%%%%%%%%%%%%%%%%%%%%%%%%%%%%%%%%%%%%%%%%%%%%%%%%%%
   %%%%%%%%%%%%%%%%%%%%%%%%%%%%%%%%%%%%%%%%%%%%%%%%%%%%%%%%%%%%%%%%%%

\subsection{RTT realization of $U_{\sfq}(L\sso_{2n+1})$}

Set $N=2n+1$. For $1\leq i\leq N$, we define $i'$ and $\bar{i}$ via:
\begin{equation}\label{eq:prime}
  i':=N+1-i,
\end{equation}
\begin{equation}\label{eq:bar}
  (\bar{1},\ldots,\bar{N}):=\left(n-\tfrac{1}{2},\ldots,\tfrac{1}{2}, 0, -\tfrac{1}{2}, \ldots, -n+\tfrac{1}{2}\right).
\end{equation}
To follow the notations of~\cite{JLM20}, we also define:
\begin{equation}\label{eq:q-vs-v}
  \sfq:=v^2 \ \ (\mathrm{so\ that}\ v=\sfq^{1/2}), \qquad \xi:=\sfq^{2-N}.
\end{equation}

Consider the trigonometric $R$-matrix with a spectral parameter $\bar{R}_\trig(u)$ given by
\begin{equation}\label{trigR}
  \bar{R}_\trig(u):=
  \frac{u-1}{u\sfq-\sfq^{-1}} \, R + \frac{\sfq-\sfq^{-1}}{u\sfq-\sfq^{-1}} \, P - \frac{(\sfq-\sfq^{-1})(u-1)\xi}{(u\sfq-\sfq^{-1})(u-\xi)} \, Q ,
\end{equation}
see~\cite[(3.1)]{JLM20}, where $P,Q,R\in (\End\, \BQ^N)^{\otimes 2}$ are defined via (with $\sfq=v^2$ as in~\eqref{eq:q-vs-v}):
\begin{equation}
\begin{split}
  & P\ =\sum_{1\leq i,j\leq N} e_{ij}\otimes e_{ji},\\
  & Q\ =\sum_{1\leq i,j\leq N} \sfq^{\bar{i}-\bar{j}} e_{i'j'}\otimes e_{ij},\\
  & R\ =\sum_{1\leq i\leq N}^{i\ne n+1} \sfq^{1-\delta_{i,n+1}} e_{ii}\otimes e_{ii} \, +
    \sum_{1\leq i,j\leq N}^{i\ne j,j'} e_{ii}\otimes e_{jj} + \sfq^{-1} \sum_{i\ne n+1} e_{ii}\otimes e_{i'i'} \ + \\
  & \qquad \quad (\sfq-\sfq^{-1})\sum_{i<j} e_{ij}\otimes e_{ji} - (\sfq-\sfq^{-1})\sum_{i>j} \sfq^{\bar{i}-\bar{j}} e_{i'j'}\otimes e_{ij}.
\end{split}
\end{equation}
This $\bar{R}_\trig(u)$ satisfies the famous \emph{Yang-Baxter equation} (with a spectral parameter):
\begin{equation}\label{qYB}
  \bar{R}_{\trig;12}(u/v) \bar{R}_{\trig;13}(u/w) \bar{R}_{\trig;23}(v/w)=
  \bar{R}_{\trig;23}(v/w) \bar{R}_{\trig;13}(u/w) \bar{R}_{\trig;12}(u/v).
\end{equation}

Following~\cite{JLM20} (with the conceptual ideology going back to~\cite{frt}), we define the
\textbf{RTT integral form of the quantum loop algebra of $\sso_{N}$}, denoted by $\rtU^\rtt_v(L\sso_{N})$, to be
the associative $\BZ[v,v^{-1}]$-algebra generated by $\{\ell^\pm_{ij}[\mp r]\}_{1\leq i,j\leq N}^{r\in \BN}$ with
the following defining relations:
\begin{equation}\label{affRTT}
\begin{split}
  & \ell^+_{ij}[0]=\ell^-_{ji}[0]=0\ \ \mathrm{for}\ 1\leq i<j\leq N,\\
  & \ell^\pm_{ii}[0]\ell^\mp_{ii}[0]=1\ \ \mathrm{for}\ 1\leq i\leq N,\\
  & \bar{R}_{\trig}(z/w)\Lf^\pm_1(z)\Lf^\pm_2(w)=\Lf^\pm_2(w)\Lf^\pm_1(z)\bar{R}_\trig(z/w),\\
  & \bar{R}_{\trig}(z/w)\Lf^+_1(z)\Lf^-_2(w)=\Lf^-_2(w)\Lf^+_1(z)\bar{R}_\trig(z/w),
\end{split}
\end{equation}
as well as
\begin{equation}\label{affRTT-extra}
  \Lf^\pm(u) D \Lf^\pm(u\xi)^{\mathrm{t}} D^{-1} = D \Lf^\pm(u\xi)^{\mathrm{t}} D^{-1} \Lf^\pm(u) = 1,
\end{equation}
where ${\mathrm{t}}$ denotes the matrix transposition with $E^{\mathrm{t}}_{ij}=E_{j'i'}$ and $D$ is the diagonal matrix
\begin{equation*}
  D=\mathrm{diag}\big(\sfq^{\bar{1}},\sfq^{\bar{2}},\ldots,\sfq^{\bar{N}}\big).
\end{equation*}
Here, $\Lf^\pm(u)\in \rtU^\rtt_v(L\sso_N)[[u,u^{-1}]]\otimes \End\, \BQ^N$ is defined by
\begin{equation}\label{eq:L-matrix}
  \Lf^\pm(u)\ =\sum_{1\leq i,j\leq N} \ell^\pm_{ij}(u)\otimes E_{ij} \quad \mathrm{with} \quad
  \ell^\pm_{ij}(u):=\sum_{r\geq 0} \ell^\pm_{ij}[\mp r] u^{\pm r}.
\end{equation}
We also define the $\BQ(v)$-counterpart $U^\rtt_v(L\sso_N):=\rtU^\rtt_v(L\sso_N)\otimes_{\BZ[v,v^{-1}]} \BQ(v)$.

\begin{Rk}
The last two relations in~\eqref{affRTT} are commonly referred to as the \emph{RTT relations}. However, without imposing
\eqref{affRTT-extra}, one actually gets an extended version of that algebra featuring an extra Heisenberg algebra factor.
\end{Rk}

Let $U_v(L\sso_N)$ be the quantum loop algebra of type $B_n$ in the new Drinfeld realization. It is a $\BQ(v)$-algebra
generated by $\{x^\pm_{i,r},\varphi_{i,-k},\psi_{i,k},k_i^{\pm 1}\}_{1\leq i\leq n}^{r\in \BZ,k\in \BN}$ with
the relations as in~\cite[\S1]{JLM20}. Identifying $x^+_{i,r}$ with our $e_{i,r}$, the subalgebra of $U_v(L\sso_N)$
generated by $\{x^+_{i,r}\}_{1\leq i\leq n}^{r\in \BZ}$ recovers $U^>_v(L\sso_N)$ from Subsection~\ref{ssec:qlas}.
In what follows, we will consider the following generating series:
\begin{equation*}
  x^\pm_i(u)=\sum_{r\in \BZ} x^+_{i,r}u^{-r},\quad
  \varphi_i(u)=\sum_{k\geq 0} \varphi_{i,-k}u^k,\quad
  \psi_i(u)=\sum_{k\geq 0} \psi_{i,k}u^{-k}.
\end{equation*}

The relation between the algebras $U_v(L\sso_N)$ and $\rtU^\rtt_v(L\sso_{N})$ was recently established in~\cite{JLM20}.
To state the main result, we consider the Gauss decomposition of the matrices $\Lf^\pm(u)$ from~\eqref{eq:L-matrix}:
\begin{equation*}
  \Lf^\pm(u)=F^\pm(u)\cdot H^\pm(u)\cdot E^\pm(u).
\end{equation*}
Here, $F^\pm(u), H^\pm(u), E^\pm(u)\in \rtU^\rtt_v(L\sso_N)[[u,u^{-1}]]\otimes \End\, \BQ^N$ are of the form
\begin{equation*}
  F^\pm(u)=\sum_{i} E_{ii}+\sum_{i>j} f^\pm_{ij}(u)\otimes E_{ij},\
  H^\pm(u)=\sum_{i} h^\pm_i(u)\otimes E_{ii},\
  E^\pm(u)=\sum_{i} E_{ii}+\sum_{i<j} e^\pm_{ij}(u)\otimes E_{ij}.
\end{equation*}

\begin{Thm}[\cite{JLM20}]\label{thm:JLM-iso}
There is a unique $\BQ(v)$-algebra isomorphism
\begin{equation*}
  \varrho \colon U_v(L\sso_N) \, \iso \, \rtU^\rtt_v(L\sso_{N})
\end{equation*}
defined by
\begin{equation}\label{JLM formulas}
\begin{split}
  & x^+_i(u)\mapsto \frac{e^+_{i,i+1}(u\sfq^i)-e^-_{i,i+1}(u\sfq^i)}{\sfq-\sfq^{-1}},\quad
    x^-_i(u) \mapsto \frac{f^+_{i+1,i}(u\sfq^i)-f^-_{i+1,i}(u\sfq^i)}{\sfq^{1-\delta_{in}/2}-\sfq^{-1+\delta_{in}/2}}, \\
  & \psi_i(u)\mapsto h^-_{i+1}(u\sfq^i)h^-_{i}(u\sfq^i)^{-1},\quad
    \varphi_i(u)\mapsto h^+_{i+1}(u\sfq^i)h^+_{i}(u\sfq^i)^{-1},
\end{split}
\end{equation}
where $\sfq=v^2$ as in~\eqref{eq:q-vs-v}.
\end{Thm}

\begin{Rk}
We note right away that~\cite{JLM20} established a similar isomorphism for centrally extended algebras,
and also that they used slightly rescaled formulas for the images of $x^\pm_n(u)$.
\end{Rk}

   %%%%%%%%%%%%%%%%%%%%%%%%%%%%%%%%%%%%%%%%%%%%%%%%%%%%%%%%%%%%%%%%%%
   %%%%%%%%%%%%%%%%%%%%%%%%%%%%%%%%%%%%%%%%%%%%%%%%%%%%%%%%%%%%%%%%%%
   %%%%%%%%%%%%%%%%%%%%%%%%%%%%%%%%%%%%%%%%%%%%%%%%%%%%%%%%%%%%%%%%%%

\subsection{The RTT realization of $\integralb$}

Let $\rtU^{\rtt,>}_v(L\sso_{N})$ be the $\BZ[v,v^{-1}]$-subalgebra of $\rtU^\rtt_v(L\sso_{N})$ generated by the coefficients
of $\{e^\pm_{ij}(u)\}_{1\leq i<j\leq N}$, the matrix coefficients of $E^\pm(u)$. The key goal of this appendix is to highlight
the natural origin of the integral form $\integralb$ introduced in Subsection~\ref{rttb} and its specific quantum root
vectors (a special case of~\eqref{eq:rtt-vectors})
\begin{equation}\label{eq:rtt-root-vectors}
\begin{split}
  & \tilde{\mathcal{E}}^{\rtt}_{[i,j],s}:= \langle 2\rangle_{v}\cdot
    [\cdots[[e_{i,s},e_{i+1,0}]_{v^{2}},e_{i+2,0}]_{v^{2}},\cdots,e_{j,0}]_{v^{2}},\\
  & \tilde{\mathcal{E}}^{\rtt}_{[i,n,j],s}:= \langle 2\rangle_{v}\cdot
    [\cdots[[[\cdots[e_{i,s},e_{i+1,0}]_{v^{2}},\cdots, e_{n,0}]_{v^{2}}, e_{n,0}],e_{n-1,0}]_{v^{2}},\cdots,e_{j,0}]_{v^{2}}.
\end{split}
\end{equation}

Let us express the matrix coefficients of $E^\pm(u)$ as series in $u^{\pm 1}$ with coefficients in $\rtU^\rtt_v(L\sso_{N})$:
\begin{equation}\label{eq:rtt-matrix-coef}
  e^+_{ij}(u)=\sum_{r>0} e^{(-r)}_{ij} u^r,\quad  e^-_{ij}(u)=\sum_{r\geq 0} e^{(r)}_{ij} u^{-r}
  \qquad \forall\ 1\leq i<j\leq N.
\end{equation}
We also define $e_{ij}(u):=e^+_{ij}(u)-e^-_{ij}(u)$. The key technical result of this subsection is:

\begin{Prop}\label{prop:E-explicitly}
For any $1\leq i<j\leq n$, we have:
\begin{equation}\label{eq:iterated-1}
  e_{i,j+1}(u)=(1-\sfq^2)^{i-j}\cdot
  [\cdots[[e_{i,i+1}(u),e^{(0)}_{i+1,i+2}]_{\sfq},e^{(0)}_{i+2,i+3}]_{\sfq},\cdots,e^{(0)}_{j,j+1}]_{\sfq}
\end{equation}
and
\begin{multline}\label{eq:iterated-2}
  e_{i,j'}(u)=\sfq(1-\sfq^2)^{i+j-2n-1}(-1)^{j-n-1}\times \\
  [\cdots[[[\cdots[e_{i,i+1}(u),e^{(0)}_{i+1,i+2}]_{\sfq},\cdots, e^{(0)}_{n,n+1}]_{\sfq},e^{(0)}_{n,n+1}],
   e^{(0)}_{n-1,n}]_{\sfq},\cdots,e^{(0)}_{j,j+1}]_{\sfq}.
\end{multline}
\end{Prop}

\begin{proof}
Due to the ``rank reduction'' embedding homomorphisms of\cite[\S3.2, Proposition 4.2]{JLM20}, it suffices to
prove both formulas~\eqref{eq:iterated-1} and~\eqref{eq:iterated-2} for $i=1$ and $1<j\leq n$.

We prove \eqref{eq:iterated-1} for $i=1$ by induction on $j\geq 1$, the base case $j=1$ being vacuous.
Comparing the matrix coefficients $\langle v_1\otimes v_j|\cdots|v_j\otimes v_{j+1}\rangle$ of both sides of the
RTT relation $\bar{R}_{\trig}(z/w)\Lf^-_1(z)\Lf^-_2(w)=\Lf^-_2(w)\Lf^-_1(z)\bar{R}_\trig(z/w)$, we get:
\begin{equation}\label{eq:rtt-1}
\begin{split}
   & \frac{z-w}{\sfq z-\sfq^{-1}w} \ell^-_{1j}(z)\ell^-_{j,j+1}(w)+\frac{(\sfq-\sfq^{-1})z}{\sfq z-\sfq^{-1}w}\ell^-_{jj}(z)\ell^-_{1,j+1}(w)=\\
   & \frac{z-w}{\sfq z-\sfq^{-1}w} \ell^-_{j,j+1}(w)\ell^-_{1j}(z)+\frac{(\sfq-\sfq^{-1})w}{\sfq z-\sfq^{-1}w}\ell^-_{jj}(w)\ell^-_{1,j+1}(z).
\end{split}
\end{equation}
Expanding all rational factors as series in $z/w$ and evaluating the $[w^0]$-coefficients, we obtain:
\begin{equation}\label{eq:rtt-2}
  \sfq \ell^-_{1j}(z)\ell^-_{j,j+1}[0]=\sfq\ell^-_{j,j+1}[0]\ell^-_{1j}(z)+(1-\sfq^2)\ell^-_{jj}[0]\ell^-_{1,j+1}(z).
\end{equation}
Likewise, comparing the matrix coefficients $\langle v_1\otimes v_j|\cdots|v_j\otimes v_{j}\rangle$ of both sides
of the same RTT relation, we also get:
\begin{equation}\label{eq:rtt-3}
   \frac{z-w}{\sfq z-\sfq^{-1}w} \ell^-_{1j}(z)\ell^-_{jj}(w)+\frac{(\sfq-\sfq^{-1})z}{\sfq z-\sfq^{-1}w}\ell^-_{jj}(z)\ell^-_{1j}(w)=
   \ell^-_{jj}(w)\ell^-_{1j}(z).
\end{equation}
Expanding both rational factors as series in $z/w$ and evaluating the $[w^0]$-coefficients, we obtain:
\begin{equation}\label{eq:rtt-4'}
  \sfq \ell^-_{1j}(z)\ell^-_{jj}[0]=\ell^-_{jj}[0]\ell^-_{1j}(z),
\end{equation}
or equivalently:
\begin{equation}\label{eq:rtt-4}
   {\ell^-_{jj}[0]}^{-1}\ell^-_{1j}(z)=\sfq^{-1} \ell^-_{1j}(z){\ell^-_{jj}[0]}^{-1}.
\end{equation}

Let $h^-_j[0]$ be the $u^0$-coefficient of $h^-_j(u)$. Then, $\ell^-_{jj}[0]=h^-_j[0]$ and
$\ell^-_{j,j+1}[0]=h_j^-[0]e^{(0)}_{j,j+1}=\ell_{jj}^-[0]e^{(0)}_{j,j+1}$. Thus, multiplying
\eqref{eq:rtt-2} by ${\ell^-_{jj}[0]}^{-1}$ on the left and applying \eqref{eq:rtt-4}, we obtain:
\begin{equation}\label{eq:rtt-5}
  (1-\sfq^2)\ell^-_{1,j+1}(z)=[\ell^-_{1j}(z),e^{(0)}_{j,j+1}]_\sfq.
\end{equation}
Here, $\ell^-_{1,j+1}(z)=h^-_1(z)e^-_{1,j+1}(z)$ and $\ell^-_{1j}(z)=h^-_1(z)e^-_{1j}(z)$.
As $h^-_1(z)$ commutes with $e^-_{j,j+1}(w)$ for $1<j<2'$, we derive:
\begin{equation}\label{eq:rtt-6}
  e^-_{1,j+1}(z)=(1-\sfq^2)^{-1} \cdot [e^-_{1j}(z),e^{(0)}_{j,j+1}]_\sfq.
\end{equation}
Arguing in the same way, but using the other RTT relation
$\bar{R}_{\trig}(z/w)\Lf^+_1(z)\Lf^-_2(w)=\Lf^-_2(w)\Lf^+_1(z)\bar{R}_\trig(z/w)$, we also obtain:
\begin{equation}\label{eq:rtt-6'}
  e^+_{1,j+1}(z)=(1-\sfq^2)^{-1} \cdot [e^+_{1j}(z),e^{(0)}_{j,j+1}]_\sfq.
\end{equation}
Subtracting \eqref{eq:rtt-6} from \eqref{eq:rtt-6'}, we finally get:
\begin{equation}\label{eq:rtt-6-fin}
  e_{1,j+1}(z)=(1-\sfq^2)^{-1} \cdot [e_{1j}(z),e^{(0)}_{j,j+1}]_\sfq.
\end{equation}
Applying the induction assumption for $e_{1j}(z)$ completes our proof of \eqref{eq:iterated-1} for $i=1$.

To prove \eqref{eq:iterated-2} for $i=1$, we note first that
$e^\pm_{1,k+1}(z)=\frac{[e^\pm_{1k}(z),e^{(0)}_{k,k+1}]_\sfq}{1-\sfq^2}$ for $n+1<k<2n$, similarly to
\eqref{eq:rtt-6}--\eqref{eq:rtt-6'}. But according to \cite[Proposition 5.4]{JLM20}
(corrected by replacing $-e^\pm_{n-1}(u\xi \sfq^{2n-2})$ with $-e^\pm_{n}(u\xi \sfq^{2n})$ in \emph{loc.cit.}),
we have $e^{(0)}_{k,k+1}=-e^{(0)}_{k'-1,k'}$. Therefore, we get:
\begin{equation}\label{eq:rtt-7}
  e^\pm_{1,k+1}(z)=-(1-\sfq^2)^{-1} \cdot [e^\pm_{1k}(z),e^{(0)}_{k'-1,k'}]_\sfq.
\end{equation}
Finally, a special care should be taken of the $k=n+1$ case as in that case the matrix coefficient
$\langle v_1\otimes v_{n+1}|\Lf^-_2(w)\Lf^\pm_1(z)\bar{R}_\trig(z/w)|v_{n+1}\otimes v_{n+1}\rangle$
is given by a different formula:
\begin{multline}\label{rtt-8}
  \big\langle v_1\otimes v_{n+1} \,\big|\, \Lf^-_2(w)\Lf^\pm_1(z)\bar{R}_\trig(z/w) \,\big|\, v_{n+1}\otimes v_{n+1}\big\rangle=\\
  \sum_{i=1}^{N} \frac{a_{i,n+1}(z/w)}{(z/w-\sfq^{-2})(z/w-\xi)}\ell^-_{n+1,i}(w)\ell^\pm_{1,i'}(z),
\end{multline}
where
\begin{equation}\label{rtt-9}
  a_{i,n+1}(z/w)=
  \begin{cases}
     \sfq^{-1}(z/w-\xi)(z/w-1)+(\xi-1)(\sfq^{-2}-1)z/w &\ \text{if}\ i=n+1\\
     (\sfq^{-2}-1) \sfq^{\bar{i}-\overline{n+1}}\xi \cdot (z/w-1) &\ \text{if}\ i<n+1\\
     (\sfq^{-2}-1) \sfq^{\bar{i}-\overline{n+1}}\cdot (z/w-1)z/w &\ \text{if}\ i>n+1
  \end{cases}.
\end{equation}
Expanding all rational factors in $z/w$ and evaluating the $w^0$-coefficient, only the $i=n+1$ term will have
a nontrivial contribution. Explicitly, we obtain the following analogue of \eqref{eq:rtt-4'}:
\begin{equation}\label{eq:rtt-10}
  \sfq \ell^\pm_{1,n+1}(z)\ell^-_{n+1,n+1}[0]=\sfq \ell^-_{n+1,n+1}[0]\ell^\pm_{1,n+1}(z),
\end{equation}
so that \eqref{eq:rtt-4} will get replaced by
\begin{equation}\label{eq:rtt-11}
   {\ell^-_{n+1,n+1}[0]}^{-1}\ell^\pm_{1,n+1}(z)=\ell^\pm_{1,n+1}(z){\ell^-_{n+1,n+1}[0]}^{-1}.
\end{equation}
This establishes \eqref{eq:iterated-2} for $i=1$ and $j=n$, while \eqref{eq:rtt-7} establishes it then for $1<j<n$.
\end{proof}

Combining Proposition~\ref{prop:E-explicitly} with $\varrho(x^+_i(u))=\frac{-\sfq}{1-\sfq^2}e_{i,i+1}(u\sfq^i)$
of~\eqref{JLM formulas}, the definition~\eqref{eq:rtt-root-vectors}, and identification of the present currents
$x^+_i(u)$ with the currents $e_i(u)$ of~\eqref{eq:basic-def}, we obtain:

\begin{Cor}\label{cor:Dr-via-rtt-roots}
For any $1\leq i < j\leq n$ and $s\in \BZ$, we have:
\begin{equation}\label{eq:Dr-via-rtt-roots}
  \varrho(\tilde{\mathcal{E}}^{\rtt}_{[i,j],s}) \doteq e^{(s)}_{i,j+1}
  \qquad \mathrm{and} \qquad
  \varrho(\tilde{\mathcal{E}}^{\rtt}_{[i,n,j],s}) \doteq e^{(s)}_{i,j'}.
\end{equation}
\end{Cor}

Since the elements~\eqref{eq:rtt-root-vectors} are specific case of quantum root vectors~\eqref{eq:rtt-vectors},
we finally obtain:

\begin{Prop}\label{prop:RTT-vs-Drinfeld integral forms}
  $\varrho(\integralb)=\rtU^{\rtt,>}_v(L\sso_{2n+1})$.
\end{Prop}

This result explains why we called $\integralb$ the RTT integral form of $\qfb$. Moreover,
Theorem~\ref{PBWDintegralb}(b) implies the PBWD theorem for $\rtU^{\rtt,>}_v(L\sso_{2n+1})$, cf.~\cite[Theorem~3.25]{FT19}:

\begin{Cor}
The ordered monomials in $\big\{e^{(r)}_{ij} \,|\ i<j \ \mathrm{such\ that}\ i+j\leq N, r\in \BZ\big\}$ form a
basis of the free $\BZ[v,v^{-1}]$-module $\rtU^{\rtt,>}_v(L\sso_{2n+1})$.
\end{Cor}

   %%%%%%%%%%%%%%%%%%%%%%%%%%%%%%%%%%%%%%%%%%%%%%%%%%%%%%%%%%%%%%%%%%
   %%%%%%%%%%%%%%%%%%%%%%%%%%%%%%%%%%%%%%%%%%%%%%%%%%%%%%%%%%%%%%%%%%
   %%%%%%%%%%%%%%%%%%%%%%%%%%%%%%%%%%%%%%%%%%%%%%%%%%%%%%%%%%%%%%%%%%

\end{document}